\documentclass[11pt, letterpaper]{article}

\usepackage[utf8]{inputenc}
\usepackage[T1]{fontenc}
\usepackage{lmodern}

\usepackage{geometry}
\usepackage{amsmath,amssymb,amsfonts}
\usepackage{amsthm}
\usepackage{mathrsfs}
\usepackage{mathtools}
\usepackage{IEEEtrantools}
\let\sig\relax
\DeclareMathOperator{\sig}{S}

\usepackage{graphicx}
\usepackage{subcaption}
\usepackage{booktabs}
\usepackage{multirow}
\usepackage{tabularx,array}
\newcolumntype{C}{>{\centering\arraybackslash}X}

\usepackage{makecell}
\usepackage{caption}
\usepackage{adjustbox}
\usepackage{chngcntr}

\usepackage{algorithm}
\usepackage{algorithmicx}
\usepackage{algpseudocode}
\usepackage{listings}

\usepackage[title]{appendix}

\usepackage{xcolor}
\usepackage{textcomp}
\usepackage{manyfoot}
\usepackage{xspace}
\usepackage{changepage}
\usepackage{xparse}

\usepackage[sorting=nyt]{biblatex}
\usepackage[colorlinks=true,allcolors=blue]{hyperref}
\usepackage{cleveref}

\hypersetup{
  pdftitle={Scalable high-degree cubature formulae for simulating SDEs},
  pdfauthor={Peter Koepernik, Thomas Coxon and James Foster}
}

\AtBeginEnvironment{appendices}{\crefalias{section}{appendix}}
\AtBeginEnvironment{appendices}{\crefalias{subsection}{appendix}}
\AtBeginEnvironment{appendices}{\crefalias{subsubsection}{appendix}}

\newtheorem{theorem}{Theorem}

\newtheorem{construction}{Construction}
\newtheorem{lemma}{Lemma}
\newtheorem{corollary}{Corollary}
\theoremstyle{definition}
\newtheorem{definition}{Definition}

\theoremstyle{remark}
\newtheorem{remark}{Remark}

\DeclareMathOperator{\E}{\mathbb{E}}
\DeclareMathOperator{\OA}{OA}
\DeclareMathOperator{\V}{\mathbb{V}}
\DeclareMathOperator{\var}{\mathrm{Var}}
\DeclareMathOperator{\mve}{MVE}

\newcommand{\iu}{\mathrm{i}}
\newcommand{\R}{\mathbb{R}}
\renewcommand{\P}{\mathbb{P}}
\newcommand{\m}{\hspace{0.25mm}}

\newcommand{\mmm}{\hspace{10mm}}

\newcommand{\N}{\mathbb{N}}
\newcommand{\NN}{\mathcal{N}}

\newenvironment{manualtheoremnumber}[1]{%
  \renewcommand\thetheorem{#1}%
}{}

\renewcommand{\P}{\mathbb{P}}
\newcommand{\diff}{\mathop{}\!\mathrm{d}}
\newcommand{\ind}{\boldsymbol{1}}
\newcommand{\e}{\mathrm{e}}
\DeclareMathOperator{\clamp}{clamp}

\NewDocumentCommand{\DeclareGoal}{mm}{%
  \expandafter\gdef\csname goal@#1\endcsname{#2}%
  \par\smallskip
  \noindent\hypertarget{goal:#1}{}%
  \vspace*{-1.0em}
  \begin{adjustwidth}{2.5em}{0pt}%
    \noindent\hbox to \linewidth{%
      \parbox[c]{\dimexpr\linewidth-4em\relax}{\itshape #2}%
      \hfill{\parbox[c]{4em}{\raggedleft\normalfont(\textup{#1})}}%
    }%
  \end{adjustwidth}%
  \vspace*{0.5em}
  \par\smallskip
}
\NewDocumentCommand{\Goal}{m}{\hyperlink{goal:#1}{(\textup{#1})}}
\NewDocumentCommand{\GoalText}{m}{\csname goal@#1\endcsname}

\newcommand{\short}{ARCANE\xspace}
\newcommand{\longg}{\textbf{A}lgorithm for \textbf{R}ecombination of \textbf{C}ubatures from Orthogonal \textbf{A}rrays that match \textbf{N}ested \textbf{E}xpected signatures}

\begin{document}

\title{\short: Scalable high-degree cubature formulae for simulating SDEs without Monte Carlo error}

\author{
  Peter Koepernik\textsuperscript{*}\thanks{Department of Statistics, University of Oxford, UK. Email: \texttt{peter.koepernik@stats.ox.ac.uk}.} \and
  Thomas Coxon\textsuperscript{*}\thanks{Dept.~Aeronautical and Automotive Engineering, Loughborough University, UK. Email: \texttt{T.Coxon2@lboro.ac.uk}.} \and
  James Foster\thanks{Department of Mathematical Sciences, University of Bath, UK. Email: \texttt{jmf68@bath.ac.uk}.}
}

\date{\today}

\makeatletter
\let\old@fnsymbol\@fnsymbol
\renewcommand{\@fnsymbol}[1]{\@arabic{#1}}
\makeatother

\maketitle

\makeatletter
\let\@fnsymbol\old@fnsymbol
\makeatother

\begingroup
\renewcommand{\thefootnote}{\fnsymbol{footnote}}
\footnotetext[1]{Equal contribution.} 
\endgroup

\begin{abstract}
Monte Carlo sampling is the standard approach for estimating properties of solutions to stochastic differential equations (SDEs), but accurate estimates require huge sample sizes. Lyons and Victoir (2004) proposed replacing independently sampled Brownian driving paths with ``cubature formulae'', deterministic weighted sets of paths that match Brownian ``signature moments'' up to some degree $D$. They prove that cubature formulae exist for arbitrary $D$, but explicit constructions are difficult and have only reached $D=7$, too small for practical use. We present ARCANE, an algorithm that efficiently and automatically constructs cubature formulae of arbitrary degree. It reproduces the state of the art in seconds and reaches $\boldsymbol{D=19}$ within hours on modest hardware. In simulations across multiple different SDEs and error metrics, our cubature formulae robustly achieve an error orders of magnitude smaller than Monte Carlo with the same number of paths.
\end{abstract}


\section{Introduction}

Stochastic differential equations (SDEs) are commonly used to model continuous-time phenomena evolving under the influence of random noise. For example, SDEs with \emph{scalar noise} (that is, driven by a single Brownian motion) have been applied across a range of topics---such as
solar forecasting \cite{badosa2018solar, iversen2014solar, chaabane2025solar}, wind speed modelling \cite{arenaslopez2020wind}, physics \cite{rehman2023optics}, engineering \cite{wang2022sdes}, epidemiology \cite{cai2015sirs, maki2013sir}, population genetics \cite{czuppon2021genetics}, mathematical finance \cite{cox1985cir}, systems biology \cite{vaisband2025nsdes}, economics \cite{chen2025wealth} and social sciences \cite[chapter 2]{cobb1981socialsdes}. In addition, SDEs with \emph{low-dimensional noise} (e.g.~driven by $d\leq 3$ independent Brownian motions) are natural models for physical objects whose motion exhibits random fluctuations. For example, they have been used to model aircraft dynamics \cite{li2023aircraft, liu2021aircraft}.

We consider low-dimensional SDEs, defined by Stratonovich integration, of the form:
\begin{equation}
    \diff y_{t} = \mu(y_{t}) \diff t + \sum_{i=1}^d \sigma_i(y_{t}) \circ \diff W_t^i,
    \label{eq:sde}
\end{equation}
where the initial condition $y_{0} = \xi$ and solution $y = \{y_t\}_{t \in [0, T]}$ take values in $\R^e$, $W = \{W_{t}\}_{t\in [0,T]}$ is a standard $d$-dimensional Brownian motion, and $\mu,\sigma_i : \R^e \to \R^e$ are the vector fields (referred to as the drift and diffusion respectively). Note that ``low-dimensional'' refers to the number $d$ of independent driving Brownian motions; the dimension $e$ of the solution path is unrestricted.

In practice, one is often interested in estimating key statistics of the solution to the SDE~\eqref{eq:sde}, such as its mean or variance, or more generally the expectation of $f(y_T)$ for some functional $f\colon \R^e \to \R$. This is traditionally achieved using Monte Carlo estimation (or Quasi-Monte Carlo, see Section~\ref{sect:results}), which takes the average over some number $M$ of independently sampled numerical solutions of~\eqref{eq:sde}, see Table~\ref{tab:mc-vs-cubature}~(left) and Figure \ref{fig:monte_carlo_estimation} for an illustration.

However, by the Central Limit Theorem, the error of the Monte Carlo estimate is proportional to~$1 / \sqrt{M}$, where~$M$ is the number of samples. An accurate estimation of $\E[f(y_T)]$ can therefore require an enormous number of sample paths and be computationally very expensive---especially for complex SDEs.

\begin{figure}
    \centering
    \includegraphics[width=0.95\linewidth]{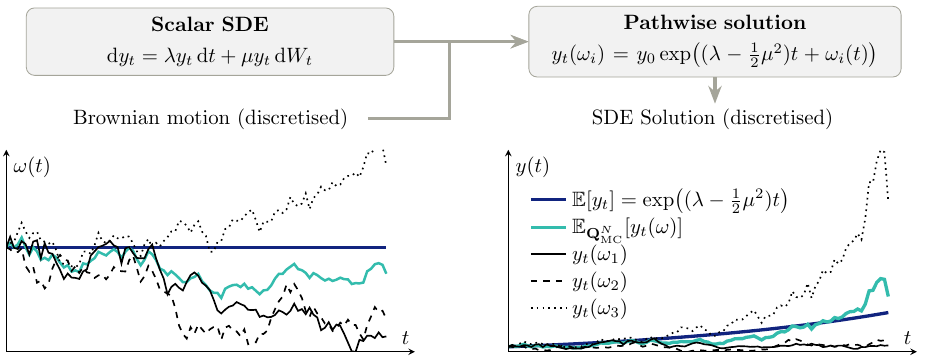}
    \caption{Monte Carlo estimation for the mean of a scalar SDE, in the ideal case where there exists a closed-form solution that can be evaluated on a discretised time domain.}
    \label{fig:monte_carlo_estimation}
\end{figure}

\begin{table}
  \centering
  \small
  \setlength{\tabcolsep}{6pt}
  \renewcommand{\arraystretch}{1.25}
  \begin{tabularx}{\linewidth}{@{} l C C @{}}
    \toprule
     & \textbf{Monte Carlo} & \textbf{\short} \\
    \midrule
    \textbf{Estimate}
      & \multicolumn{2}{l}{%
        \refstepcounter{equation}\label{eq:estimator}%
        \hfill%
        $\displaystyle
          \E\big[f(y_T)\big]
          \approx
          \sum_{i=1}^{M}\lambda_i f\big(y_T(\omega_i)\big)
        $%
        \hfill(\theequation)%
      } \\
    \midrule
    \addlinespace[0.8ex]
    \textbf{Paths} $(\omega_i)$
      & randomly sampled Brownian paths
      & deterministic Brownian‑like paths \\
    \addlinespace[0.8ex]
    \textbf{Weights} $(\lambda_i)$
      & uniform, $\lambda_i=\tfrac{1}{M}$
      & not necessarily uniform\\
    \bottomrule
  \end{tabularx}
  \caption{Comparison of the Monte Carlo and \short estimators. Here, $y_T(\omega) \in \R^e$ denotes the solution at time $T$ of the system \eqref{eq:sde} driven by a path $\omega\colon [0,T]\to \R^d$.}
  \label{tab:mc-vs-cubature}
\end{table}

\begin{figure}
    \centering
    \includegraphics[width=\linewidth]{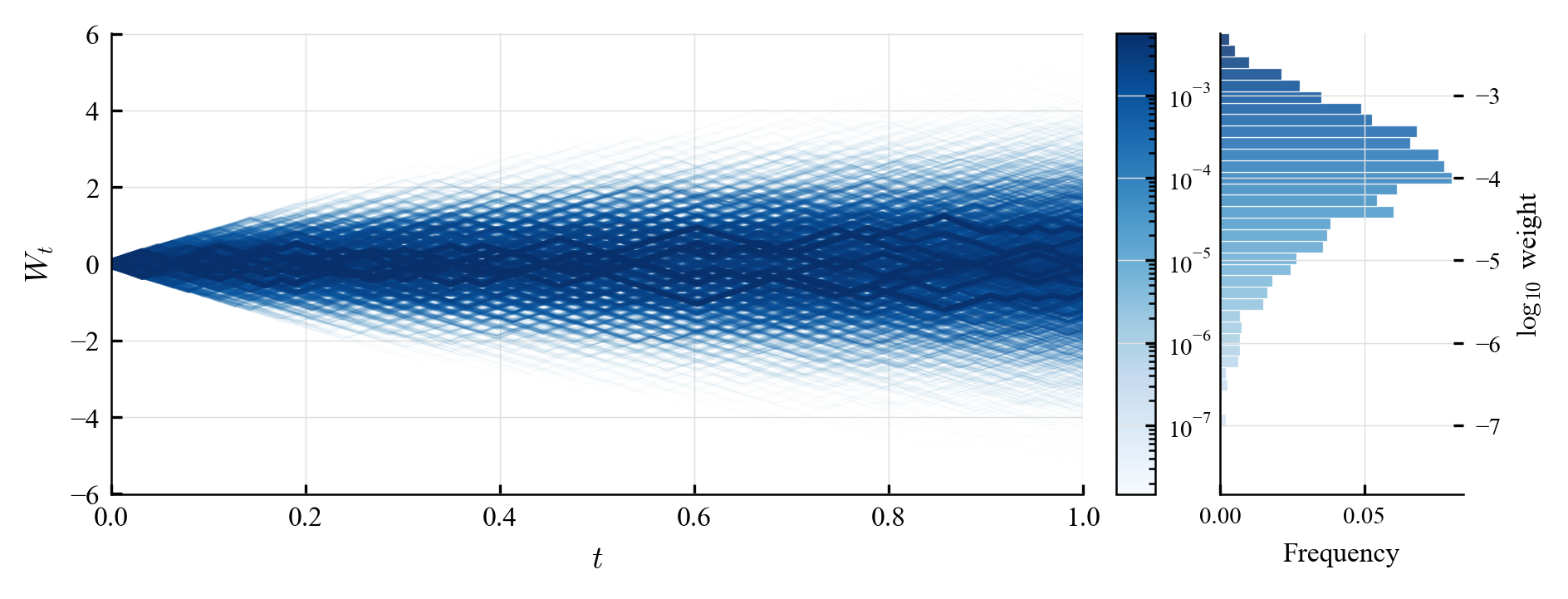}
    \bigbreak
    \includegraphics[width=\linewidth]{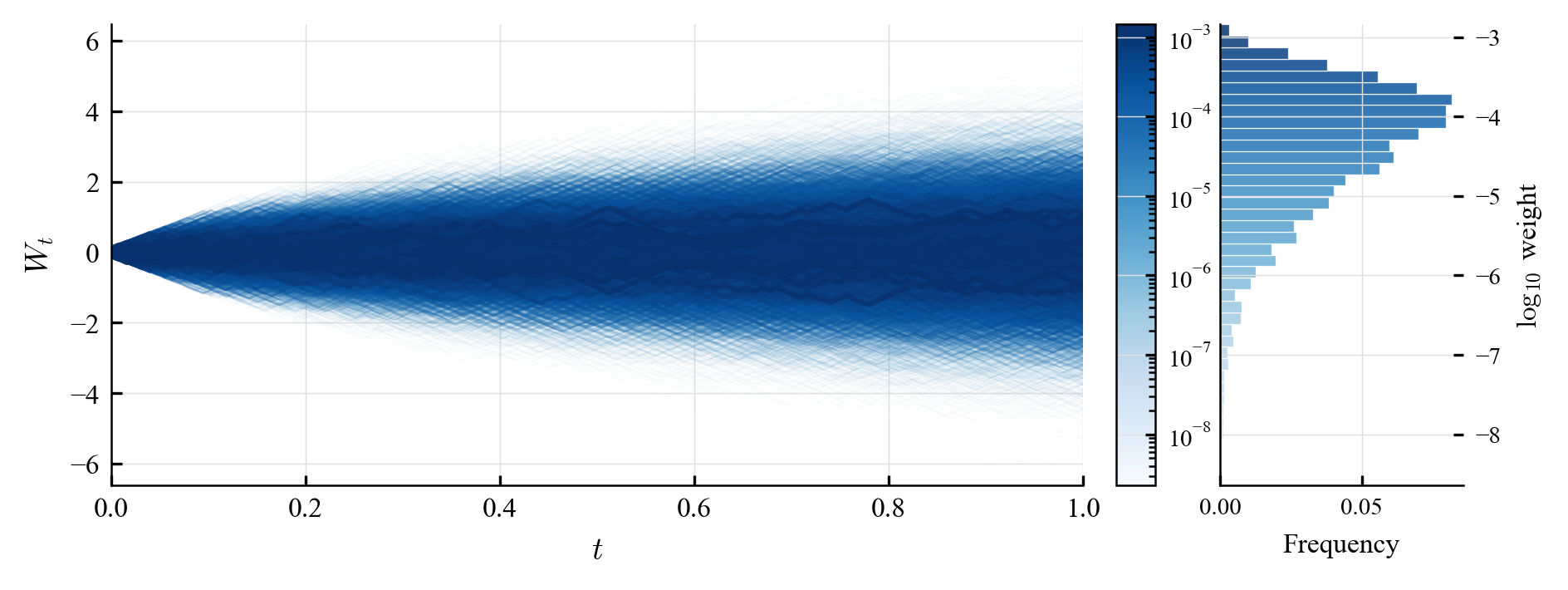}
    \bigbreak
    \includegraphics[width=\linewidth]{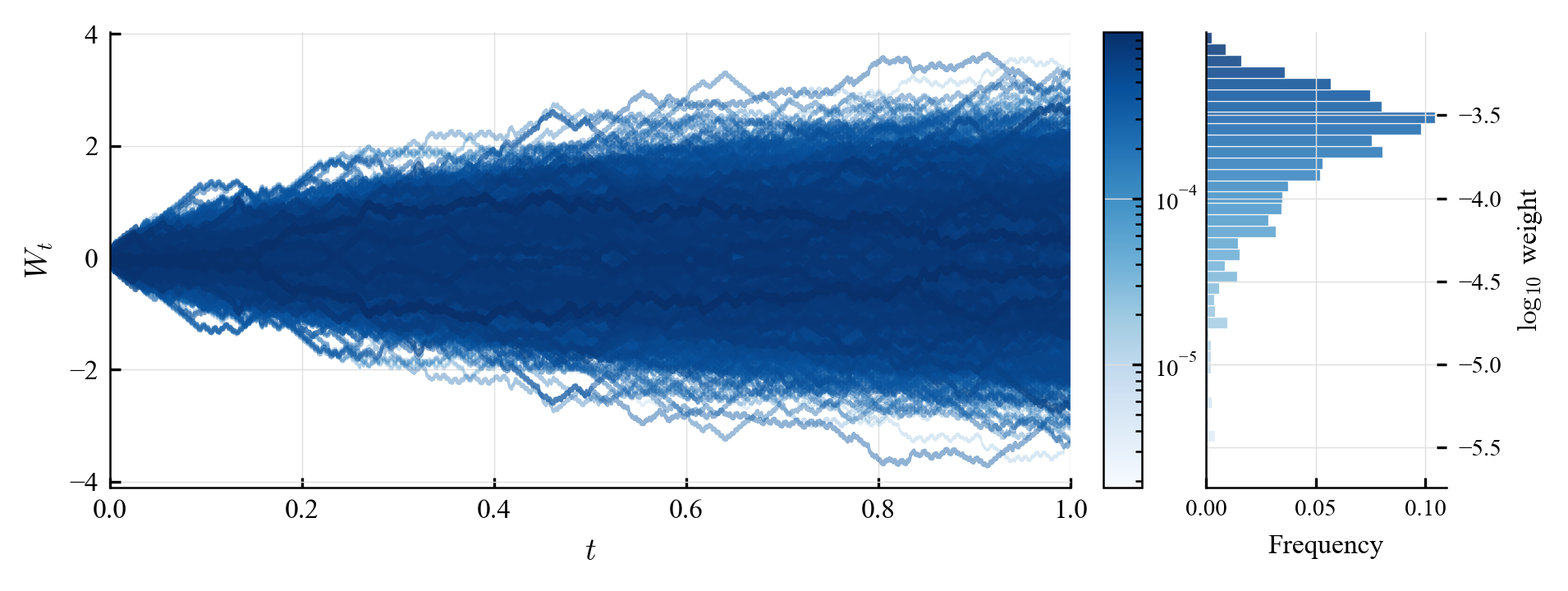}
    \caption{An illustration of our ARCANE cubature formulae with degrees 17, 19, and 5 (dyadic depth 8). As the formulae contain 3194, 8362, and 3952 paths, respectively, we have added noise to help distinguish paths. Here, the thickness and opacity of each path is taken to be proportional to its associated weight (i.e.~paths with larger weights are more visible); a histogram of the weights is to the right of each cubature formula.
    }
    \label{fig:degree_19_plot}
\end{figure}

In this paper, we aim to reduce this estimation error by replacing the uniformly weighted and independently sampled Brownian paths in \eqref{eq:estimator} with a carefully designed \emph{deterministic} set of Brownian-like paths with (not necessarily uniform) associated probability weights, see Table~\ref{tab:mc-vs-cubature} (right). 
These paths are piecewise linear so that the SDE~\eqref{eq:sde} simply reduces to an ODE along each piece, which can be solved by standard methods.
Experimentally, the estimation error using our paths is consistently orders of magnitude lower than the error obtained by a similar number of Monte Carlo samples.
An example of such a collection of Brownian-like paths produced by our methodology is illustrated in Figure~\ref{fig:degree_19_plot}.

More precisely, we propose a computational methodology called \short\footnote{\longg}, which is an efficient and scalable algorithm for producing what are known as \emph{cubature formulae}: sets of weighted paths that match the ``signature moments''~\cite{chevyrev2022signature} of Brownian motion up to some \emph{degree}~$D$. It was first observed by Lyons and Victoir in their seminal 2004 paper~\cite{lyons2004cubature} that cubature formulae could be used in place of Monte Carlo estimation for simulating SDEs in the way we outlined in Table~\ref{tab:mc-vs-cubature}. They proved that cubature formulae exist for arbitrary Brownian dimension $d$ and arbitrary degree~$D$, but were only able to find explicit constructions for $D \le 5$.
Further constructions have been found in the subsequent literature~\cite{gyurko2011cubature, shinozaki2017cubature, hayakawa2022cubature, malyarenko2022cubature, ferrucci2026cubature}, however, none have exceeded $D=7$.
In the one-dimensional setting, cubature formulae of degree $D=11$ have technically been constructed~\cite{gyurko2011cubature}, but they consist of high-order ``rough paths'';
solving the SDE~\eqref{eq:sde} driven by such a path involves the computation of hundreds to thousands of high-order derivatives (up to eleventh order) of the drift and diffusion vector fields, as well as their evaluation at every step of the numerical solver. Even if the derivatives were efficiently computed by a symbolic program, solving~\eqref{eq:sde} for one of these paths would take up to thousands of times longer than for an ordinary (e.g.~piecewise linear) path.

Our \short algorithm can automatically construct cubature formulae of arbitrary dimension and degree.
They consist of piecewise linear paths and can therefore be used plug-and-play in place of random Brownian sample paths.
Our algorithm is optimised for GPU and highly parallelisable; using \short, we were able to reproduce the state of the art $D=7$ in seconds, and construct cubatures with degree up to $D=19$ in one hour on a single GPU node.

Using several real-world SDEs (from mathematical finance and population genetics), we demonstrate that our \short cubature formulae consistently achieve orders of magnitude more accuracy when estimating statistics for SDEs compared to traditional (Quasi-)Monte Carlo simulation. 

Underlying the efficacy of our \short algorithm is a series of innovations that drastically reduce the run-time of a known and conceptually simple algorithm (from exponential to low-order polynomial in the relevant parameter, see Section~\ref{sec:methods} for details), combined with a highly optimized implementation of the algorithm's core components in JAX~\cite{jax2018github}.
Code for our \short algorithm along with datasets containing the cubature formulae used in our experiments can be found at:
\vspace{-1mm}
\begin{center}
\href{https://github.com/tttc3/ARCANE-Cubature}{github.com/tttc3/ARCANE-Cubature}
\end{center}
\vspace{-1mm}

We summarise our key contributions below:
\begin{enumerate}
    \item[(1)] We introduce \short, a scalable, highly parallelisable and GPU-optimised algorithm for generating high-degree cubature formulae for SDE simulation.\smallskip
    \item[(2)] We demonstrate that cubature formulae obtained with this algorithm on low-end compute (a few GPU hours) already improve drastically over the previously best known cubature formulae, and lead to an SDE estimation error several orders of magnitude smaller compared with traditional Monte Carlo methods.\smallskip
    \item[(3)] We publish our constructed cubature formulae, as well as the code for the \short algorithm, which can be used with more compute to construct even higher degree cubatures.\smallskip
    \item[(4)] We prove several theoretical results on the correctness of our approach, which may be of independent interest to members of the stochastic analysis and rough path theory communities.\vfill
\end{enumerate}

\section{Results}\label{sect:results}
In order to be able to estimate and compare the simulation errors of the cubature and the Monte Carlo SDE solvers, we need a ground truth to compare against. Most statistics of most SDEs are not available in closed form, but for a few SDEs, the mean and variance of the solution, $m(t) = \E \left[ y_t \right]$ and $v(t) = \V(y_t)$, are available in closed form (as a function of $t$, the initial condition $y_0$, and the parameters of the SDE). Then, for a given set of Monte Carlo paths, or a given cubature $(\lambda_i,\omega_i)_{i=1}^M$, we calculate the empirical mean and variance, \[
\hat{m}(t) = \sum_{i=1}^M \lambda_i y_t(\omega_i),\qquad \hat{v}(t) = \sum_{i=1}^M \lambda_i \Big(y_t(\omega_i) - \hat{m}(t)\Big)^2,
\] and combine them in the \emph{Mean-Variance error} (MVE)
\begin{equation}\label{eq:mean-variance-error}
    \mve
    = \sqrt{(\hat{m}(T) - m(T))^2 + (\sqrt{\hat{v}(T)} - \sqrt{v(T)})^2  } ,
\end{equation}
for a fixed time $T$ (where $T = 1.0$ unless stated otherwise). The motivation behind formula~\eqref{eq:mean-variance-error} is that it describes the $2$-Wasserstein distance between two one-dimensional Gaussians $\mathcal{N}(\hat{m}(T),\hat{v}(T))$ and $\mathcal{N}(m(T),v(T))$.
We study the MVE for four different one-dimensional SDEs summarised in Table~\ref{table:sdes}, with results presented in Fig.~\ref{fig:sde-multifigure}.

\begin{figure}
    \centering
    \adjustbox{center}{%
        \includegraphics[width=1.05\linewidth]{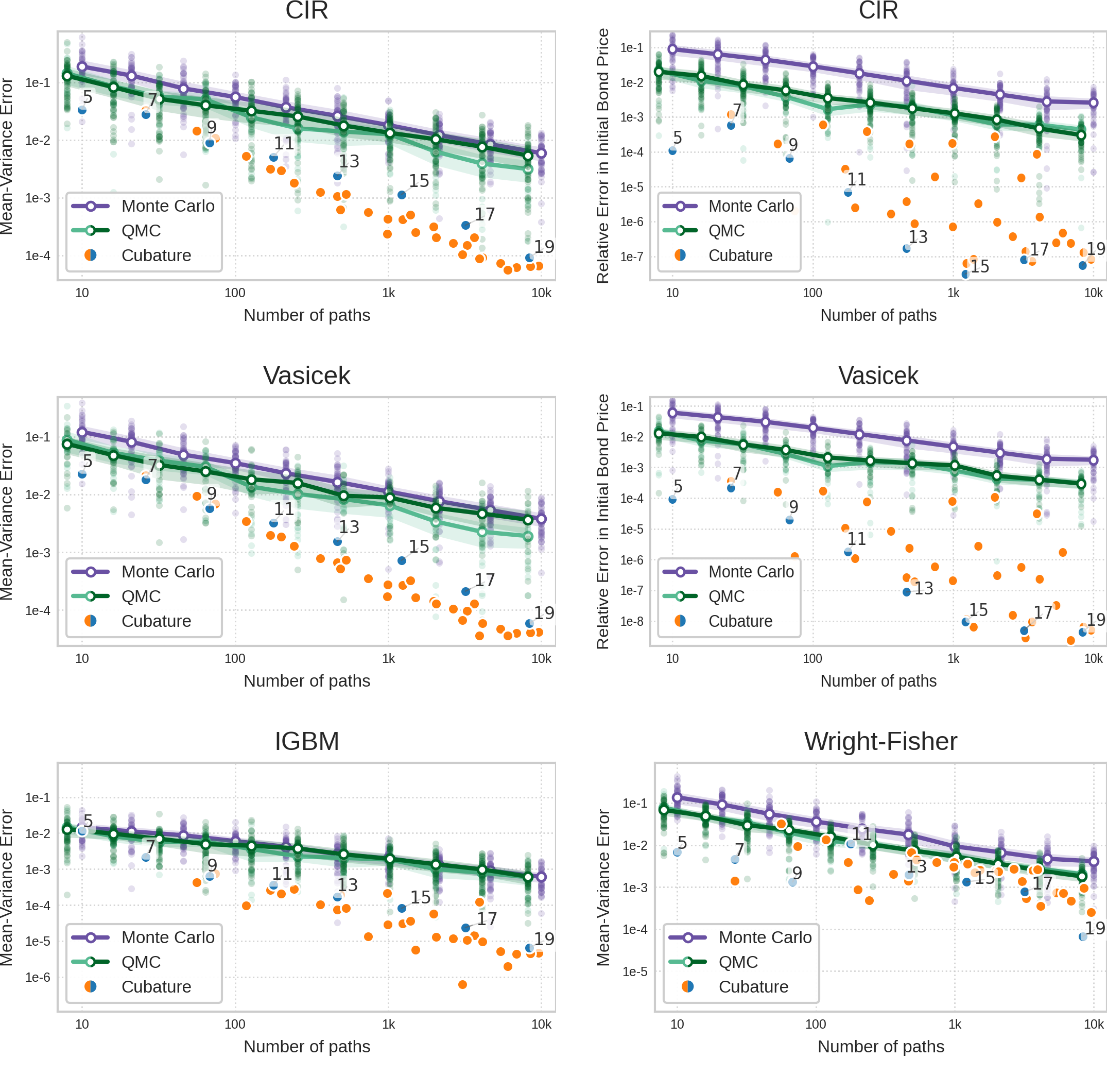}
    }
    \caption{Error plots showing the performance of our cubature formulae in comparison with plain Monte Carlo and QMC methods across a range of different SDEs and error metrics. Unlabelled cubature formulae in orange are \emph{dyadic} cubature formulae, see Section~\ref{sec:methods} as well as Appendix~\ref{app:plots} for full-sized fully labelled versions of the same plots.}
    \label{fig:sde-multifigure}
\end{figure}

\begin{figure}
    \centering
    \includegraphics[width=5.28in]{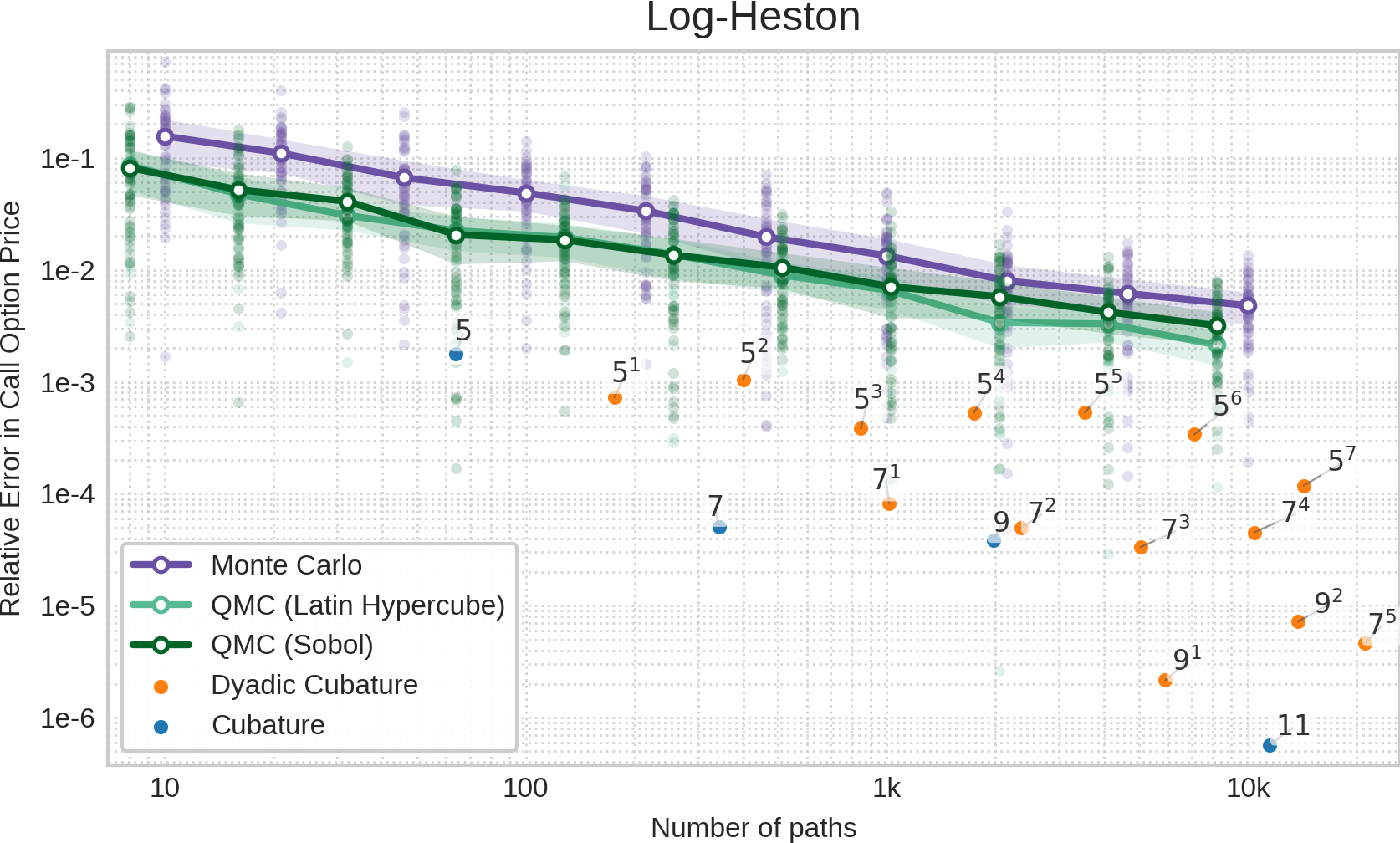}
    \caption{Error plot showing the performance of our cubature formulae in comparison with plain Monte Carlo and QMC methods measured in \emph{Call Price Relative Error} (see~\eqref{eq:loghestoncallprice}) in the log--Heston model. Superscripts in cubature labels refer to the \emph{dyadic depth} of the cubature, see Section~\ref{sec:methods}.}
    \label{fig:log-heston}
\end{figure}

\begin{table}[h]
    \centering
    \begin{tabular}{@{} l l l l @{}}
    \toprule
    Model & Stochastic Differential Equation (It\^{o} form) & Results & Details\\
    \midrule
    Vasicek / OU
    & $\diff X_{t} = a(b - X_t)\diff t + \sigma \diff W_t$
    & \multirow{4}{*}{Fig.~\ref{fig:sde-multifigure}}
    & \Cref{app:ou} \\[2pt]

    IGBM
    & $\diff X_{t} = a(b - X_{t}) \diff t + \sigma X_{t} \diff W_{t}$
    & 
    & \Cref{app:igbm} \\[2pt]

    CIR
    & $\diff X_{t} = a(b - X_{t}) \diff t + \sigma \sqrt{X_{t}} \diff W_{t}$
    & 
    & \Cref{app:cir} \\[1pt]

    Wright--Fisher
    & $\diff X_{t} = s X_{t}(1 - X_{t}) \diff t + \sqrt{\gamma X_{t}(1 - X_{t})} \diff W_t$
    & 
    & \Cref{app:wf} \\
    \midrule

    Log-Heston
    & $\begin{aligned}
        \diff X_t &= \Big(\mu - \frac{V_t}{2}\Big) \diff t + \sqrt{V_t} \diff W_t^x,\\
        \diff V_t &= a(b - V_t) \diff t + \sigma \sqrt{V_{t}} \diff W_t^v,\\
        \rho &= \mathrm{Correlation}(W^x, W^v).
    \end{aligned}$
    & Fig.~\ref{fig:log-heston}
    & \Cref{app:heston}\\
    \bottomrule
    \end{tabular}
    \caption{Examples of widely-used SDEs.}
    \label{table:sdes}
\end{table}

For the CIR and Vasicek models, there are furthermore closed formulas for the \emph{Initial Bond Price}
\begin{equation}\label{eq:initialbondprice}
    B(T) = \E \left[ \exp\left( - \int_0^T X_s \diff s\right)\right].
\end{equation}
If $X_t$ models a short rate, then $B(T)$ is the fair price at time zero of a zero-coupon bond paying~$1$ at maturity time~$T$ (see e.g.~\cite[Ch.~15]{shortrate}).
For a given cubature or set of Monte Carlo paths, we then compute the empirical expectation $\hat{B}(T)$ and with it the relative error $|B(T) - \hat{B}(T)| / B(T)$, which we call the \emph{Initial Bond Price Relative Error}, see Figure~\ref{fig:sde-multifigure}. Again, we use $T = 1.0$.
An advantage of this error metric over the MVE is that the bond price depends on the entire solution path rather than just its final value.

Finally, we also consider the log--Heston model (see Table~\ref{table:sdes}), which is driven by \emph{two} independent Brownian motions and therefore requires \emph{two}-dimensional cubature formulae to simulate. For the log--Heston model, there is a known closed formula for the \emph{Call Price}
\begin{equation}\label{eq:loghestoncallprice}
    C(T) = \e^{-\mu T} \E\left[ \left( \e^{X_T} - K\right)_+\right],
\end{equation}
which is the price at time zero of a European call option maturing at time $T$ with strike price $K$ (we use $T = 1.0$ and $K = 2.0$). Analogously to the bond price, we use the \emph{Call Price Relative Error} as the error metric for the log--Heston model, see Fig.~\ref{fig:log-heston}.

In addition to plain Monte Carlo, the results in Fig.~\ref{fig:sde-multifigure} and~\ref{fig:log-heston} also include two Quasi--Monte Carlo (QMC) methods as baselines. When using QMC, one replaces the sequence of independent samples of Gaussian increments of the discretised Brownian motion driving the SDE in Monte Carlo with deterministic sequences that are deliberately constructed to explore the range of possible outcomes more uniformly. In low-dimensional problems, QMC is known to improve the convergence rate of Monte Carlo from $O(1 / \sqrt{M})$ to $O(1 / M)$, see e.g.~\cite{wangqmc,niederreiterqmc}. No guaranteed error bound is known in the infinite-dimensional setting of SDE simulation, but empirically QMC usually performs slightly better than plain Monte Carlo.  In our experiments, we include Sobol sequences~\cite{sobol} and Latin hypercube sequences~\cite{latin} as QMC baselines (dark and light green in Fig.~\ref{fig:sde-multifigure}, respectively). We also tested Halton sequences~\cite{halton}, but they were consistently inferior (typically close to ordinary Monte Carlo), so we omit them from the plots.

As the results in Fig.~\ref{fig:sde-multifigure} and~\ref{fig:log-heston} show, our cubature formulae consistently outperform both standard Monte Carlo and QMC by orders of magnitude. The only exception is the Wright--Fisher Diffusion (bottom right in Fig.~\ref{fig:sde-multifigure}), where cubature formulae only perform marginally better. We believe this is due to the fact that the Wright--Fisher SDE has non-smooth diffusion coefficients close to the boundary; we go into detail on this in the next section.
Note further that the error in the CIR and Vasicek Initial Bond Price Relative Error plots (Fig.~\ref{fig:sde-multifigure} right column, top and center row) plateaus between $10^{-7}$ and $10^{-8}$.
We believe that this is because the cubature error becomes so small that the overall accuracy is instead limited by numerical errors introduced by the ODE solver, which are of the same order of magnitude.

For the IGBM model, we also illustrate the actual solution paths for one of our smaller cubatures, compared with an equivalent number of Monte Carlo paths, see Figure~\ref{fig:solutionpaths}.

\subsection{Limitations}

\paragraph{Smoothness}
Just as standard quadrature formulae can lose accuracy when the test functions are non-smooth, the accuracy of our cubature formulae will depend on the smoothness of the SDE coefficients. For example, the diffusion coefficient of the Wright--Fisher diffusion, $\sqrt{\gamma X_t (1-X_t)}$, is not smooth close to the boundaries at $0$ and $1$. But unlike in the CIR model, where a similar non-smoothness exists at the boundary $0$ but most solution paths never get close to it, solution paths to the Wright--Fisher diffusion do interact strongly with and frequently get absorbed by the boundary. To illustrate this effect, Fig.~\ref{fig:WF-multifigure} shows error plots for the Wright--Fisher diffusion when simulated until times $T=0.1$, $T = 1.0$, and $T = 10.0$, together with the associated solution paths from our degree $19$ cubature. When simulated until time $T = 0.1$, most solution paths have not reached the boundary and our cubatures perform much better than Monte Carlo, to a degree that is comparable with the other Mean--Variance plots in Fig.~\ref{fig:sde-multifigure}. When simulated until time $T = 1.0$ (as in Fig.~\ref{fig:sde-multifigure}), a significant proportion of the paths have been absorbed by or are close to the boundary where the diffusion coefficients are non-smooth, and the cubature formulae drop in performance. Interestingly, when simulated until time $T = 10.0$, cubature formulae regain their advantage. We suspect this is because at that time, most solution paths have been absorbed by the boundary, which simplifies the distribution of $y_T$, making its mean and variance easier to estimate.

\begin{figure}
    \centering
    \includegraphics[width=\linewidth]{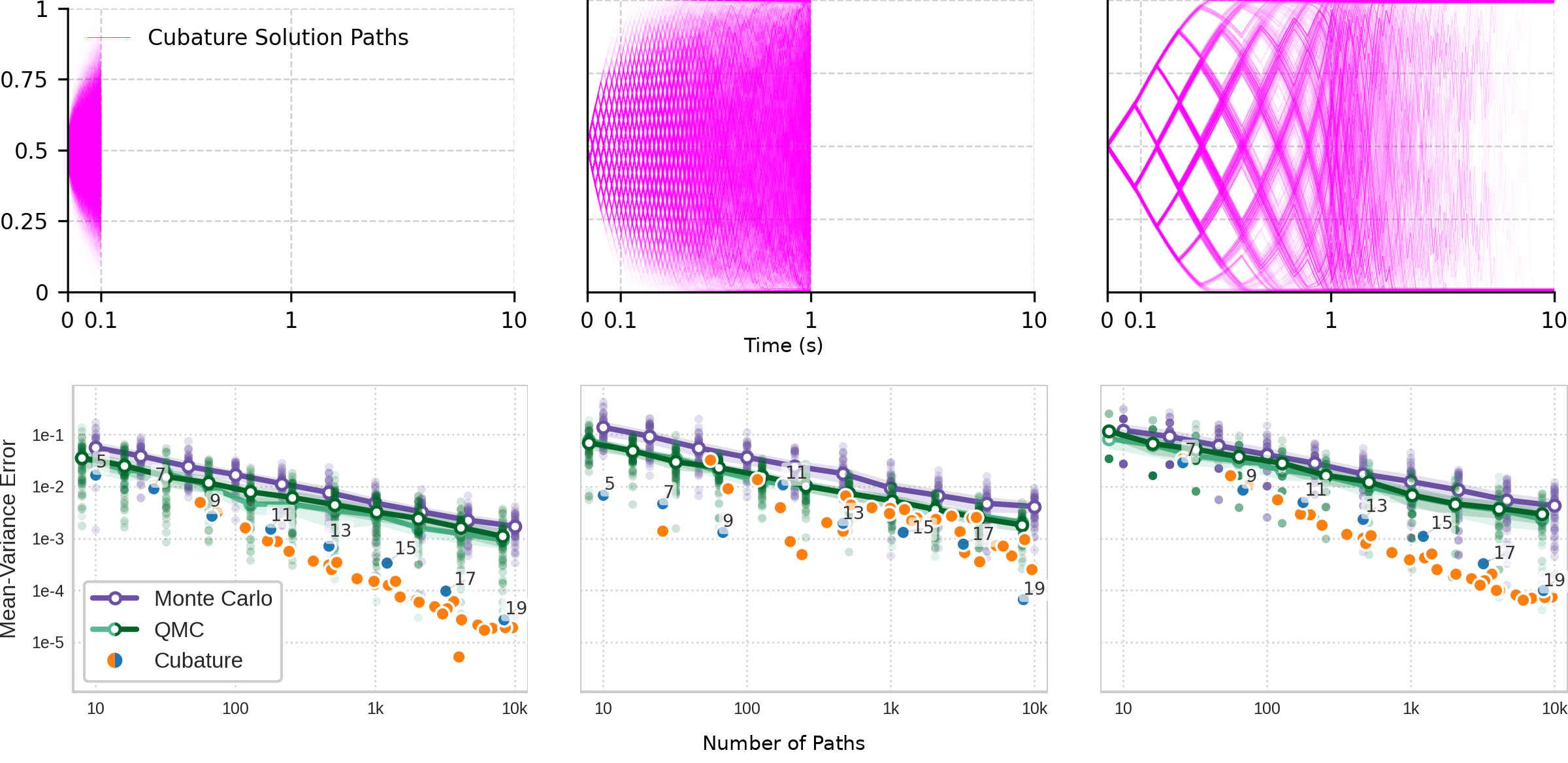}
    \caption{Solution paths generated by the degree-19 cubature formula for the Wright–Fisher diffusion, simulated up to final times $T=0.1$, $T=1.0$, and $T=10.0$ (top row, left to right).
    Corresponding MVE plots are shown in the bottom row, with a shared y-axis scale.
    }
    \label{fig:WF-multifigure}
\end{figure}

\paragraph{Long time horizons}
Our cubature formulae were constructed for a unit time interval $T = 1.0$. Although they can be rescaled to longer horizons—as done in the experiments shown in Fig.~\ref{fig:WF-multifigure}—one would generally expect their accuracy to degrade as the time horizon increases. To assess this effect, we simulated the IGBM model over time intervals of length $T = 1.0$ (as in Fig.~\ref{fig:sde-multifigure}), $T = 3.0$, and $T = 10.0$. The resulting error plots are shown in Fig.~\ref{fig:IGBM-multifigure}. As anticipated, performance does deteriorate with increasing time horizon, but the effect is noticeably weaker than expected. In particular, the cubature formulae retain a consistent advantage over Monte Carlo and QMC methods across all horizons considered. We also observe that dyadic cubatures exhibit a more pronounced loss of accuracy than their non-dyadic counterparts, a phenomenon for which we currently have no clear explanation.

\begin{figure}
    \centering
    \includegraphics[width=\linewidth]{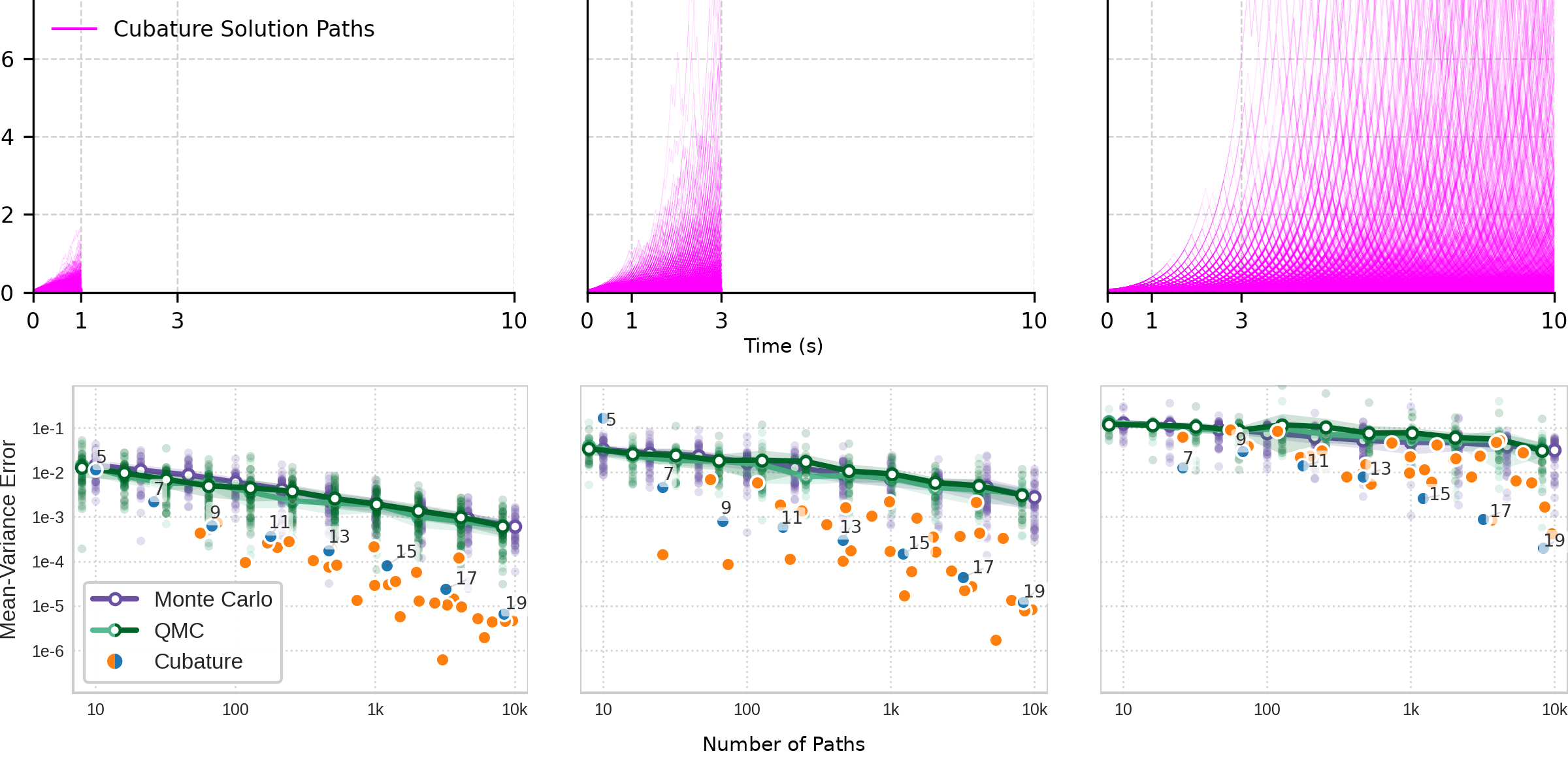}
    \caption{Solution paths generated by the degree-19 cubature formula for the IGBM SDE, simulated up to final times $T=1.0$, $T=3.0$, and $T=10.0$ (top row, left to right).
    Corresponding MVE plots are shown in the bottom row, with a shared y-axis scale.
    }
    \label{fig:IGBM-multifigure}
\end{figure}

\begin{figure}
    \centering
    \includegraphics[width=\linewidth]{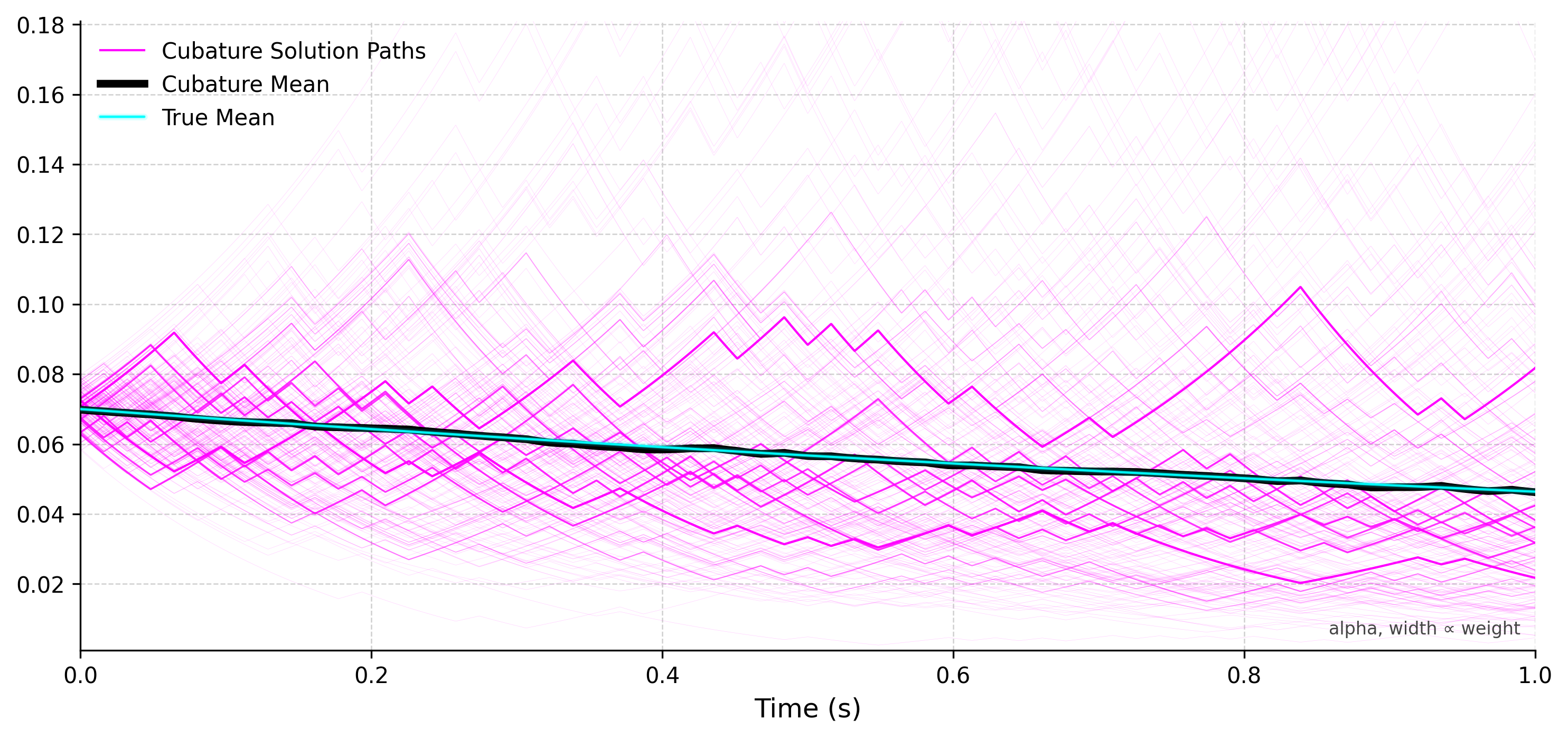}
    \includegraphics[width=\linewidth]{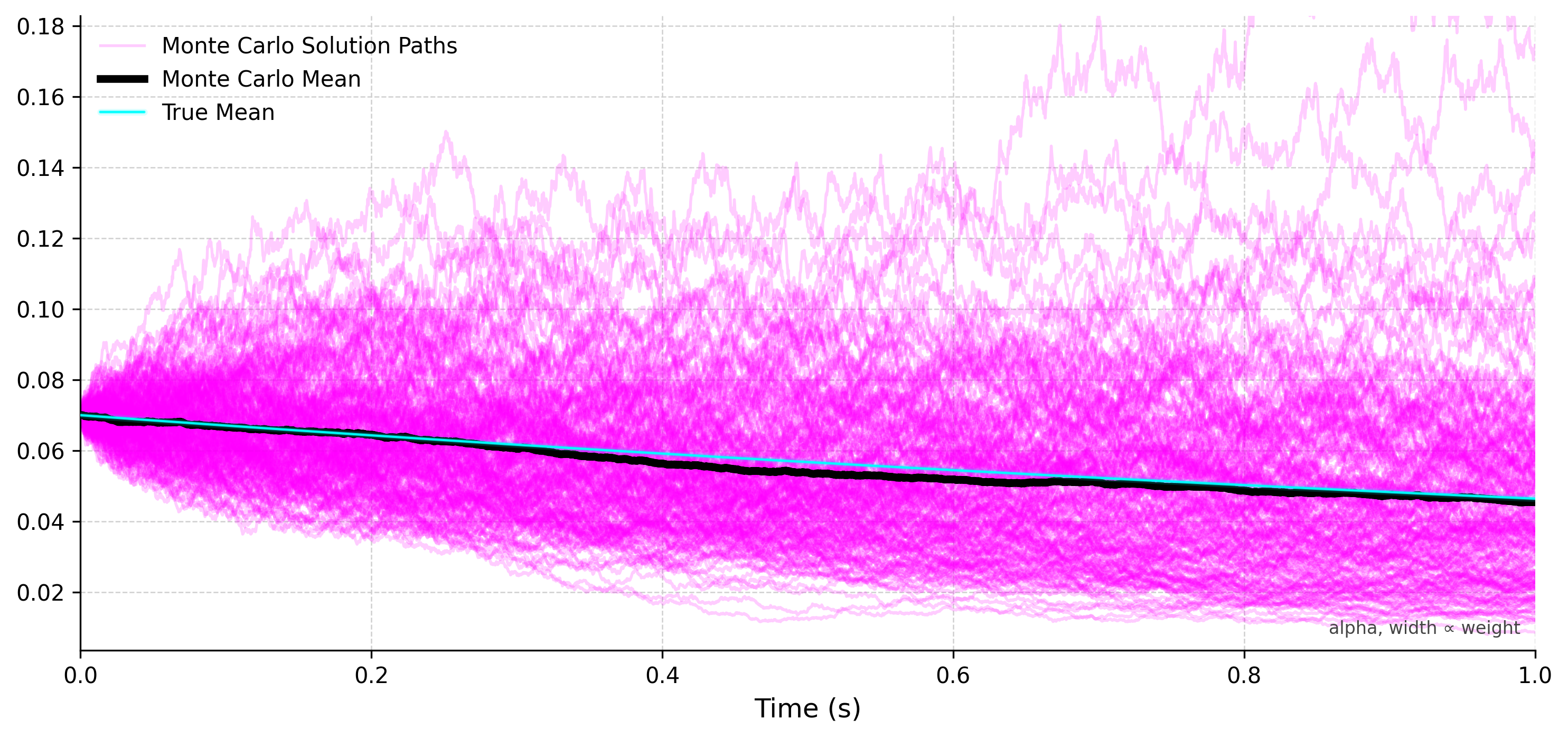}
    \caption{Illustration of the solution paths of IGBM, $\diff X_{t} = a(b - X_{t}) \diff t + \sigma X_{t} \diff W_{t}\m$, when driven by our $9^1$ cubature (degree $9$, dyadic depth $1$), which consists of 200 paths (top), and the solution paths to the same SDE when driven by 200 randomly sampled (Monte Carlo) paths (bottom). Despite consisting of the same number of paths, and hence having the exact same cost when used to solve the SDE, the mean of the cubature solution paths stays much closer to the true mean than that of the Monte Carlo solution paths. (The cubature looks as though it has fewer paths because many of its paths have small weight and are therefore less visible.)}
    \label{fig:solutionpaths}
\end{figure}
\section{Methods}\label{sec:methods}
In this section, we give a detailed explanation of how the \short algorithm works and present our main theoretical results. However, before we can do either, we start with a brief primer on the relevant mathematical background and notation.

\subsection{Background}

The mathematical foundation underlying our \short algorithm relies on concepts from rough path theory and the theory of signatures. We provide here a minimal introduction to these concepts, referring the reader to Appendices~\ref{app:roughpaths} and~\ref{app:theoretical results} for a more comprehensive treatment.

\paragraph{Signatures and Cubature Formulae}
Given a path $x\colon [0,T] \to \R^{d+1}$, where $x^0(t) = t$ is fixed to be the time coordinate, the \emph{signature} of $x$ is the collection of all its iterated integrals:
\[
\sig(x) = \left(1, \sig^1(x), \sig^2(x), \ldots\right),
\]
where $\sig^m(x) = ( \sig^J(x)\colon J \in \{0,\ldots,d\}^m)$ is the collection of $m$-fold iterated integrals of $x$. That is, if $J = (j_1,\ldots, j_m) \in \{0,\ldots,d\}^m$, then \[
\sig^J(x) = \int_{0 < u_1 < \cdots < u_m < T} \diff x^{j_1}(t) \ldots \diff x^{j_m}(t) \in \R.
\]
For Brownian motion, these integrals are understood in the Stratonovich sense. The \emph{degree} of a multi-index $J = (j_1,\ldots,j_m)$ and its associated signature term $\sig^J(x)$ is 
\begin{equation}\label{eq:degreeJ}
   \|J\| = m + |\{i\colon j_i = 0\}|,
\end{equation}
that is, the number of terms in $J$ with zeros (which correspond to the time coordinate) counted twice. This definition is rooted in the fact that the Brownian signature $\sig^J_{0,t}(W)$ is equal in distribution to $(\sqrt{t}\,)^{\|J\|} \sig^J_{0,1}(W)$.

The signature captures important geometric information about a path; the first level is the increment, $\sig^1(x) = x(T) - x(0)$, the second level encodes the area enclosed between the path and the straight line from its startpoint to its endpoint, and the higher order terms can be viewed as higher order geometric information (e.g.~volumes).

A key insight of rough path theory (see Appendix~\ref{app:roughpaths} for a short and \cite{terrystflour} for a comprehensive introduction) is that the signature of a path determines how it behaves when used as a driving signal in differential equations like \eqref{eq:sde}. Another insight is that the \emph{expected} signature of a \emph{random} path captures its distribution (see Lemma~\ref{lem:expsig characteristic}). In combination, this motivates the definition of a \emph{cubature formula on Wiener space}. That is, a collection of \emph{deterministic} paths $(\omega_i)_{i=1}^M$ with associated probability weights $(\lambda_i)_{i=1}^M$ which seek to ``emulate'' a Brownian motion by matching its expected signature up to some degree~$D$:
\[
\sum_{i=1}^M \lambda_i \sig^J(\omega_i) = \E\big[\sig^J(B)\big],
\]
for all $J$ with $\|J\| \leq D$. Then, $D$ is called the \emph{degree} of the cubature.

\paragraph{Size of Cubature Formulae}
Lyons and Victoir~\cite[Thm.\ 2.4]{lyons2004cubature} proved that for any $d,D\in\N$, a $d$-dimensional cubature formula of degree $D$ exists that consists of at most
\begin{equation}\label{eq:cubature size}
    M(d,D) \coloneqq \left| \left\{ J \in \{0, \ldots ,d\}^*\colon \left\|J\right\| \in \{D-1,D\} \right\}  \right|
\end{equation}
paths, where $\{0, \ldots ,d\}^* = \bigcup_{k =0} ^\infty \{0, \ldots ,d\}^k$ is the set of all multi-indices.
See Table~\ref{tab:cubature size comparisons} for a few example values of $M(d,D)$.

Our \short algorithm constructs cubature formulae with size at most $M(d,D)$, often even smaller (see Table~\ref{tab:cubature size comparisons}).

\paragraph{Connection to rough path theory}
Related to cubature formulae for SDEs is rough path theory---specifically expected path signatures \cite{chevyrev2022signature} and the universal limit theorem (Theorem~\ref{thm:universal limit theorem} in the appendix).
 
This theorem guarantees that if a sequence of cubature formulae converges to Brownian motion in an appropriate sense, then solutions of SDEs driven by these cubatures will converge to the true solution driven by Brownian motion. This also underpins our theoretical results in Section~\ref{sec:theory}, where we prove asymptotic correctness of our cubature formulae under certain assumptions. This provides a theoretical justification for replacing Monte Carlo sampling with deterministic cubature formulae for SDE simulation, as outlined in Table~\ref{tab:mc-vs-cubature}.

Prior to our work, the universal limit theorem had seen applications in the context of SDE simulation. For example, in \cite{foster2024adaptive}, it was used to show the pathwise convergence of numerical methods for simulating multidimensional SDEs with adaptive step sizes.

\subsection{The \short algorithm}

As an input to the algorithm, we fix a Brownian dimension $d\ge 1$ and a degree $D \ge 1$.
The \short algorithm consists of the following core steps, each of which we explain in more detail below and illustrate in Fig.~\ref{fig:arcane_algorithm}.
\begin{itemize}
    \item We start by constructing an \emph{approximate} cubature from a binary orthogonal array (or, equivalently, from a kind of error-correcting code).\medbreak
    
    \item We then compress that approximate cubature to minimal size through an algorithm called ``Recombination'' \cite{litterer2012cubature, tchernychova2015cubature}---implemented efficiently on GPU.\medbreak
    
    \item Next, we use a linear program to modify the weights of the compressed, approximate cubature in such a way that it becomes an exact cubature.\medbreak
    
    \item In an optional final step, we take the product of the cubature with itself and recombine it again, yielding a cubature that matches the Brownian signature moments to the given degree $D$ not only on $[0,1]$, but also on $[0,0.5]$ and $[0.5,1]$. This step can be repeated for any number of times, yielding cubatures of arbitrary \emph{dyadic depth}.
\end{itemize}

\begin{figure}
    \centering
    \includegraphics[width=\linewidth]{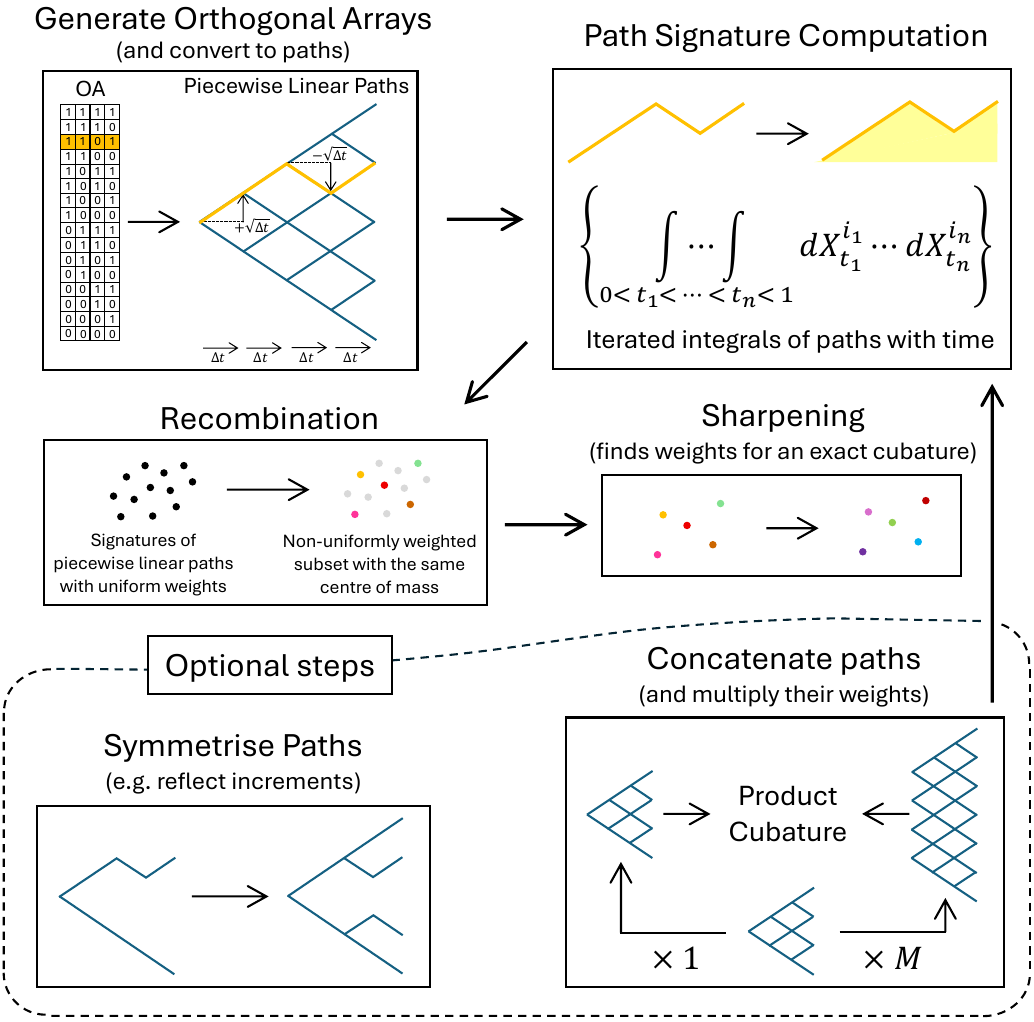}\vspace{1mm}
    \caption{High-level overview of the ARCANE algorithm.}
    \label{fig:arcane_algorithm}
\end{figure}

We now describe each step in more detail.

\medskip\textbf{Step 1: Construction of Approximate Cubature}

\par\noindent\textbf{Outcome}: Approximate cubature of size $(Nd)^{\lfloor D/2\rfloor}$ (with uniform weights).

An obvious way of generating approximate cubatures are product constructions. In the simplest case, we could divide $[0,1]$ into $N$ equally sized subintervals, and consider the (uniformly weighted) set of $2^{Nd}$ paths that take a step of size $\pm 1 / \sqrt{N}$ in each of the $d$ dimensions on each of the $N$ intervals, with linear interpolations in between. This is nothing but the simple symmetric random walk on $[0,1]$ with $N$ steps and Brownian scaling, which is well-known to converge to Brownian motion as $N\to \infty$. Its expected signature can be shown to be within a relative error of $O(1 / N)$ of that of Brownian motion up to any fixed degree (see Lemma~\ref{lem:product-cubature}), making it a valid yet impractically large approximate cubature.

Our key observation is that constructions from the theory of error-correcting codes and orthogonal arrays~\cite{hedayat1999,huffman2003} can be adapted to find a subset of those $2^{Nd}$ paths of much smaller size $(Nd)^{\lfloor D / 2\rfloor}$ that, again with uniform weights, has the \emph{same} expected signature up to the relevant degree $D$.
This reduction in size is transformative; for example, for dimension $d=1$ and degree $D=9$, if $N=64$ (which we found to be minimal for the success of step 3 below), the naive construction would consist of $2^N \approx 1.8\times 10^{19}$ paths, which our construction reduces to $N^4 \approx 1.7\times 10^7$ paths. Further examples are given in Table~\ref{tab:cubature size comparisons}.

More specifically, suppose that $A$ is a binary orthogonal array with $Nd$ columns and strength $D$. That is, $A$ is an $M \times Nd$ matrix with $\pm 1$ entries such that any $M \times D$ submatrix obtained by selecting any $D$ columns has the property that all $2^D$ possible rows appear the same number of times. The probabilistic interpretation of a binary orthogonal array with strength $D$ and $Nd$ columns is that of a coupling between $Nd$ Rademacher random variables (coin flips $\pm 1$) such that any $D$ of them are independent. We construct our cubature formulae from $A$ by transforming each row of the matrix into a path, whose increments are equal to the row's entries (in arbitrary order, and scaled by $1 / \sqrt{N}$).

An orthogonal array can be constructed, for example, by letting its rows be the code words of a binary error-correcting code of length $Nd$ with ``dual distance'' $D+1$ (see Definition~\ref{def:code} and Theorem~\ref{thm:dualdistance}, or~\cite[Thm.~4.9]{hedayat2012}). If $Nd$ is a power of two, then using a certain code based on work of Bose and Ray~\cite{bose-ray1960} yields such an orthogonal array with $M = (Nd)^{\lfloor D / 2\rfloor}$ rows.

We use the ``orray'' python package~\cite{orray}, which contains efficient JAX~\cite{jax2018github} implementations of state-of-the-art orthogonal arrays, and was implemented primarily by the first author of this paper.

Note that flipping all entries of an arbitrary column of an orthogonal array yields another orthogonal array. We have found that it increases the stability of the algorithm to randomly flip each column with probability $1/2$ each, independently of each other. One advantage of this is that, regardless of the original array, the marginal distribution of every row will be a sequence of independent coin flips, matching the distribution of a simple random walk.

The run-time of this step is $O(n \log n)$ where $n$ is the size of the output (the approximate cubature).

We note that orthogonal arrays have previously been used to construct cubature formulae for measures on $\R^n$ \cite{victoirOA}.

\begin{table}[t]
\caption{\small For different degrees $D\geq 5$, and $d=1$ fixed, this table lists the minimal $N$ (restricted to powers of $2$) for which the sharpening in Step~3 succeeds, the output size of Step~1 (naive and actual), and an upper bound on the output size of Steps~2 and~3 (with and without symmetrisation).
}\label{tab:cubature size comparisons}
\begin{tabularx}{\textwidth}{*{3}{C}cccc}
\toprule
\makecell[c]{\small Dimension} &
\makecell[c]{\small Degree} &
\makecell[c]{\small Minimum $N$} &
\multicolumn{2}{c}{\makecell[c]{\small Step 1\\[2pt] \small Output Size}} &
\multicolumn{2}{c}{\makecell[c]{\small Steps 2 \& 3\\[2pt] \small Output Size}} \\
\cmidrule(lr){4-5}
\cmidrule(lr){6-7}
$d\vphantom{\Big|}$ & $D$ & $N$ & \makecell[c]{\footnotesize naive \\[2pt] \small $2^{Nd}$} & \makecell[c]{\footnotesize actual\\[2pt] \small $(Nd)^{\lfloor D/2 \rfloor}$} & \makecell[c]{\footnotesize default\\[2pt] \small $M(d,D)$} & \makecell[c]{\footnotesize symmetrised\\[2pt] \small $2^d M_{\text{even}}(d,D)$} \\
\midrule
$1$ & $\phantom {1}5$  & $\phantom{1}32$\phantom{?} & $4.3 \times 10^{9\phantom{1}}$ & $1.0 \times 10^{3\phantom{1}}$ & $\phantom{111}13$  & $\phantom{11}10$  \\
$1$ & $\phantom{1}7$  & $\phantom{1}32$\phantom{?} & $4.3 \times 10^{9\phantom{1}}$ & $3.3 \times 10^{4\phantom{1}}$ & $\phantom{111}34$  & $\phantom{11}26$  \\
$1$ & $\phantom{1}9$  & $\phantom{1}64$\phantom{?} & $1.8 \times 10^{19}$           & $1.7 \times 10^{7\phantom{1}}$ & $\phantom{111}89$            & $\phantom{11}68$            \\
$1$ & $11$            & $\phantom{1}64$\phantom{?} & $1.8 \times 10^{19}$           & $1.1 \times 10^{9\phantom{1}}$ & $\phantom{11}233$            & $\phantom{1}178$            \\
$1$ & $13$            & $\phantom{1}64$\phantom{?} & $1.8 \times 10^{19}$           & $6.9 \times 10^{10}$           & $\phantom{11}610$            & $\phantom{1}466$            \\
$1$ & $15$            & $\phantom{1}64$\phantom{?} & $1.8 \times 10^{19}$           & $4.4 \times 10^{12}$           & $\phantom{1}1597$            & $1220$            \\
$1$ & $17$            & $\phantom{1}64$\phantom{?} & $1.8 \times 10^{19}$           & $2.8 \times 10^{14}$           & $\phantom{1}4181$            & $3194$            \\
$1$ & $19$            & $128$ & $3.4 \times 10^{38}$           & $9.2 \times 10^{18}$           & $10976$            & $8362$            \\
\bottomrule
\end{tabularx}
\end{table}

\medskip\textbf{Step 2: Recombination}

\par\noindent\textbf{Outcome}: Approximate cubature of size at most $M(d,D)$ (with non-uniform weights)

\medskip\noindent
\emph{Recombination} \cite{litterer2012cubature,tchernychova2015cubature} is an algorithm that takes a set of points $(x_i)_{i=1}^M$ in $\R^m$ for some $m \in \N$, with probability weights $(\lambda_i)_{i=1}^M$, and returns a subset $I \subset \{1,\ldots,M\}$ of size $|I| \le m + 1$ together with new probability weights $(\tilde{\lambda}_i)_{i\in I}$ such that the weighted mean is preserved: \[
\sum_{i=1}^M \lambda_i x_i = \sum_{i\in I} \tilde{\lambda}_i x_i\m.
\]
Carath\'eodory's convex hull theorem guarantees that such a subset always exists, and that the size $m+1$ of the subset is optimal in general (in the sense that there exist examples where no smaller subsets with the desired property exist; for example, consider the $d+1$ corners of the standard simplex in $\R^d$ with e.g.\ uniform weights). The run-time of recombination is $O(Mm + \log(M / m) m^3)$~\cite[page 61]{tchernychova2015cubature}.

Since signatures can be viewed simply as vectors of real numbers, we can apply recombination to the approximate cubature obtained in Step 1 to obtain a new approximate cubature with the same expected signature, but size at most $M(d,D)$ (recall~\eqref{eq:cubature size}), and in general with non-uniform weights.
For example, if $d=1$ and $D=9$, then the size of the compressed cubature is at most $143$, down from $1.6\times 10^7$. Examples for more values of $D$ and $N$ are in Table~\ref{tab:cubature size comparisons}.

We use the Baryx library~\cite{baryx}, which contains an efficient JAX~\cite{jax2018github} implementation of recombination that was implemented primarily by the second author of this paper.

\smallskip\noindent\textbf{Symmetrisation.} An optional speed-up in this step can be achieved by exploiting symmetries: a $d$-dimensional Brownian motion $B = (B^{(1)},\ldots,B^{(d)})$ has the same distribution as $(\varepsilon_1 B^{(1)},\ldots,\varepsilon_d B^{(d)})$ for any $\varepsilon_1,\ldots,\varepsilon_d \in \{-1,+1\}$.
This implies that many signature terms in the expected Brownian signature are automatically zero---namely the \emph{odd} terms, which are those corresponding to multi-indices $J$ where at least one non-zero index $j\in \{1,\ldots,d\}$ appears an odd number of times. By enforcing the same set of symmetries on our cubature paths, we can ensure that their expected signature is also zero at odd terms, and hence matches exactly. The recombination algorithms must then only be applied to \emph{even} terms, 
which makes it significantly faster since recombination scales cubically in the number of terms whose weighted mean is preserved.
Furthermore, the size of the resulting cubature is $2^d M_\text{even}(d,D)$, where \[
[
    M_\text{even}(d,D) = \left| \left\{ I \in \{0, \ldots ,d\}^*\colon \|I\|\le \{D-1,D\} \text{ and $\|I\|$ even} \right\}  \right| .
\] (Cmp.\ \eqref{eq:cubature size}.) A comparison with $M(d,D)$ is in Table~\ref{tab:cubature size comparisons} for a few example values.

\medskip\textbf{Step 3: Sharpening}

\par\noindent\textbf{Outcome}: If successful, exact cubature of size $M(d,D)$ or $2^d M_\text{even}(d,D)$.

In this step, we check whether the weights of the approximate cubature obtained in Step~2 can be adjusted in such a way that the resulting reweighted cubature is exact. There are different ways of doing so, and we implemented several (including gradient descent, a linear program, or an alternating projections method; see our code base for details). We have no theoretical guarantee for the success of this step, but in our experiments it always succeeded if $N$ was sufficiently large. We record the minimum value of $N$ (restricted to powers of $2$) required for the sharpening to succeed as a function of the degree $D$ in Table~\ref{tab:cubature size comparisons}. We also found that it can be helpful to stop the recombination in Step 2 early, e.g.\ stop when the approximate cubature has been compressed to three or four times its minimal size, in order to give the linear program more degrees of freedom.

\medskip\textbf{Step 4: Dyadic Construction}

\par\noindent\textbf{Outcome}: Exact \emph{dyadic} cubature.

Given an exact cubature formula made up of paths $(\omega_i)_{i=1}^M$ and probability weights $(\lambda_i)_{i=1}^M$, we can construct the product of the cubature with itself: first we rescale each of the paths $\omega_i \colon [0,1]\to \R^d$ to a path $\tilde{\omega}_i\colon [0,0.5] \to \R^d$ as follows: \[
\tilde{\omega}_i(t) = \frac{ \omega_i(2t)}{\sqrt{2}}.
\]
Then the product cubature is made up of $M^2$ paths $(\omega^\text{prod}_{i,j})_{i,j=1}^M$, where $\omega^\text{prod}_{i,j}\colon [0,1]\to \R^d$ is the concatenation of $\tilde{\omega}_i$ and $\tilde{\omega}_j$, with associated weight $\lambda_{i,j}^\text{prod} = \lambda_i \lambda_j$. This cubature has the property that it matches the expected Brownian signature on $[0,0.5]$ and $[0.5,1]$ in addition to $[0,1]$. We then apply recombination to compress this cubature to its minimal size.

This process can be repeated iteratively, to obtain cubatures that match the expected Brownian signature on all intervals of the form $[k2^{-n_d},(k+1)2^{-n_d}]$ for some $n_d\ge 0$, which we call the \emph{dyadic depth} of the cubature (an ordinary cubature as produced by Step 3 has dyadic depth zero). We have a theoretical result (Theorem~\ref{thm:maindyadic} below) indicating that scaling the dyadic depth of a cubature alongside its degree is guaranteed to yield an asymptotically consistent SDE estimator.

\subsection{Parameter Ablation}
The two free parameters of our algorithm are the number of steps $N$, and the seed for the randomisation in Step~1. We have always used the minimum number of steps $N$ for which the sharpening step succeeds, but it would be reasonable to ask if increasing $N$ further has an impact on performance. In Fig.~\ref{fig:ablation}, we show our error metrics on the CIR, IGBM, and Vasicek models with different cubature formulae; for each degree from $5$ to $11$, we run the construction with several different values of $N$ as well as several randomisation seeds. As Fig.~\ref{fig:ablation} illustrates, increasing $N$ beyond the minimum value does not seem to impact performance positively or negatively. But it does significantly increase the construction cost, justifying our approach of always using the minimum required value of $N$ for successful sharpening.

\begin{figure}
    \centering
    \includegraphics[width=\linewidth]{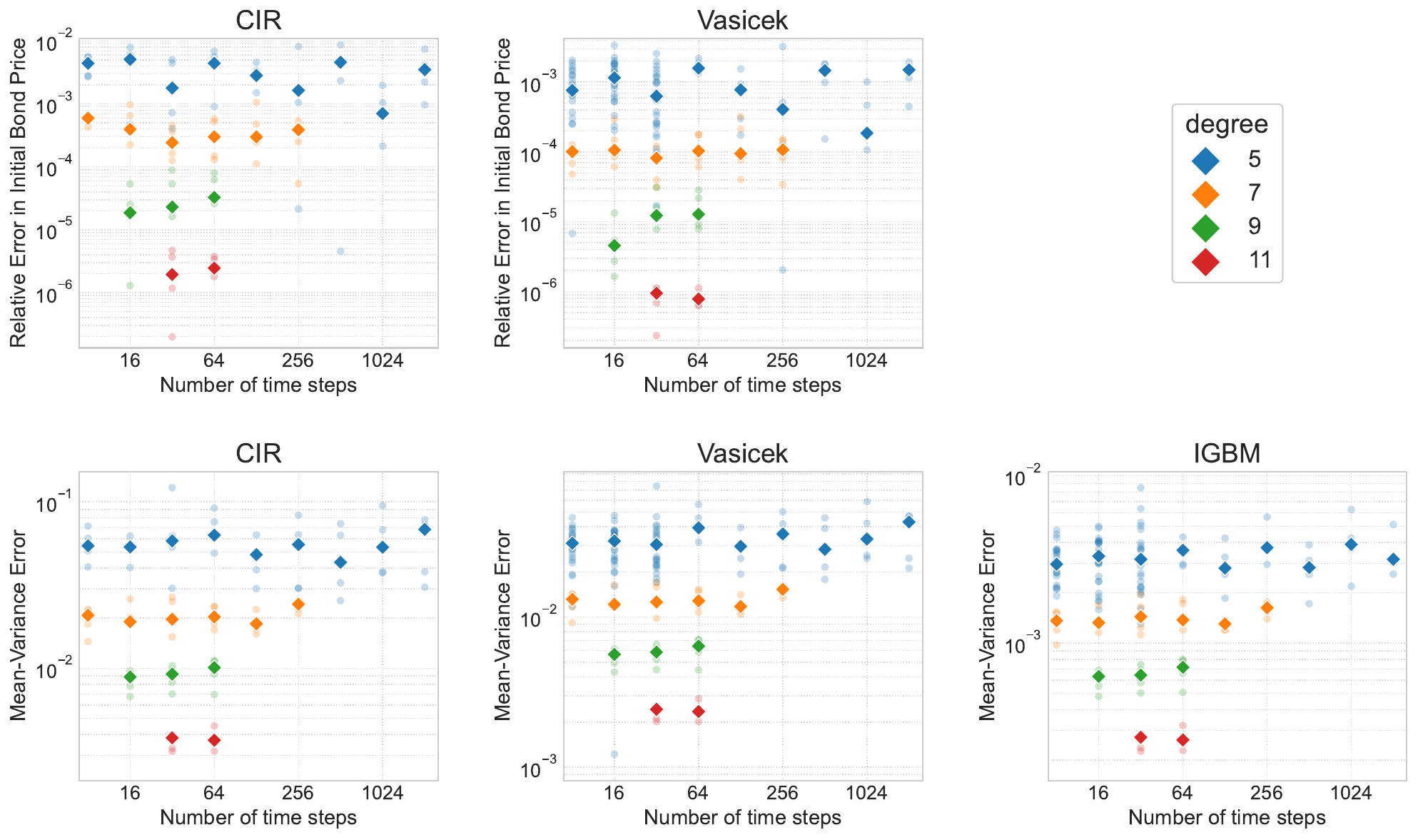}
    \caption{Error plots for cubature formulae of degrees $5$, $7$, $9$, and $11$, shown as a function of the number of steps $N$ used in Step~1 of the construction (with multiple randomisation seeds each). Increasing $N$ beyond the minimum required for successful sharpening has no discernible effect on performance.}
    \label{fig:ablation}
\end{figure}

\subsection{Theoretical results}\label{sec:theory}
In this section, we want to establish conditions under which our cubatures are guaranteed to exhibit asymptotic consistency in the following sense:

\DeclareGoal{C}{%
    If $y^n$ is the solution to the SDE~\eqref{eq:sde} driven by an \short cubature $(\lambda_i^n,\omega_i^n)$, and $y$ is the solution to the same SDE driven by Brownian motion, then $y^n \to y$ in distribution in the space of continuous paths.%
}

We conjecture that \Goal{C} holds as long as $D_n \to \infty$, with no assumptions on the dyadic depth. We fail to prove this (for the reader familiar with the language, we can show uniqueness of the subsequential limit, but not tightness),
but what we can prove is that \Goal{C} holds if, in addition to $D_n \to \infty$, either the dyadic depth goes to infinity along with the degree (Theorem~\ref{thm:maindyadic}), or we skip the recombination Step (Theorem~\ref{thm:mainOA}).

\begin{theorem}\label{thm:maindyadic}
    Suppose $(\lambda_i^n, \omega_i^n)$ is an \short cubature with degree $D_n \to \infty$, and dyadic depth $m_n \in \N$ such that \[
    m_n \ge \Big(\frac 34 + \varepsilon\Big) \log_2 N_n - C
    \] for some $\varepsilon, C > 0$.
    Then \Goal{C} holds.
\end{theorem}
\begin{remark}
    For the assertion of Theorem~\ref{thm:maindyadic} to hold, it is in fact sufficient if the cubature matches the expected Brownian signature up to degree $D_n \to \infty$ just on $[0,1]$, and only to degree $D=7$ on all of the finer dyadic intervals.
\end{remark}
\begin{theorem}\label{thm:mainOA}
    Let $(\lambda_i^n,\omega_i^n)$ be the approximate cubature resulting from the first step of the \short algorithm, prior to recombination, sharpening, and any dyadic constructions, with $N_n \to \infty$ and degree $D_n \to \infty$. Then \Goal{C} holds.
\end{theorem}

For the reader familiar with rough path theory, we remark that the essence of the proofs is showing that the cubatures converge in distribution to canonical Brownian motion \emph{in the space of rough paths}, in which case \Goal{C} follows from the universal limit theorem \cite[Section 3.5]{terrystflour} (convergence in the space of ordinary paths would not imply \Goal{C}).

\medskip
\begin{remark}
    Work closely related to these types of theorems has been done on \emph{Donsker-type theorems on rough path space}. Specifically, Breuillard, Friz, and Huesmann~\cite{donskerRW} give conditions under which a random walk with i.i.d.\ increments converges to Brownian motion in rough path space, and Bayer and Friz~\cite{donskercubature} prove convergence of product cubature formulae to Brownian motion in rough path space. The latter would imply (a stronger version of) our Theorem~\ref{thm:maindyadic} \emph{if} we didn't do the (essential) recombination step after the product construction. Nevertheless, the types of arguments used in the proof are very similar.
\end{remark}

Some of the facts established in the proofs may be of independent interest, such as Lemma~\ref{lem:roughpaths}, which gives conditions under which pointwise convergence of expected signatures of random rough paths implies convergence in distribution and vice versa.

The proofs of Theorems~\ref{thm:maindyadic} and~\ref{thm:mainOA} are relatively involved and in Appendix~\ref{app:theoretical results}.

\section{Conclusion}

In this paper, we have developed a computational pipeline called ``\short'' for constructing new state-of-the-art cubature formulae for numerically simulating SDEs.

Using the \short algorithm, we are able to construct high-degree cubature formulae with thousands of paths that empirically achieve orders of magnitude more estimation accuracy compared to standard Monte Carlo simulation across several real-world SDEs from mathematical finance and population genetics.

In addition, we showed that our cubature formulae are theoretically principled in the sense that, as their degree increases (and their degree on dyadic subintervals is $\geq 7$), their distribution will converge to that of Brownian motion in a rough path sense.
Importantly, and in contrast to past convergence results for SDE cubature formulae, this does not restrict the time interval on which the cubature formulae can be applied.

To help facilitate applications, we have made our implementation of the \short algorithm publicly available---along with datasets containing the cubature formulae used in our experiments. These can be found at \href{https://github.com/tttc3/ARCANE-Cubature}{github.com/tttc3/ARCANE-Cubature}.

\section{Future work}

\paragraph{Neural SDEs} Whilst SDEs have a wide range of uses, we expect that our state-of-the-art ARCANE cubature formulae will be more impactful in computationally demanding applications.
However, many such applications (e.g.~Langevin Markov Chain Monte Carlo \cite{scott2025langevin}) require simulating a high-dimensional Brownian motion---which is currently out of scope for our ARCANE cubature methodology. Instead, we believe that a more suitable application would be ``Neural'' SDEs \cite{li2020NSDEs, kidger2021neural, kidger2021efficient} where only a low-dimensional Brownian motion is needed, but performing backpropagation through the steps of the SDE solver (for each sample path) can result in a significant computational cost during training.
As an example, in \cite{chen2025wealth}, the following SDE model for wealth distribution is considered:
\begin{align}\label{eq:wealth_model}
\begin{split}
\diff w_t & = (r_t w_t - c_\theta(t, w_t\m, \eta_t) + \eta_t) \diff t,\\[2pt]
\diff\eta_t & = \mu(t, \eta_t) \diff t + \sigma(t, \eta_t) \diff W_t\m,
\end{split}
\end{align}
where $w_t$ and $\eta_t$ denote the wealth and income of a household at time $t$, $r_t$ is the interest rate, $c_\theta$ is the household's consumption (depending on $(t, w_t, \eta_t)$ and parameters $\theta$) and $\mu, \sigma : [0,T]\times\R\rightarrow\R$ are functions governing the dynamics of the household income.

Since it is challenging to prescribe how a household determines its level of consumption, the authors of \cite{chen2025wealth} make the assumption that a household's consumption policy $c_\theta$ is chosen to maximise a utility $u_1(c_\theta)$ whilst leaving some wealth $w_T$ at terminal time $T$.
More precisely, they model $c_\theta$ using a neural network---which they train to maximise  
\begin{align*}
\max_{\theta} \bigg(\E\bigg[\int_0^T e^{-\delta t}u_1(c_\theta(t, w_t, \eta_t)) \diff t + \lambda u_2(w_T)\bigg]\bigg),
\end{align*}
where $\delta, \lambda > 0$ and $u_1\m, u_2$ are constant-relative-risk-aversion (CRRA) utility functions.

Therefore, \eqref{eq:wealth_model} can be seen as a real-world example of a Neural SDE within economics. Since backpropagating through the dynamics of a (Neural) SDE is computationally expensive \cite{kidger2021neural}, we expect that reducing the number of required simulations through ARCANE cubature formulae to be particularly beneficial for such policy-based models. Similarly, we expect that ARCANE cubature can reduce the number of simulations required for the simulation-based inference of diffusions \cite{jovanovski2025SBI}.

\paragraph{Higher dimensions} On the methodological side, we believe that the main area of future work would be to extend ARCANE cubatures to higher dimensions. In particular, it would be interesting to see if taking existing ARCANE cubature formulae (e.g.~with high degree and/or dyadic depth) and `reshaping' them into new multidimensional cubatures is an effective strategy. This could potentially avoid the ``curse of dimensionality'' inherent in the signatures of high-dimensional paths. In any case, our ARCANE methodology can be scaled further to produce cubature formulae with higher degrees and for SDEs with higher dimensional noise, requiring more computational power alongside higher precision floating point operations (particularly as the terms in signatures of paths with finite length decay factorially fast with depth \cite[Proposition 2.2]{terrystflour}).

\paragraph{Efficient features} Currently, we compute the (truncated) signatures of paths when constructing ARCANE cubature formulae. However, as path signatures contain various algebraic redundancies, we believe that the number of ‘signature’ features used in our pipeline can be significantly reduced. For example, we expect that this feature reduction could be achieved by using log-signatures or by ignoring certain iterated integrals in the signature. In any case, such a feature reduction would particularly speed up the recombination step---which scales cubically with the number of features.

\paragraph{Long time horizons}
The reduction in accuracy when simulating SDEs over longer time horizons with cubatures is well-known. One possible solution is to cut the time interval into some number $N$ of equally sized sub-intervals, and apply the cubature formula on each of them.
However, if the cubature formula has $M$ paths, then applying it $N$ times would result in $M^N$ paths. To avoid this exponential blow-up, \cite{litterer2012cubature} proposed to use the recombination algorithm on the remaining paths after each sub-interval to reduce the number of particles that represent the SDE solution.
The challenge in this setting is to then identify a suitable collection of test functions so that the reduced collection of particles obtained by recombination remains accurate. We refer the reader to \cite{ninomiya2025cubature} for a recent article on applying cubature formulae to SDE simulation in this manner. In \cite{ninomiya2025cubature}, the authors apply (high-order) recombination along with space partitioning to reduce the total number of particles used by their algorithm.
A separate issue of reducing the number of particles during the SDE simulation is that it may be incompatible with automatic differentiation (e.g.~for Neural SDEs \cite{li2020NSDEs, kidger2021neural, kidger2021efficient}).
Nevertheless, applying our new cubature formulae in this iterative fashion is left as a topic of future work.

\section{Acknowledgements}

PK acknowledges support from EPSRC grant EP/W523781/1 and from the Department of Statistics at the University of Oxford.
TC acknowledges the support of the Department of Aeronautical and Automotive Engineering at Loughborough University.
JF acknowledges the support of the Department of Mathematical Sciences at the University of Bath.

\printbibliography

@article{albrecher2007little,
  title     = {The Little Heston Trap},
  author    = {Albrecher, Hansj{\"o}rg and Mayer, Philipp and Schoutens, Wim and Tistaert, Jurgen},
  journal   = {Wilmott Magazine},
  volume    = {2007},
  number    = {1},
  pages     = {83--92},
  year      = {2007},
  month     = {01},
  publisher = {Wiley}
}

@article{alfonsi2005cir,
  author  = {Alfonsi, Aur\'{e}lien},
  title   = {{On the discretization schemes for the CIR (and Bessel squared) processes}},
  journal = {{Monte Carlo Methods and Applications}},
  pages   = {355--384},
  volume  = {11},
  number  = {4},
  year    = {2005}
}

@article{alfonsi2010cir,
  author  = {Alfonsi, Aur\'{e}lien},
  title   = {{High order discretization schemes for the CIR process:
             Application to Affine Term Structure and Heston models}},
  journal = {{Mathematics of Computation}},
  volume  = {79},
  number  = {269},
  pages   = {209--237},
  year    = {2010}
}

@article{alfonsi2013cir,
  author  = {Alfonsi, Aur\'{e}lien},
  title   = {{Strong order one convergence of a drift implicit Euler scheme: Application to the CIR process}},
  journal = {{Statistics \& Probability Letters}},
  pages   = {602--607},
  volume  = {83},
  number  = {2},
  year    = {2013}
}

@article{arenaslopez2020wind,
  title   = {{A Fokker–Planck equation based approach for modelling wind speed and its power output}},
  author  = {Arenas-L\'{o}pez, J. Pablo and Badaoui, Mohamed},
  journal = {{Energy Conversion and Management}},
  volume  = {222},
  year    = {2020}
}

@inproceedings{badosa2018solar,
  author    = {Badosa, Jordi and Gobet, Emmanuel and Grangereau, Maxime and Kim, Daeyoung},
  title     = {{Day-Ahead Probabilistic Forecast of Solar Irradiance: A Stochastic Differential Equation Approach}},
  booktitle = {{Renewable Energy: Forecasting and Risk Management}},
  pages     = {73--93},
  year      = {2018}
}

@software{baryx,
  author  = {Thomas Coxon},
  title   = {{Baryx: Barycenter preserving measure reduction in JAX}},
  url     = {https://github.com/datasig-ac-uk/baryx},
  version = {0.1.0},
  year    = {2026}
}

@incollection{BCH,
  title     = {{9 BCH codes}},
  editor    = {F.~J. MacWilliams and N.~J.~A. Sloane},
  series    = {North-Holland Mathematical Library},
  publisher = {Elsevier},
  volume    = {16},
  pages     = {257-293},
  year      = {1977},
  booktitle = {The Theory of Error-Correcting Codes},
  issn      = {0924-6509},
  doi       = {10.1016/S0924-6509(08)70534-8}
}

@article{boedihardjo2016,
  title   = {{The signature of a rough path: Uniqueness}},
  journal = {Advances in Mathematics},
  volume  = {293},
  pages   = {720-737},
  year    = {2016},
  issn    = {0001-8708},
  url     = {https://www.sciencedirect.com/science/article/pii/S0001870816301104},
  author  = {Horatio Boedihardjo and Xi Geng and Terry Lyons and Danyu Yang}
}

@article{bose-ray1960,
  title   = {On a class of error correcting binary group codes},
  journal = {Information and Control},
  volume  = {3},
  number  = {1},
  pages   = {68-79},
  year    = {1960},
  issn    = {0019-9958},
  url     = {https://www.sciencedirect.com/science/article/pii/S0019995860902874},
  author  = {R.~C. Bose and D.~K. Ray-Chaudhuri}
}

@article{cai2015sirs,
  title   = {{A stochastic SIRS epidemic model with infectious force under intervention strategies}},
  author  = {Yongli Cai and Yun Kang and Malay Banerjee and Weiming Wang},
  journal = {Journal of Differential Equations},
  volume  = {259},
  number  = {12},
  pages   = {7463--7502},
  year    = {2015}
}

@article{capriotti2018igbm,
  author  = {Capriotti, Luca and Jiang, Yupeng and Shaimerdenova, Gaukhar},
  title   = {{Approximation Methods for Inhomogeneous Geometric Brownian Motion}},
  journal = {International Journal of Theoretical and Applied Finance},
  volume  = {22},
  number  = {2},
  year    = {2019}
}

@article{chaabane2025solar,
  author  = {Chaabane, Khaoula Ben and Kebaier, Ahmed and Scavino, Marco and Tempone, Ra\'{u}l},
  title   = {{Data-driven uncertainty quantification for constrained stochastic differential equations and application to solar photovoltaic power forecast data}},
  journal = {Statistics and Computing},
  volume  = {35},
  number  = {163},
  year    = {2025}
}

@misc{cheltsov2025sampling,
  title         = {{Hadamard Langevin dynamics for sampling the l1-prior}},
  author        = {Cheltsov, Ivan and Cornalba, Federico and Poon, Clarice and Shardlow, Tony},
  year          = {2025},
  eprint        = {2411.11403},
  archiveprefix = {arXiv},
  primaryclass  = {math.NA}
}

@misc{chen2025wealth,
  title         = {{Stochastic Analysis of Overlapping Generations Models Under Incomplete Markets}},
  author        = {Chen, Cangxiong and Ellingsrud, Sigmund and Harang, Fabian and Irarrazabal, Alfonso and Mayorcas, Avi},
  year          = {2025},
  eprint        = {2509.05170},
  archiveprefix = {arXiv},
  primaryclass  = {math.PR}
}

@article{chevyrev2016,
  title     = {Characteristic functions of measures on geometric rough paths},
  author    = {Chevyrev, Ilya and Lyons, Terry},
  year      = {2016},
  month     = nov,
  journal   = {The Annals of Probability},
  volume    = {44},
  number    = {6},
  publisher = {Institute of Mathematical Statistics},
  issn      = {0091-1798}
}

@article{chevyrev2022signature,
  author  = {Ilya Chevyrev and Harald Oberhauser},
  title   = {{Signature Moments to Characterize Laws of Stochastic Processes}},
  journal = {Journal of Machine Learning Research},
  volume  = {23},
  number  = {176},
  pages   = {1--42},
  year    = {2022},
  url     = {https://jmlr.org/papers/v23/20-1466.html}
}

@book{cobb1981socialsdes,
  title     = {{Mathematical Frontiers Of The Social And Policy Sciences}},
  author    = {Cobb, Loren},
  publisher = {Routledge},
  year      = {1981},
  address   = {New York}
}

@article{cox1985cir,
  author  = {Cox, John C. and Ingersoll, Jonathan E. and Ross, Stephen A.},
  title   = {{A Theory of the Term Structure of Interest Rates}},
  journal = {{Econometrica}},
  volume  = {53},
  number  = {2},
  pages   = {385--407},
  year    = {1985}
}

@article{cozma2020cir,
  author  = {Cozma, Andrei and Reisinger, Christoph},
  title   = {{Strong order 1/2 convergence of full truncation Euler approximations to the Cox–Ingersoll–Ross process}},
  journal = {{IMA Journal of Numerical Analysis}},
  pages   = {358--376},
  volume  = {40},
  number  = {1},
  year    = {2020}
}

@misc{crisostomo2015heston,
  title         = {{An Analysis of the Heston Stochastic Volatility Model: Implementation and Calibration using Matlab}},
  author        = {Crisostomo, Ricardo},
  year          = {2015},
  eprint        = {1502.02963},
  archiveprefix = {arXiv},
  primaryclass  = {q-fin.PR}
}

@article{czuppon2021genetics,
  author  = {Czuppon, Peter and Traulsen, Arne},
  title   = {{Understanding evolutionary and ecological dynamics using a continuum limit}},
  journal = {{Ecology and Evolution}},
  volume  = {11},
  number  = {11},
  pages   = {5857--5873},
  year    = {2021}
}

@article{delsarte-goethals,
  title   = {{Alternating bilinear forms over GF(q)}},
  author  = {P. Delsarte and J.~M. Goethals},
  journal = {Journal of Combinatorial Theory, Series A},
  volume  = {19},
  number  = {1},
  pages   = {26-50},
  year    = {1975}
}

@article{delsarte1973algebraic,
  title   = {An algebraic approach to the association schemes of coding theory},
  author  = {Delsarte, Philippe},
  journal = {Philips Res. Rep. Suppl.},
  volume  = {10},
  year    = {1973}
}

@article{dereich2011cir,
  author  = {Dereich, Steffen and Neuenkirch, Andreas and Szpruch, Lukasz},
  title   = {{An Euler-type method for the strong approximation of the Cox–Ingersoll–Ross process}},
  journal = {{Proceedings of the Royal Society A: Mathematical, Physical and Engineering Sciences}},
  volume  = {468},
  number  = {2140},
  pages   = {1105--1115},
  year    = {2011}
}

@article{donskercubature,
  title     = {Cubature on Wiener space: pathwise convergence},
  author    = {Bayer, Christian and Friz, Peter K},
  journal   = {Applied Mathematics \& Optimization},
  volume    = {67},
  number    = {2},
  pages     = {261--278},
  year      = {2013},
  publisher = {Springer}
}

@article{donskerRW,
  title   = {From random walks to rough paths},
  author  = {Breuillard, Emmanuel and Friz, Peter and Huesmann, Martin},
  journal = {Proceedings of the American Mathematical Society},
  volume  = {137},
  number  = {10},
  pages   = {3487--3496},
  year    = {2009}
}

@book{etheridge2011genetics,
  title     = {Some Mathematical Models from Population Genetics},
  author    = {Etheridge, Alison},
  year      = {2011},
  publisher = {Springer},
  address   = {Berlin, Heidelberg},
  series    = {Lecture Notes in Mathematics},
  volume    = {2012},
  doi       = {10.1007/978-3-642-16632-7}
}

@book{ewens2004genetics,
  title     = {Mathematical Population Genetics: I. Theoretical Introduction},
  author    = {Ewens, Warren J.},
  year      = {2004},
  edition   = {2nd},
  publisher = {Springer},
  address   = {New York},
  series    = {Interdisciplinary Applied Mathematics},
  volume    = {27},
  doi       = {10.1007/978-0-387-21822-9}
}

@article{feller1951diffusion,
  author  = {Feller, William},
  title   = {Two Singular Diffusion Problems},
  journal = {Annals of Mathematics},
  volume  = {54},
  number  = {1},
  pages   = {173--182},
  year    = {1951},
  doi     = {10.2307/1969318}
}

@article{ferrucci2026cubature,
  author  = {Ferrucci, Emilio and Herschell, Timothy and Litterer, Christian and Lyons, Terry},
  title   = {{High-degree cubature on Wiener space through unshuffle expansions}},
  journal = {Proceedings of the Royal Society A: Mathematical, Physical and Engineering Sciences},
  volume  = {482},
  number  = {2330},
  year    = {2026}
}

@misc{foster2024adaptive,
  title         = {{On the convergence of adaptive approximations for stochastic differential equations}},
  author        = {James Foster and Andra\v{z} Jelin\v{c}i\v{c}},
  year          = {2024},
  eprint        = {2311.14201},
  archiveprefix = {arXiv},
  primaryclass  = {math.NA}
}

@article{foster2024highSIAM,
  author  = {Foster, James M. and dos Reis, Gon\c{c}alo and Strange, Calum},
  title   = {{High Order Splitting Methods for SDEs Satisfying a Commutativity Condition}},
  journal = {SIAM Journal on Numerical Analysis},
  volume  = {62},
  number  = {1},
  pages   = {500--532},
  year    = {2024}
}

@inbook{gyurko2011cubature,
  author    = {Gyurk{\'o}, Lajos Gergely and Lyons, Terry},
  title     = {{Efficient and Practical Implementations of Cubature on Wiener Space}},
  booktitle = {Stochastic Analysis 2010},
  year      = {2011},
  publisher = {Springer},
  address   = {Berlin Heidelberg},
  pages     = {73--111}
}

@article{halton,
  author  = {Halton, J. H.},
  title   = {On the efficiency of certain quasi-random sequences of points in evaluating multi-dimensional integrals},
  journal = {Numerische Mathematik},
  volume  = {2},
  pages   = {84--90},
  year    = {1960}
}

@article{hammons94,
  title   = {{The Z/sub 4/-linearity of Kerdock, Preparata, Goethals, and related codes}},
  author  = {Hammons, A.~R. and Kumar, P.~V. and Calderbank, A.~R. and Sloane, N.~J.~A. and Sole, P.},
  journal = {IEEE Transactions on Information Theory},
  year    = {1994},
  volume  = {40},
  number  = {2},
  pages   = {301-319}
}

@article{hayakawa2022cubature,
  author  = {Hayakawa, Satoshi and Tanaka, Ken'ichiro},
  title   = {{Monte Carlo construction of cubature on Wiener space}},
  journal = {Japan Journal of Industrial and Applied Mathematics},
  volume  = {39},
  number  = {2},
  pages   = {543--571},
  year    = {2022}
}

@book{hedayat1999,
  title     = {Orthogonal Arrays: Theory and Applications},
  author    = {Hedayat, A.~S. and Sloane, N.~J.~A. and Stufken, John},
  year      = {1999},
  publisher = {Springer},
  series    = {Springer Series in Statistics},
  address   = {New York}
}

@book{hedayat2012,
  title     = {Orthogonal arrays: theory and applications},
  author    = {Hedayat, A Samad and Sloane, Neil James Alexander and Stufken, John},
  year      = {2012},
  publisher = {Springer Science \& Business Media}
}

@article{hefter2019cir,
  author  = {Hefter, Mario and Jentzen, Arnulf},
  title   = {{On arbitrarily slow convergence rates for strong numerical approximations of Cox-Ingersoll-Ross processes and squared Bessel processes}},
  journal = {{Finance and Stochastics}},
  volume  = {23},
  pages   = {139-172},
  year    = {2019}
}

@article{heston1993,
  author  = {Heston, Steven L.},
  title   = {{A Closed-Form Solution for Options with Stochastic Volatility with Applications to Bond and Currency Options}},
  journal = {{The Review of Financial Studies}},
  pages   = {327--343},
  volume  = {6},
  number  = {2},
  year    = {1993}
}

@article{hill73,
  author  = {Hill, Raymond},
  title   = {{On the largest size of cap in $S_{53}$}},
  journal = {Rendiconti del Seminario Matematico della Universit\`{a} di Padova},
  volume  = {54},
  pages   = {378--380},
  year    = {1973},
  url     = {https://www.bdim.eu/item?id=RLINA_1973_8_54_3_378_0}
}

@incollection{hill83,
  title     = {{On Pellegrino's 20-Caps in $S_{43}$}},
  author    = {Hill, Raymond},
  series    = {North-Holland Mathematics Studies},
  publisher = {North-Holland},
  volume    = {78},
  pages     = {433-447},
  year      = {1983},
  booktitle = {Combinatorics '81 in honour of Beniamino Segre},
  address   = {Amsterdam}
}

@article{huang2019,
  title     = {Sidon sets and 2-caps in $\mathbb{F}_3^n$},
  author    = {Huang, Yixuan and Tait, Michael and Won, Robert},
  volume    = {12},
  issn      = {1944--4176},
  number    = {6},
  journal   = {Involve, a Journal of Mathematics},
  publisher = {Mathematical Sciences Publishers},
  year      = {2019},
  pages     = {995--1003}
}

@book{huffman2003,
  title     = {Fundamentals of Error-Correcting Codes},
  author    = {Huffman, W. Cary and Pless, Vera},
  year      = {2003},
  publisher = {Cambridge University Press},
  address   = {Cambridge}
}

@article{iversen2014solar,
  author  = {Iversen, Emil B. and Morales, Juan M. and M{\o}ller, Jan K. and Madsen, Henrik},
  title   = {{Probabilistic forecasts of solar irradiance using stochastic differential equations}},
  journal = {Environmetrics},
  volume  = {25},
  number  = {3},
  pages   = {152--164},
  year    = {2014}
}

@software{jax2018github,
  author  = {James Bradbury and Roy Frostig and Peter Hawkins and Matthew James Johnson and Chris Leary and Dougal Maclaurin and George Necula and Adam Paszke and Jake Vander{P}las and Skye Wanderman-{M}ilne and Qiao Zhang},
  title   = {{JAX}: composable transformations of {P}ython+{N}um{P}y programs},
  url     = {https://github.com/jax-ml/jax},
  version = {0.3.13},
  year    = {2018}
}

@misc{jovanovski2025SBI,
  title         = {{Simulation-based inference using splitting schemes for partially observed diffusions in chemical reaction networks}},
  author        = {Jovanovski, Petar and Golightly, Andrew and Picchini, Umberto and Tamborrino, Massimiliano},
  year          = {2025},
  eprint        = {2508.11438},
  archiveprefix = {arXiv},
  primaryclass  = {stat.ME}
}

@article{keefe96,
  title   = {{Ovoids in PG(3, q): a survey}},
  author  = {Christine M. O'Keefe},
  journal = {Discrete Mathematics},
  volume  = {151},
  number  = {1},
  pages   = {175--188},
  year    = {1996},
  issn    = {0012-365X}
}

@article{kelly2022cir,
  author  = {Kelly, C\'{o}nall and Lord, Gabriel and Maulana, Heru},
  title   = {{The role of adaptivity in a numerical method for the Cox–Ingersoll–Ross model}},
  journal = {{Journal of Computational and Applied Mathematics}},
  volume  = {410},
  year    = {2022}
}

@article{kerdock,
  title   = {A class of low-rate nonlinear binary codes},
  author  = {A.~M. Kerdock},
  journal = {Information and Control},
  volume  = {20},
  number  = {2},
  pages   = {182-187},
  year    = {1972},
  issn    = {0019-9958}
}

@inproceedings{kidger2021efficient,
  author    = {Kidger, Patrick and Foster, James and Li, Xuechen and Lyons, Terry},
  booktitle = {{Advances in Neural Information Processing Systems}},
  pages     = {18747--18761},
  title     = {{Efficient and Accurate Gradients for Neural SDEs}},
  volume    = {34},
  year      = {2021}
}

@article{kidger2021neural,
  title   = {{Neural SDEs as Infinite-Dimensional GANs}},
  author  = {Kidger, Patrick and Foster, James and Li, Xuechen and Lyons, Terry J},
  journal = {Proceedings of the 38th International Conference on Machine Learning},
  pages   = {5453--5463},
  year    = {2021}
}

@book{KP1992,
  author    = {Kloeden, P.E. and Platen, E.},
  publisher = {Springer-Verlag},
  series    = {Applications of Mathematics, Stochastic Modelling and Applied Probability, 23},
  title     = {Numerical solution of stochastic differential equations },
  year      = {1992}
}

@article{latin,
  author  = {McKay, M. D. and Beckman, R. J. and Conover, W. J.},
  title   = {A comparison of three methods for selecting values of input variables in the analysis of output from a computer code},
  journal = {Technometrics},
  volume  = {21},
  number  = {2},
  pages   = {239--245},
  year    = {1979}
}

@inproceedings{li2020NSDEs,
  title     = {{Scalable Gradients for Stochastic Differential Equations}},
  author    = {Li, Xuechen and Wong, Ting-Kam Leonard and Chen, Ricky T. Q. and Duvenaud, David},
  booktitle = {Proceedings of the 23rd International Conference on Artificial Intelligence and Statistics},
  pages     = {3870--3882},
  year      = {2020},
  volume    = {108}
}

@article{li2023aircraft,
  author  = {Haiquan Li and Xiaoqian Chen and Jiatu Zhang and Bochen Wang and Jiahui Peng and Liang Wang},
  title   = {{Identification of a Stochastic Dynamic Model for Aircraft Flight Attitude Based on Measured Data}},
  journal = {International Journal of Aerospace Engineering},
  volume  = {2023},
  number  = {1},
  year    = {2023}
}

@article{litterer2012cubature,
  author  = {Litterer, Christian and Lyons, Terry},
  title   = {{High order recombination and an application to cubature on Wiener space}},
  journal = {Annals of Applied Probability},
  year    = {2012},
  volume  = {22},
  number  = {4},
  pages   = {1301--1327}
}

@article{liu2021aircraft,
  author  = {Ying Liu and Zhiyao Zhao and Haibiao Ma and Quan Quan},
  title   = {A stochastic approximation method for probability prediction of docking success for aerial refueling},
  journal = {Applied Soft Computing},
  volume  = {103},
  year    = {2021}
}

@article{lord2010truncation,
  author    = {Roger Lord and Remmert Koekkoek and Dick Van Dijk},
  title     = {A comparison of biased simulation schemes for stochastic volatility models},
  journal   = {Quantitative Finance},
  volume    = {10},
  number    = {2},
  pages     = {177--194},
  year      = {2010},
  publisher = {Routledge}
}

@article{lyons2004cubature,
  author  = {Lyons, Terry and Victoir, Nicolas},
  title   = {{Cubature on Wiener Space}},
  journal = {Proceedings of the Royal Society A: Mathematical, Physical and Engineering Sciences},
  volume  = {460},
  number  = {2041},
  pages   = {169--198},
  year    = {2004}
}

@book{lyonsqian,
  title     = {System control and rough paths},
  author    = {Lyons, Terry and Qian, Zhongmin},
  year      = {2002},
  publisher = {Oxford University Press},
  address   = {Oxford}
}

@inproceedings{maki2013sir,
  author    = {Maki, Yoshihiro and Hirose, Hideo},
  booktitle = {2013 4th International Conference on Intelligent Systems, Modelling and Simulation},
  title     = {{Infectious Disease Spread Analysis Using Stochastic Differential Equations for SIR Model}},
  year      = {2013},
  pages     = {152--156}
}

@inproceedings{malyarenko2022cubature,
  author    = {Malyarenko, Anatoliy and Nohrouzian, Hossein},
  title     = {{Testing Cubature Formulae on Wiener Space Versus Explicit Pricing Formulae}},
  booktitle = {Stochastic Processes, Statistical Methods, and Engineering Mathematics},
  publisher = {Springer International Publishing},
  address   = {Cham},
  pages     = {223--248},
  year      = {2022}
}

@incollection{meser2016wrightfisher,
  title     = {{Neutral Models of Genetic Drift and Mutation}},
  editor    = {Richard M. Kliman},
  booktitle = {Encyclopedia of Evolutionary Biology},
  publisher = {Academic Press},
  address   = {Oxford},
  pages     = {119--123},
  year      = {2016},
  isbn      = {978-0-12-800426-5},
  author    = {P.~W. Messer}
}

@article{milstein2015cir,
  author  = {Milstein, Grigori N. and Schoenmakers, John},
  title   = {{Uniform approximation of the Cox-Ingersoll-Ross process}},
  journal = {{Advances in Applied Probability}},
  volume  = {47},
  number  = {4},
  pages   = {1132-1156},
  year    = {2015}
}

@book{niederreiterqmc,
  author    = {H. Niederreiter},
  title     = {Random Number Generation and Quasi-Monte Carlo Methods},
  series    = {CBMS–NSF Regional Conference Series in Applied Mathematics},
  volume    = {63},
  year      = {1992},
  publisher = {SIAM}
}

@misc{ninomiya2025cubature,
  title         = {{A high-order recombination algorithm for weak approximation of stochastic differential equations}},
  author        = {Ninomiya, Syoiti and Shinozaki, Yuji},
  year          = {2025},
  eprint        = {2504.19717},
  archiveprefix = {arXiv},
  primaryclass  = {math.PR}
}

@article{ninomiyavictoir2008,
  author  = {Ninomiya, Syoiti and Victoir, Nicolas},
  title   = {{Weak Approximation of Stochastic Differential Equations and Application to Derivative Pricing}},
  journal = {Applied Mathematical Finance},
  volume  = {15},
  number  = {2},
  pages   = {107--121},
  year    = {2008}
}

@misc{nohrouzian,
  title         = {Constructing Trinomial Models Based on Cubature Method on Wiener Space: Applications to Pricing Financial Derivatives},
  author        = {Hossein Nohrouzian and Anatoliy Malyarenko and Ying Ni},
  year          = {2022},
  eprint        = {2204.10692},
  archiveprefix = {arXiv},
  primaryclass  = {q-fin.MF}
}

@article{ornstein1930,
  title   = {{On the Theory of the Brownian Motion}},
  author  = {Uhlenbeck, G.~E. and Ornstein, L.~S.},
  journal = {Physical Review},
  volume  = {36},
  issue   = {5},
  pages   = {823--841},
  year    = {1930}
}

@software{orray,
  author  = {Peter Koepernik},
  title   = {{orray: Orthogonal Arrays in JAX}},
  url     = {https://github.com/peter-koepernik/orray},
  version = {0.1.0},
  year    = {2025}
}

@article{potechin2008,
  author  = {Potechin, Aaron},
  title   = {Maximal caps in {AG}(6,3)},
  journal = {Designs, Codes and Cryptography},
  year    = {2008},
  month   = {03},
  day     = {01},
  volume  = {46},
  number  = {3},
  pages   = {243--259},
  issn    = {1573-7586},
  doi     = {10.1007/s10623-007-9132-z}
}

@article{rao1947,
  issn      = {14666162},
  url       = {https://www.jstor.org/stable/2983576},
  author    = {C. Radhakrishna Rao},
  journal   = {Supplement to the Journal of the Royal Statistical Society},
  number    = {1},
  pages     = {128--139},
  publisher = {[Oxford University Press, Royal Statistical Society]},
  title     = {Factorial Experiments Derivable from Combinatorial Arrangements of Arrays},
  urldate   = {2025-02-18},
  volume    = {9},
  year      = {1947}
}

@article{rehman2023optics,
  title   = {{Analysis of Brownian motion in stochastic Schrödinger wave equation using Sardar sub-equation method}},
  author  = {Rehman, Hamood Ur and Akber,Rehan and  Wazwaz, Abdul-Majid and Alshehri, Hashim M. and Osman, M. S.},
  journal = {Optik},
  volume  = {289},
  year    = {2023}
}

@misc{scott2025langevin,
  title         = {{Underdamped Langevin MCMC with third order convergence}},
  author        = {Maximilian Scott and D\'{a}ire O'Kane and Andra\v{z} Jelin\v{c}i\v{c} and James Foster},
  year          = {2025},
  eprint        = {2508.16485},
  archiveprefix = {arXiv},
  primaryclass  = {stat.ML}
}

@article{shinozaki2017cubature,
  author  = {Shinozaki, Yuji},
  title   = {{Construction of a Third-Order K-Scheme and Its Application to Financial Models}},
  journal = {SIAM Journal on Financial Mathematics},
  volume  = {8},
  number  = {1},
  pages   = {901--932},
  year    = {2017}
}

@book{shortrate,
  author    = {Björk, Tomas},
  title     = {Arbitrage Theory in Continuous Time},
  publisher = {Oxford University Press},
  year      = {2019},
  month     = {12},
  isbn      = {9780198851615},
  doi       = {10.1093/oso/9780198851615.001.0001}
}

@article{sobol,
  author  = {Sobol', I. M.},
  title   = {On the distribution of points in a cube and the approximate evaluation of integrals},
  journal = {USSR Computational Mathematics and Mathematical Physics},
  volume  = {7},
  number  = {4},
  pages   = {86--112},
  year    = {1967}
}

@inproceedings{song2021scorebased,
  title     = {{Score-Based Generative Modeling through Stochastic Differential Equations}},
  author    = {Yang Song and Jascha Sohl-Dickstein and Diederik P Kingma and Abhishek Kumar and Stefano Ermon and Ben Poole},
  booktitle = {International Conference on Learning Representations},
  year      = {2021}
}

@article{stinson2008,
  author     = {Stinson, Douglas R.},
  title      = {Combinatorial designs: constructions and analysis},
  year       = {2008},
  issue_date = {December 2008},
  publisher  = {Association for Computing Machinery},
  address    = {New York, NY, USA},
  volume     = {39},
  number     = {4},
  issn       = {0163-5700},
  doi        = {10.1145/1466390.1466393},
  journal    = {SIGACT News},
  pages      = {17–21},
  numpages   = {5}
}

@phdthesis{tchernychova2015cubature,
  title  = {{Carath\'{e}odory cubature measures}},
  author = {Tchernychova, Maria},
  school = {University of Oxford},
  year   = {2015}
}

@book{terrystflour,
  title     = {Differential equations driven by rough paths},
  author    = {Lyons, Terry J and Caruana, Michael and L{\'e}vy, Thierry},
  year      = {2007},
  publisher = {Springer},
  address   = {Berlin}
}

@article{tubikanec2022igbm,
  author  = {Tubikanec, Irene and Tamborrino, Massimiliano and Lansky, Petr and Buckwar, Evelyn},
  title   = {{Qualitative properties of different numerical methods for the inhomogeneous geometric Brownian motion}},
  journal = {Journal of Computational and Applied Mathematics},
  volume  = {406},
  year    = {2022}
}

@article{vaisband2025nsdes,
  title   = {{Loss formulations for assumption-free neural inference of SDE coefficient functions}},
  author  = {Vaisband, Marc and von Bornhaupt, Valentin and Schmid, Nina and Abulizi, Izdar and Hasenauer, Jan},
  journal = {{Nature Partner Journal: Systems Biology and Applications}},
  volume  = {11},
  number  = {22},
  year    = {2025}
}

@article{vasicek1977,
  title   = {An equilibrium characterization of the term structure},
  journal = {Journal of Financial Economics},
  volume  = {5},
  number  = {2},
  pages   = {177-188},
  year    = {1977},
  issn    = {0304-405X},
  url     = {https://www.sciencedirect.com/science/article/pii/0304405X77900162},
  author  = {Oldrich Vasicek}
}

@article{victoirOA,
  author  = {Victoir, Nicolas},
  title   = {Asymmetric Cubature Formulae with Few Points in High Dimension for Symmetric Measures},
  journal = {SIAM Journal on Numerical Analysis},
  volume  = {42},
  number  = {1},
  pages   = {209-227},
  year    = {2004},
  doi     = {10.1137/S0036142902407952}
}

@article{wang2022sdes,
  title   = {{Data-Driven Discovery of Stochastic Differential Equations}},
  author  = {Wang, Yasen and Fang, Huazhen and Jin, Junyang and Ma, Guijun and He, Xin and Dai, Xing and Yue, Zuogong and Cheng, Cheng and Zhang, Hai-Tao and Pu, Donglin and Wu, Dongrui and Yuan, Ye and Gon\c{c}alves, Jorge and Kurths, J\"{u}rgen and Ding, Han},
  journal = {{Engineering}},
  volume  = {17},
  pages   = {244--252},
  year    = {2022}
}

@article{wangqmc,
  author    = {X. Wang},
  title     = {Low discrepancy sequences in high dimensions: How well do they integrate?},
  journal   = {Journal of Computational and Applied Mathematics},
  volume    = {215},
  number    = {2},
  pages     = {301--314},
  year      = {2008},
  publisher = {Elsevier}
}
\begin{appendices}
    \counterwithin{theorem}{section}
    \counterwithin{lemma}{section}
    \setcounter{theorem}{0}
    \setcounter{lemma}{0}
    \clearpage
\section{Example SDEs}\label{app:sdes}
This appendix provides the details for all the example SDEs used in our experiments.
In all the examples, we first consider the It\^{o} formulation of the SDE to obtain reference values. We then convert the SDE into Stratonovich form to perform the simulations.

Throughout, we will use the following shorthand notation:
\begin{itemize}
\item We use $\E_{x_0}[X_t]$ for $\E[X_t \mid X_0 = x_0]$, the expectation of $X_t$ conditional on $X_0 = x_0$.\medbreak
\item We use $\V_{x_0}[X_t]$ for $\var(X_t \mid X_0 = x_0)$, the variance of $X_t$ conditional on $X_0 = x_0$.\medbreak
\end{itemize}

\subsection{Ornstein--Uhlenbeck (Vasicek) process}\label{app:ou}
The Ornstein--Uhlenbeck (OU) or Vasicek process is a Gaussian process 
originally developed as a model for the mean values of the velocity and displacement of a free particle in Brownian motion \cite{ornstein1930}.
Beyond the usual applications in the natural sciences, the OU process has been applied in mathematical finance as a mean-reverting one-factor short-rate model (i.e.~a model in which the interest rate is driven by a single source of randomness).
For derivative pricing, this is also known as the Vasicek model \cite{vasicek1977}.
More recently, multidimensional OU processes have seen applications in data science as the forward noising processes in SDE-based generative diffusion models \cite{song2021scorebased}.

\subsubsection*{Definition}
Let $a > 0$, $b \in \R$, and $\sigma > 0$. Then the OU process with mean reversion rate $a$, long-term mean $b$, and volatility $\sigma$ is
\begin{equation}\label{eq:ou_ito}
    \diff X_t = a\,(b - X_t) \diff t + \sigma \diff W_t,
\end{equation}
where $t > 0$ and $W_t$ is a standard Brownian motion.
In Stratonovich form, it becomes
\begin{equation}\label{eq:ou_strat}
    \diff X_t = a_\text{strat}\,(b_\text{strat} - X_t) \diff t + \sigma \circ \diff W_t,
\end{equation}
where
\begin{equation*}
a_{\text{strat}} \coloneqq a,\qquad
  b_{\text{strat}} \coloneqq b,
\end{equation*}
which coincides with the It\^{o} form since the diffusion coefficient is constant.

\subsubsection*{Default parameters}
Unless otherwise specified, our results for the OU process use the parameters:
\[X_{0}=2.0,\ a=1.0,\ b=3.0,\ \sigma=0.5,\ T = 1.0.\]

\subsubsection*{Reference values}
All errors quoted in our results for the OU process are with respect to the following theoretically derived reference values:

\begin{enumerate}
    \item Conditional mean \cite{ornstein1930}: \begin{equation}
       \E_{x_0}[X_t] = b + (x_0 - b)\exp(-a t).
    \end{equation}
    \item Conditional variance \cite{ornstein1930}: \begin{equation}
        \V_{x_0}[X_t] = \frac{\sigma^2}{2a} \bigl( 1 - \exp(-2a t)\bigr).
    \end{equation}
    \item Conditional distribution \cite{ornstein1930}:
    \begin{equation}
        (X_t \mid X_0) \sim \NN\bigl(\E_{x_0}[X_t], \V_{x_0}[X_t]\bigr).
    \end{equation}
    \item Price at time $s$ of a zero-coupon bond maturing at time $T$ \cite{vasicek1977}: \begin{equation}
        \E\Biggl[\exp\biggl(-\int_s^T X_t \diff t\biggr)\biggm| X_s\Biggr] = A(s, T) \exp\bigl(-X_s B(s, T)\bigr),
    \end{equation}
    where
    \begin{align*}
        A(s, T) &= \exp\biggl(\gamma\, \bigl(B(s, T) - (T - s)\bigr) - \frac{\sigma^2}{4a} B(s, T)^2\biggr),\\[3pt]
        B(s, T) &= \frac{1 - \exp\bigl(-a\,(T - s)\bigr)}{a},\qquad
        \gamma = b - \frac{\sigma^2}{2a^2}.
    \end{align*}
\end{enumerate}

\subsubsection*{Numerical considerations}
We do not take any special numerical considerations when simulating the OU process.

\subsection{Inhomogeneous Geometric Brownian Motion}\label{app:igbm}
Inhomogeneous geometric Brownian motion (IGBM) is another example of a one-factor short-rate model used in mathematical finance \cite{capriotti2018igbm}.
However, unlike the OU process, IGBM has multiplicative noise and remains non-negative for all $t\geq 0$.
It is therefore suitable for modelling interest rates, stochastic volatilities and default intensities.

From a theoretical viewpoint, IGBM is also one of the simplest SDEs that has no known method of exact simulation.
As a consequence, it has been used as a test problem to study the properties of numerical methods for Monte Carlo simulation \cite{tubikanec2022igbm}.

\subsubsection*{Definition}
Let $a > 0$, $b > 0$, and $\sigma > 0$. The inhomogeneous geometric Brownian motion (IGBM) with mean reversion rate $a$, long-term mean $b$, and volatility $\sigma$ is given by
\begin{equation}\label{eq:igbm_ito}
    \diff X_t = a\,(b - X_t) \diff t + \sigma X_t \diff W_t,
\end{equation}
where $t > 0$ and $W_t$ is a standard Brownian motion.
In Stratonovich form, this becomes
\begin{equation}\label{eq:igbm_strat}
\diff X_t = a_{\text{strat}}\,(b_{\text{strat}} - X_t) \diff t + \sigma X_t \circ \diff W_t,
\end{equation}
where
\begin{equation*}
  a_{\text{strat}} \coloneqq a + \tfrac{1}{2}\sigma^2,\qquad
  b_{\text{strat}} \coloneqq \frac{2ab}{2a + \sigma^2}.
\end{equation*}

\subsubsection*{Default parameters}
Unless otherwise specified, our results for the IGBM process use the parameters:
\[X_{0}=0.06,\ a=0.1,\ b=0.04,\ \sigma=0.6,\ T = 1.0.\]

\subsubsection*{Reference values}
All errors quoted in our results for the IGBM process are with respect to the following theoretically derived reference values:

\begin{enumerate}
    \item Conditional mean \cite{tubikanec2022igbm}: \begin{equation}
        \E_{x_0}[X_t] = b + (x_0 - b)\exp(-a t).
    \end{equation}
    \item Conditional variance \cite{tubikanec2022igbm}: \begin{equation}
    \V_{x_0}[X_t] = \begin{cases}
        \begin{aligned}
        &\exp(-at)\bigl( 2 ab (t x_0 -\tfrac{1}{a} x_0 -t b)+x_0^2 \bigr) \\&\quad - \exp(-2at) (x_0 - b)^2 + b^2, 
        \end{aligned}& \quad \text{if} \ \frac{\sigma^2}{a}=1,\\[15pt]
        \begin{aligned}
        &\exp(-at) \bigl(4b (b - x_0) \bigr) - \exp(-2at) (x_0 - b)^2\\&\quad + 2b^2 a t - 3b^2 + 2bx_0 + x_0^2,
        \end{aligned} & \quad \text{if} \ \frac{\sigma^2}{a}=2, \\[15pt]
        \begin{aligned}
        & \exp\bigl(-(2a - \sigma^2)t\bigr) \bigl(x_0^2 - \tfrac{2ab}{a - \sigma^2} x_0 + \tfrac{2a^2b^2}{(2a - \sigma^2)(a - \sigma^2)} \bigr) + \tfrac{b^2 \sigma^2 }{2a - \sigma^2}\\
        &\quad + \tfrac{2b\sigma^2(x_0 - b)}{a-\sigma^2} \exp(-at) - \exp(-2at) (x_0 - b)^2,
         \end{aligned} & \quad  \text{otherwise}.
    \end{cases}
\end{equation}
\end{enumerate}

\subsubsection*{Numerical considerations}
We do not take any special numerical considerations for simulating the IGBM process.

\subsection{Cox--Ingersoll--Ross Model}\label{app:cir}
The Cox--Ingersoll--Ross (CIR) model is another one-factor short-rate model.
It has similar properties to the IGBM process (e.g.~non-negativity), but is significantly more popular in financial applications (i.e.~interest rates \cite{cox1985cir} and stochastic volatilities \cite{heston1993}).
However, the key advantage of the CIR model is that it leads to analytical formulae for both bond and option prices (when used as the volatility in the Heston model).

Due to its popularity, a wide range of numerical methods have been developed for simulating the CIR model \cite{alfonsi2005cir, alfonsi2010cir, alfonsi2013cir, cozma2020cir, dereich2011cir, foster2024highSIAM, hefter2019cir, kelly2022cir, lord2010truncation, milstein2015cir, ninomiyavictoir2008}.
Recently, one of these numerical methods \cite{dereich2011cir} served as the inspiration for a Langevin SDE simulation algorithm for sampling from non-smooth target distributions (which often arise in Bayesian inverse problems) \cite{cheltsov2025sampling}.

\subsubsection*{Definition}
Let $a > 0$, $b > 0$, and $\sigma > 0$. The Cox--Ingersoll--Ross (CIR) model with mean reversion rate $a$, long-term mean $b$, and volatility $\sigma$ is
\begin{equation}\label{eq:cir_ito}
    \diff X_t = a\,(b - X_t) \diff t + \sigma \sqrt{X_t} \diff W_t,
\end{equation}
where $t >0$ and $W_t$ is a standard Brownian motion \cite{cox1985cir}.
In Stratonovich form, this is
\begin{align}\label{eq:cir_strat}
\diff X_t = a_\text{strat}\,(b_\text{strat} - X_t) \diff t + \sigma\sqrt{X_{t}} \circ \diff W_t,
\end{align}
where
\begin{equation*}
    a_{\text{strat}} \coloneqq a,\qquad
    b_{\text{strat}} \coloneqq b - \frac{\sigma^2}{4a}.
\end{equation*}

\subsubsection*{Default parameters}
Unless otherwise specified, our results for the CIR model use the parameters:
\[X_{0}=2.0,\ a=1.0,\ b=3.0,\ \sigma=0.5, \ T = 1.0.\]

\subsubsection*{Reference values}
All errors quoted in our results for the CIR model are with respect to the following theoretically derived reference values \cite{cox1985cir, feller1951diffusion}:
\begin{enumerate}
    \item Conditional mean \cite{cox1985cir}:
    \begin{equation}
        \E_{x_0}[X_t] = x_0 \exp(-a t) + b \bigl(1 - \exp(-at)\bigr).
    \end{equation}
    \item Conditional variance \cite{cox1985cir}:
    \begin{equation}
        \V_{x_0}[X_t] = \frac{\sigma^2 x_0}{a} \bigl(\exp(-at) - \exp(-2a t)\bigr)  + \frac{\sigma^2 b}{2a} \bigl(1 - \exp(-at)\bigr)^2.
    \end{equation}
    \item Conditional distribution \cite{cox1985cir}:
    \begin{equation}
            (X_t \mid X_0) \sim  \dfrac{\sigma^2}{4 a} (1 - \exp(-a t))  \cdot \chi^2(\delta, \lambda),
    \end{equation}
    where $\chi^2(\delta, \lambda)$ denotes the non-central chi-squared distribution with $\delta$ degrees of freedom and non-centrality parameter $\lambda$,
    \begin{align*}
        \delta = \frac{4 a b}{\sigma^{2}},\qquad \lambda = \frac{4 a x_0}{\sigma^{2}(\exp(a t) - 1)}.
    \end{align*}
    \item Price at time $s$ of a zero-coupon bond maturing at time $T$ \cite{cox1985cir}: \begin{equation}
        \E\Biggl[\exp\biggl(-\int_s^T X_t \diff t\biggr)\biggm| X_s\Biggr] = A(s, T) \exp\bigl(-X_s B(s, T)\bigr),
    \end{equation}
    where
    \begin{align*}
            A(s, T) &= \left( \frac{2 \gamma \exp\Bigl(\frac{(\gamma + a)(T - s)}{2}\Bigr)}{(\gamma - a) + (\gamma + a) \exp\bigl(\gamma (T - s)\bigr)} \right)^{\tfrac{2ab}{\sigma^2}},\\[3pt]
            B(s, T) &= \frac{2\bigl(\exp\bigl(\gamma (T-s)\bigr) - 1\bigr)}{(\gamma - a) + (\gamma + a) \exp\bigl(\gamma (T - s)\bigr)},\qquad
            \gamma = \sqrt{a^2 + 2\sigma^2}.
    \end{align*}
\end{enumerate}

\subsubsection*{Numerical considerations}
When simulating the CIR model using standard methods, it is possible to generate paths whose values are non-positive even if $\sigma^2 \leq 2ab$ (the Feller condition \cite{feller1951diffusion} is met) and $X_0 > 0$.
These negative values prevent the computation of $\sqrt{X_t}$, causing the simulation to fail.
To fix this problem, we instead compute $\sqrt{X_t^+} \coloneqq \sqrt{\max{(X_t, 0)}}$.
This adjustment when simulating the CIR model is similar to the ``full-truncation'' scheme of \cite{lord2010truncation}.

\subsection{Wright-Fisher diffusion with drift}\label{app:wf}
The Wright-Fisher diffusion is a continuous-time approximation to the discrete Wright-Fisher model used in population genetics, where the state variable models the relative frequency of an allele in a given population \cite{ewens2004genetics, etheridge2011genetics, meser2016wrightfisher}.

\subsubsection*{Definition}
Let $s\in \R$ and $\gamma > 0$. Then the Wright-Fisher diffusion with selection coefficient $s$ and variance $\gamma$ is given by
\begin{align*}
\diff X_t = sX_t(1-X_t) \diff t + \sqrt{\gamma X_t(1-X_t)} \diff W_t,
\end{align*}
where $t > 0$ and $W_t$ is a standard Brownian motion.
In Stratonovich form, this becomes
\begin{align*}
\diff X_t = \Bigl(sX_t(1-X_t) - \gamma\Bigl(\frac{1}{4} - \frac{1}{2}X_t\Bigr)\Bigr) \diff t + \sqrt{\gamma X_t(1-X_t)} \circ \diff W_t.
\end{align*}
\subsubsection*{Default parameters}
Unless otherwise specified, our results for the Wright-Fisher diffusion use the parameters:
\[X_{0}=0.5,\ s=0,\ \gamma=1,\ T = 1.0.\]

\subsubsection*{Reference values}
All errors quoted in our results for the Wright-Fisher diffusion are with respect to the following theoretically derived reference values:

\begin{enumerate}
    \item Conditional mean (when $s = 0$) \cite{ewens2004genetics}: \begin{equation}
        \E_{x_0}[X_t] = x_0.
    \end{equation}
    \item Conditional variance (when $s = 0$) \cite{ewens2004genetics}: \begin{equation}
        \V_{x_0}[X_t] = x_0(1-x_0) \bigl(1-\exp(-\gamma t)\bigr).
    \end{equation}
    \item Conditional boundary absorption probability \cite{ewens2004genetics}:
    \begin{align}
        \mathbb{P}_{x_0}(X_\tau = 1) = \begin{cases}x_0, &\quad \text{if } s = 0,\\
        \frac{1 - \exp\bigl(-\frac{2s}{\gamma} x_0\bigr)}{1 - \exp\bigl(-\frac{2s}{\gamma}\bigr)}, &\quad \text{if } s\neq 0,
        \end{cases}
    \end{align}
    where $\tau = \inf \{t > 0\colon X_t \in \{0,1\}\}$ so that $\mathbb{P}(\tau < \infty) = 1$ for any $s\in\mathbb{R}$ and $\gamma > 0$.
\end{enumerate}

\subsubsection*{Numerical considerations}
To satisfy the absorbing boundary condition, we compute the state variable as $X^+_-\coloneqq \min(\max(X, 0), 1)$.

\subsection{Log-Heston model}\label{app:heston}
The Heston model is a well-known stochastic volatility model for option pricing \cite{heston1993}. Equivalent to the Heston model is the log-Heston model, which can be obtained simply by applying a logarithm to the stock price component. Let $a>0$, $b\in\R$, $\sigma>0$, $\mu\in\R$, and $\rho\in[-1,1]$. Then the log-Heston model with mean reversion rate $a$, long-term mean $b$, volatility-of-volatility $\sigma$, risk-free rate $\mu$, and instantaneous correlation $\rho$ is
\begin{equation}\label{eq:heston_ito}
\begin{aligned}
    \diff X_t &= \bigl(\mu - \tfrac{1}{2}V_t\bigr) \diff t + \sqrt{V_t}\, \diff W^{(1)}_t,\\
    \diff V_t &= a\,(b - V_t)\diff t + \sigma \sqrt{V_t} \diff W^{(2)}_t,
\end{aligned}
\end{equation}\medbreak
where $t>0$, $X_t \coloneqq \log {(S_t)}$ is the log-asset price, $V_t$ is its stochastic variance, and the Brownian motions have correlation $\rho$.
In Stratonovich form, this becomes
\begin{equation}\label{eq:heston_strat}
\begin{aligned}
    \diff X_t &= \bigl(\mu_{\text{strat}} - \tfrac{1}{2}V_t\bigr) \diff t + \sqrt{V_t}\circ \diff W^{(1)}_t,\\
    \diff V_t &= a\,(b_{\text{strat}} - V_t) \diff t + \sigma \sqrt{V_t}\circ \diff W^{(2)}_t,
\end{aligned}
\end{equation}
where
\begin{align*}
\mu_{\text{strat}} &\coloneqq \mu - \tfrac{1}{4}\rho\sigma,\qquad b_{\text{strat}} \coloneqq b - \tfrac{\sigma^2}{4a}.
\end{align*}
The correction to $b$ is the same as for the CIR model. The correction to $\mu$ arises from the cross-term in the It\^{o}-to-Stratonovich conversion due to the correlated Brownian motions \cite{KP1992}.

\subsubsection*{Default parameters}
Unless otherwise specified, our results for the log-Heston model use the parameters:
\[X_{0}=\log(20),\ V_0=0.4,\ \mu=0.1,\ a=2,\ b=0.1,\ \sigma=0.5,\ \rho=0,\ T = 1.0.\]

\subsubsection*{Reference values}
All errors quoted in our results for the log-Heston model are with respect to the following theoretically derived reference value:

\begin{enumerate}
    \item Price at time $t=0$ of a European call option with expiration time $T$ \cite{heston1993, crisostomo2015heston}:
    \begin{equation}
    \begin{aligned}
        C(t=0,\, T) &= \E\bigl[\exp(-\mu T) \max\bigl(S_T - K, 0\bigr)\bigr]\\[2pt]
                    &= S_{0} \Pi_1 - \exp(-\mu T) K \Pi_2,
    \end{aligned}
    \end{equation}
    where
    \begin{align*}
        \Pi_1 &= \frac{1}{2} + \frac{1}{\pi} \int_0^\infty \operatorname{Re}\left(\frac{\exp(-\iu u \log K) \psi_{X_T}(u - \iu)}{\iu u \psi_{X_T}(-\iu)}\right) \diff u,\\
        \Pi_2 &= \frac{1}{2} + \frac{1}{\pi} \int_0^\infty \operatorname{Re}\left(\frac{\exp(-\iu u \log K) \psi_{X_T}(u)}{\iu u}\right) \diff u,
    \end{align*}
    and $T > 0$, $K$ is the strike price, $\mu$ is the risk-free rate of return, and $\psi$ is the characteristic function \footnote{While it is not a problem for our default parameters, the characteristic function as presented can become unstable for long maturity times and high volatilities due to the little Heston trap \cite{albrecher2007little}.}  defined for $u \in \mathbb{C}$ as:
    \begin{align*}
        \psi_{X_T}(u) &= \exp\bigl(A_T(u) b + B_T(u) V_0 + \iu u (X_0 + \mu T)\bigr),  \\[3pt]
        A_T(u) &= a \biggl[ T\cdot r_{-} - \frac{1}{\gamma} \log\biggl(\frac{1 - g\exp(-hT)}{1 - g}\biggr)\biggr],\\
        B_T(u) &= r_{-}\cdot\frac{1 - \exp(-hT)}{1 - g \exp(-hT)},\\[6pt]
        r_{\pm} &= \frac{\beta \pm h}{\sigma^2}, \qquad h = \sqrt{\beta^2 - 4 \alpha \gamma},\qquad g = \frac{r_{-}}{r_{+}},\\
        \alpha &= \frac{-u^2}{2} - \frac{\iu u}{2},\qquad \beta = a - \rho \sigma \iu u, \qquad\gamma = \frac{\sigma^2}{2}.
    \end{align*}
\end{enumerate}

\subsubsection*{Numerical considerations}

The CIR model is used as the stochastic volatility component of the (log-)Heston model. Therefore, we apply the same numerical considerations as for the CIR model.
That is, we compute $\sqrt{V_t^+} \coloneqq \sqrt{\max{(V_t, 0)}}$ to prevent issues resulting from negativity.

    \clearpage
\section{Rough Path Theory}\label{app:roughpaths}

\subsection{Motivation}\label{sec:rp:motivation}
We give a brief motivation of and introduction to rough path theory. A more comprehensive introduction is given in \cite{terrystflour}.

The goal of rough path theory is to study controlled differential equations (CDEs) of the form 
\begin{equation}\label{eq:CDE}
    \diff Y_t = f(Y_t) \diff X_t\m,\qquad Y_0 = \xi,
\end{equation}
where $X= (X^{(0)},X^{(1)},\ldots,X^{(d)})$ is the \emph{control} or \emph{signal} (living in $V \coloneqq \R^{d+1}$ for some $d\in \N$, where $X^{(0)}(t) = t$), and $Y$ is the \emph{response} (living in $W \coloneqq \R^m$ for some $m\in \N$). The solution, if it exists, is denoted $Y = I_f(X,\xi)$, and $I_f$ is called the \emph{It{\^o} map} associated with $f$.

Recall that if $p \ge 1$ and $X\colon [0,1] \to V$ is continuous, then the \emph{$p$-variation} of $X$ is

\begin{equation}\label{eq:pvariation}
    \|X\|_{p} = \sup_{\mathcal{D}} \bigg[ \sum |X_{t_{i+1}}-X_{t_i}|^p\bigg]^{1 / p},
\end{equation}
where the supremum is over all subdivisions of $[0,1]$ (not to be confused with the $L^p$ norm which we would denote by $\left\|\cdot\right\|_{L^p}$).
If $X$ is of bounded ($1$-)variation and $f$ is sufficiently smooth (say Lipschitz), then classical theory ensures existence and uniqueness of the solution to \eqref{eq:CDE}, as well as continuity of the It\^o map. More precisely, if $X^n$, $n\in \N$, and $X$ are bounded variation paths, and $X^n \to X$ in $1$-variation, then $I_f(X^n,\xi) \to I_f(X,\xi)$ in $1$-variation for any initial condition $\xi$.
However, in reality signals are often much more rough; a famous example is Brownian motion, whose paths almost-surely have finite $p$-variation if and only if $p > 2$, but the classical theory of controlled differential equation based on the Young integral fails as soon as $p \ge 2$. Equations controlled specifically by Brownian motion can still be made sense of in an inherently stochastic way using It\^o calculus, but the aim of rough path theory is to develop a flexible framework in which to study \eqref{eq:CDE} for a broad class of rough signals in a \emph{pathwise} (i.e.\ deterministic) manner.

To summarise, what we would like to do is to extend the map $I_f$, say for a fixed, smooth $f$, to a space of functions that includes rough signals such as continuous paths that have finite $p$-variation only for $p > 2$, in such a way that $I_f$ is continuous w.r.t.\ some complete metric. A na\"ive approach, however, is doomed to fail: there exist many examples of sequences of smooth paths $(x^1_n)$ and $(x^2_n)$ that converge to the same path $x$ uniformly and in $p$-variation say for some $p > 2$, but such that $I_f(x^1_n,\xi)$ and $I_f(x^2_n,\xi)$ have different limits.

The problem, it turns out, is not due to a wrong choice of metric, or an inherent impossibility of (deterministically) integrating against rough signals, but rather of conceptual nature: the mistake lies in thinking of a rough signal as ``just'' a continuous path $[0,1] \to V$; to characterise its behaviour as a driving signal in \eqref{eq:CDE}, more information is required. It turns out that, in many of the examples alluded to at the beginning of the paragraph, $(x^1_n)$ and $(x^2_n)$ do in fact both have well-defined but \emph{distinct} limits in an appropriate space of so-called rough paths that both ``live over'' the same classical path $x\colon [0,1] \to V$.

To motivate what additional information may be required to characterise such a rough path, consider for a moment the case of CDEs driven by paths of bounded variation. If $f$ is sufficiently regular, they can be solved through Picard iteration---i.e.\ letting $Y^0_t = Y_0$ and iteratively solving $\diff Y^{n+1}_t = f(Y^n_t) \diff X_t$---which reveals that $Y_t$ is actually a function of the collection of iterated integrals of the signal $X$ on $[0,t]$, that is its  \emph{signature} \[
    \sig_{0,t}(X) \coloneqq \left( 1,\sig^1_{0,t}(X), \sig^2_{0,t}(X),\ldots \right) \in T((V)) = \bigoplus_{k=0}^\infty V^{\otimes k},
\] where $V^{\otimes k}$ is the $k$-fold tensor product of $V$ with itself, and 
\begin{equation*}
    X^k_{0,t} \coloneqq \sig^k_{0,t}(X) \coloneqq \int_{0 < u_1 < \ldots< u_k < t} \diff X_{u_1} \otimes \ldots\otimes \diff X_{u_k} \in V^{\otimes  k}
\end{equation*}
can be thought of as compact notation for the collection of $k$-fold iterated integrals of $X = (X^{(0)}, X^{(1)}, \ldots ,X^{(d)})\colon [0,1] \to V = \R^{d+1}$, i.e.\ \[
    X^k_{0,t} = \big(X^{k,I}_{0,t}\big)_{I = (i_1 \ldots i_k) \in [d]^k}\quad \text{where} \quad X^{k,I}_{0,t} = \int_{0 < u_1 < \ldots< u_k < t} \diff X_{u_1}^{(i_1)} \ldots \diff X_{u_k}^{(i_k)}.
\]
Note that the first level of the signature are just the increments $X^1_{0,t} = X_t - X_0$, which determine the path up to a translation that doesn't affect its behaviour as a signal.

From this perspective, it seems natural to think of a signal not as a path in $V$, but as a path in $T((V))$. If the underlying path has bounded variation, then this distinction is insubstantial because all levels of the signature are determined by the first level (the increments) through classical integration. But now say that we have a continuous path $X\colon [0,1] \to V$ with finite $p$-variation only for $p \in (2,3)$, such as a realisation of Brownian motion. Then we can define $X^1_{s,t} = X_t - X_s$, but for higher levels we run into the old problem that iterated integrals of $X$ cannot be defined classically. However, it turns out that if we \emph{assign} a second level $X^2$ in such a way that 
\begin{equation}\label{eq:chenlevel2}
    X^2_{s,t} = X^2_{s,u} + X^2_{u,t} + X^1_{s,u} \otimes X^1_{u,t}\m,\qquad \forall s < u < t,
\end{equation}
an identity that holds automatically for the signature of a bounded variation path by standard properties of the integral, and such that $X^2$ has bounded $p$-variation in an appropriate sense, then there is a \emph{unique} way of extending the signature to all higher levels in such a way that the natural generalisation of \eqref{eq:chenlevel2} holds at all levels (and that the entire signature has bounded $p$-variation in an appropriate sense). Such an object is then called a ($p$-)\emph{rough path} over the original path $X\colon [0,1] \to V$, and, crucially, there can be many rough paths over the same classical path $X$. The central result of rough path theory is that there is a natural way to define integration against and CDEs driven by rough paths, in such a way that the It\^o map is continuous with respect to a suitably chosen metric.

\subsection{Rough paths and the universal limit theorem}

Let $\triangle = \left\{ (s,t)\colon 0 \le s \le t \le 1 \right\} $, and $X\colon \triangle \to T((V))$ be a continuous mapping. To avoid confusion, we will use lower case letters such as $x$ to denote classical paths $[0,1]\to V$ in this section. The generalisation of \eqref{eq:chenlevel2} to higher levels is obtained by writing (the components of) $X_{s,t}$ in terms of $X_{s,u}$ and $X_{u,t}$ for some $s < u < t$, in the case where $X$ is the signature of a bounded variation path. This can be done using standard properties of integrals, and the result can be written compactly as
\begin{equation}\label{eq:chen}
    X_{s,t} = X_{s,u} \otimes X_{u,t}\m,\qquad \forall\, 0 \le s \le u \le t \le 1,
\end{equation}
also known as \emph{Chen's identity}.
The zeroth level of \eqref{eq:chen} is $1 = 1$, the first level is $X^1_{s,t} = X^1_{s,u} + X^1_{u,t}$, and the second level is \eqref{eq:chenlevel2}.

\begin{definition}\label{def:proughpath}
    Let $p \ge 1$. A \emph{$p$-rough path} is a continuous mapping $X\colon \triangle \to T((V))$ that satisfies Chen's identity \eqref{eq:chen} and has finite $p$-variation in the sense that 
    \begin{equation}\label{eq:pvarroughpath}
        \sup_{\mathcal{D}} \sum \|X^{k}_{s_i,s_{i+1}}\|^{p / k} < \infty,
    \end{equation} 
    for all $k\in \N$, where the supremum is over finite subdivisions of $[0,1]$.
\end{definition}
Note that \eqref{eq:chen} implies that the first level of a $p$-rough path is always the increments of a continuous path $[0,1] \to V$, which by \eqref{eq:pvarroughpath} has finite $p$-variation. Note further that if $x\colon [0,1]\to V$ is a continuous path with finite $p$-variation for some $p\in [1,2)$, then its signature $\sig(X) = (1,X^1,X^2,\ldots)$ is classically defined in terms of iterated (Young) integrals, and is a $p$-rough path according to Definition~\ref{def:proughpath}. Denote by $T^{(n)}(V) = \bigoplus_{k=0}^n V^{\otimes k}$ the truncated tensor algebra over $V$ of degree $n$, equipped with the product $\boldsymbol a \boldsymbol b = \boldsymbol c$ where $c_k = a_0 \otimes b_k +\ldots + a_k \otimes b_0$.

\begin{theorem}[Extension Theorem]\label{thm:extension}
    Let $p \ge 1 $, and let $X\colon \triangle \to T^{(\left\lfloor p \right\rfloor )}(V)$ be a continuous mapping that satisfies Chen's identity and has finite $p$-variation in the sense that \eqref{eq:pvarroughpath} holds for all $k \le \left\lfloor p \right\rfloor $. Then $X$ has a unique extension to a $p$-rough path $X\colon \triangle \to T((V))$.
\end{theorem}
See Theorem 3.7 in \cite{terrystflour}.
In particular, a $p$-rough path is uniquely determined by its truncation at level $\left\lfloor p \right\rfloor $, and we sometimes identify it with this truncation. Denote by $\Omega_p(V)$ the space of $p$-rough paths, which is a complete metric space when equipped with the \emph{$p$-variation metric} \cite[Sect.\ 3.3.1]{lyonsqian} \[
    d_p(X,Y) = \max_{i=1, \ldots ,\left\lfloor p \right\rfloor } \sup_{\mathcal{D}} \bigg( \sum \| X^{k}_{s_i,s_{i+1}} - Y^k_{s_i,s_{i+1}}\|^{\frac{p}{k}} \bigg)^{\frac{1}{p}}.
\] Note that $\Omega_1(V)$ can be identified with the space of (classical) continuous bounded variation paths $[0,1]\to V$ (modulo constant translations) with the $1$-variation metric.

\begin{definition}\label{def:geometricroughpath}
    The space $G \Omega_p(V)$ of \emph{geometric rough paths} is the closure of $\Omega_1(V)$ in $\Omega_p(V)$ w.r.t.\ the $p$-variation metric. It is in particular complete.
\end{definition}

The inclusion $G \Omega_p(V) \subset \Omega_p(V)$ is strict except if $p < 2$. A simple example of a $2$-rough path that is not geometric is $X_{s,t} = (1,0,(t-s) \boldsymbol a)$ for any $\boldsymbol a \in V^{\otimes 2} \setminus \left\{ 0 \right\} $. It is easy to check that Chen's identity holds and $X$ has finite $p$-variation for any $p \ge 2$, so $X$ is a $2$-rough path. But for any bounded variation path, and therefore any geometric $p$-rough path by approximation in $d_p$, it must hold that the symmetric part of $X^{2}_{s,t}$ is $\frac{1}{2} (X^1_{s,t})^{\otimes  2}$, but this is not true except if $\boldsymbol a = 0$.

\begin{remark}
    An alternative way to introduce (geometric) rough paths directly, without any reference to tensor algebras and Chen's relation, would have been to say that the space of (geometric) $p$-rough paths is the completion of the space of continuous paths $[0,1] \to V$ of bounded variation w.r.t.\ a metric in which two paths are close if they and their first $\left\lfloor p \right\rfloor $ iterated integrals are close in $p$-variation.
\end{remark}

The following theorem is the main result and justification of rough path theory. A function $f\colon A \to B$ between two finite-dimensional Banach spaces $A,B$ is said to be $\textrm{Lip}(\gamma)$ for some $\gamma > 0$ if it is bounded, $\left\lfloor \gamma \right\rfloor $ times continuously differentiable and such that the $\left\lfloor \gamma \right\rfloor $'th derivative is H\"{o}lder continuous with exponent $\gamma - \left\lfloor \gamma \right\rfloor $. Recall that $V = \R^{d+1}$ and $W = \R^m$.

\begin{theorem}[Universal Limit Theorem]
\label{thm:universal limit theorem}
    Let $p \ge 1$ and $\gamma > p$, and $f\colon W \to \mathcal{L}(V,W)$ be a $\textrm{Lip}(\gamma)$ function. Then for all $X \in G \Omega_p(V)$ and $\xi \in W$, the equation \[
        \diff Y_t = f(Y_t) \diff X_t\m,\quad Y_0 = \xi,
    \] has a unique solution in $G \Omega_p(W)$ (in an appropriate sense), and the map \[
        I_f\colon G \Omega_p(V) \times W \to G \Omega_p(W)
    \] that sends $(X,\xi)$ to $Y$ is continuous and equal to the unique continuous extension of the classical It\^o map $I_f \colon \Omega_1(V) \times W \to \Omega_1(W)$ (which is continuous in $p$-variation).
\end{theorem}
A more precise statement and a proof are in Section 5.3 of \cite{terrystflour}.

\subsection{Brownian motion as a rough path}
Since the main application of rough path theory is to differential equations driven by Brownian motion, it will be useful to consider for a moment the special case of a geometric $p$-rough path when $p \in [2,3)$. As mentioned before, signatures of bounded variation paths and therefore geometric $p$-rough paths have the property that the symmetric part $S_{s,t} \in V^{\otimes  2}$ defined by \[
    S^{(ij)}_{s,t} = \frac{1}{2} \Big( X^{2,(ij)}_{s,t} + X^{2,(ji)}_{s,t} \Big),\qquad (s,t) \in \triangle,
\] satisfies $S_{s,t} = \frac{1}{2} (X^1_{s,t})^{\otimes  2}$. If we denote the anti-symmetric part by $A_{s,t} \in V^{\otimes  2}$, i.e. \[
    A^{ij}_{s,t} = \frac{1}{2} \Big( X^{2,(ij)}_{s,t} - X^{2,(ji)}_{s,t} \Big),\qquad (s,t) \in \triangle,
\] then $X^2 = S + A$, which means that $X$ as a whole is fully determined by its increments $X^1$ and the anti-symmetric part $A$ of its second level, \[
X = \bigg(1, X^1, \frac{1}{2} (X^1)^{\otimes 2} + A\bigg).
\] The tensor $A$ can be interpreted as the \emph{area} of $X$, in the following sense: if $X$ is the signature of a path $x \in \Omega_1(V)$, then \[
A^{(ij)}_{s,t} = \frac{1}{2} \,\,\smashoperator{\iint_{s \le u_1 \le u_2 \le t}}\,\, \diff x^i_{u_1} \diff x^j_{u_2} - \diff x^j_{u_1} \diff x^i_{u_2},
\] which is the area enclosed between the bounded variation path $(x^i,x^j)$ and its secant on the time interval $[s,t]$. The area is ``signed'', e.g.\ the area enclosed by $t \mapsto (t, \sin(t))$ on $[0,\pi]$, $[\pi, 2\pi]$, and $[0,2\pi]$ is, respectively, $2$, $-2$, and $0$. For genuine $p$-rough paths with $p \ge 2$, it is still common to refer to $A$, the anti-symmetric part of $X^2$, as its \emph{area}.

Brownian motion is a random continuous path which almost-surely has finite $p$-variation for $p > 2$.
By the considerations in the previous section, a (geometric) rough path over a fixed realisation of the Brownian increments is determined by the choice of its area tensor $A$, and this choice is not unique! 
If we choose $A$ to be (a continuous version of) the \emph{L\'evy area} of Brownian motion, that is \[
    A^{ij}_{s,t} = \frac 12 \smashoperator{\iint_{s \le u_1 \le u_2 \le t}} \big( \diff B^i_{u_1} \diff B^j_{u_2} - \diff B^i_{u_2} \diff B^j_{u_1}\big),\qquad i,j\in [d], \, (s,t)\in\triangle,
\] where the integral is in the It\'o or Stratonovich sense (which doesn't make a difference in this case), then we obtain what is called the \emph{canonical Brownian rough path}. It can be obtained as the almost-sure limit of its piecewise linear approximation on an increasingly fine set of support points, so it is indeed geometric. It can further be thought of as the signature of Brownian motion where iterated integrals are calculated in the Stratonovich sense. Indeed, if $X\colon \triangle \to T((V))$ is a canonical Brownian rough path over a Brownian motion $B$, then for every $k\in \N$ and $J =(j_1\ldots j_k) \in [d]^k$, and every $0 \le s \le t \le 1$, almost-surely 
\begin{equation}\label{eq:sigstratonovich}
X^{k,J}_{s,t} = \,\,\smashoperator{\int_{s < u_1 < \ldots < u_k < t}}\,\, \circ \diff B^{(j_1)}_{u_1} \ldots\circ \diff B^{(j_k)}_{u_k}\m,
\end{equation}
where $ \circ \diff B$ denotes integration in the Stratonovich sense. A consequence of this is that the solution of a differential equation driven by Brownian motion when thought of as a rough path has the same law as the solution of the same equation when interpreted in the classical Stratonovich sense; hence why we use the Stratonovich interpretation of \eqref{eq:sde} (recall, however, that Stratonovich and It\^o SDEs can easily be transformed into one another).
The advantage of considering Brownian motion as a rough path is that, once a continuous version of the L\'evy area has been fixed, we can solve all differential equations driven by Brownian motion in a pathwise sense on the same set of probability one, and without requiring predictability or other such inherently probabilistic conditions common in stochastic integration.

    \clearpage
\section{Theoretical Results}\label{app:theoretical results}

\subsection{(Re-)statement of results}
With the additional setup, we can state slightly stronger and more precise versions of Theorems~\ref{thm:maindyadic} and~\ref{thm:mainOA}. We assume that the reader is familiar with some standard notions in probability theory, such as weak convergence of probability measures (i.e.\ convergence in distribution), which we denote $\mu_n \implies \mu$, and tightness.
\begin{manualtheoremnumber}{\ref*{thm:maindyadic}*}
\begin{theorem}\label{thm:maindyadic-stronger}
    Suppose $(\lambda_i^n, \omega_i^n)$ is an \short cubature with degree $D_n \to \infty$, and dyadic depth $m_n \in \N$ such that \[
    m_n \ge \Big(\frac 34 + \varepsilon\Big) \log_2 N_n - C,
    \] for some $\varepsilon, C > 0$.
    Then there exists a $p \in (2,3)$ such that the cubature converges in distribution to Brownian motion as a $p$-rough path. In particular, if $y^n$ denotes the solution to an SDE with smooth vector fields driven by the cubature, and $y$ denotes the solution to the same SDE driven by Brownian motion, then $y^n \implies y$ in distribution as rough (and in particular as ordinary) paths.
\end{theorem}
\end{manualtheoremnumber}

\begin{manualtheoremnumber}{\ref*{thm:mainOA}*}
\begin{theorem}\label{thm:mainOA-stronger}
    Let $(\lambda_i^n,\omega_i^n)$ be the approximate cubature resulting from the first step of the \short algorithm, prior to recombination, sharpening, and any dyadic constructions, with $N_n \to \infty$ and degree $D_n \to \infty$. Then the assertion of Theorem~\ref{thm:maindyadic} holds for \emph{all} $p\in (2,3)$.
\end{theorem}
\end{manualtheoremnumber}

\begin{remark}
    \begin{enumerate}
        \item[(i)] The fact that the convergence of the cubature to Brownian motion holds on the level of \emph{rough} paths is a much stronger assertion than just on the level of ordinary paths. In particular, the main assertion that solutions driven by the cubature paths converge to the real solutions only holds thanks to the universal limit theorem for rough paths (Theorem~\ref{thm:universal limit theorem}), and would not follow from convergence on the level of ordinary paths.\medbreak
        
        \item[(ii)] As mentioned in the main paper, for the assertion of Theorem~\ref{thm:maindyadic-stronger} to hold, it is in fact sufficient if the cubature matches the expected Brownian signature up to degree $D_n \to \infty$ just on $[0,1]$, and only to degree $D=7$ on all of the finer dyadic intervals.
    \end{enumerate}
\end{remark}

\subsection{Proofs}
In the remainder of this appendix, we prove Theorems~\ref{thm:maindyadic-stronger} and~\ref{thm:mainOA-stronger}. 
We start by establishing conditions under which weak convergence and pointwise convergence of expected signatures imply each other for rough paths.

Recall that $V = \R^{d+1}$, and $G\Omega_p(V)$ denotes the space of geometric $p$-rough paths $x=(x^0,x^1,\ldots,x^d)$ in $V$ defined on $[0,T]$ for some fixed and arbitrary $T > 0$ with $x^0(t) = t$. This is a complete and separable metric space \cite[Proposition 39]{chevyrev2022signature}. Recall that the signature is a map $\sig \colon G\Omega_p \to T((V))$.
Let \[
\mathcal{M} \coloneqq \big\{\mu \in \mathcal{M}_1(G\Omega_p)\colon \mu\big( \|\sig^m\|_{V^{\otimes m}}\big) < \infty, \forall m\in \N \big\},
\]
where we used the shorthand $\mu(f)$ for $\int f \diff \mu$ if $f$ is a function and $\mu$ a probability measure such that the integral is well-defined, and $\mathcal{M}_1(X)$ for a measurable space $X$ denotes the set of probability measures on $X$. In particular, if $\mu \in \mathcal{M}$, then $\mu(\sig) \in T((V))$ (exists and) is the expected signature of the random rough path defined by $\mu$.

\begin{lemma}\label{lem:expsig characteristic}
    Let $\mu \in \mathcal{M}$ such that its expected signature has an infinite radius of convergence, that is 
    \begin{equation}\label{eq:infinite radius of conv}
    \forall \lambda > 0\colon \sum_{m=1}^\infty \lambda^m \mu\left( \|\sig^m\|\right) < \infty.
    \end{equation}
    Then the expected signature of $\mu$ characterises it, that is 
    \begin{equation}\label{eq:exp signature characteristic}
    \forall \nu\in \mathcal{M}\colon \mu(\sig) = \nu(\sig) \implies \mu = \nu.
    \end{equation}
    Furthermore, \eqref{eq:infinite radius of conv} holds if $\mu$ is the law of canonical Brownian motion.
\end{lemma}
\begin{proof}
    Suppose $\mu$ satisfies the assumptions of the lemma above, and $\nu \in \mathcal{M}$ with $\mu(\sig) = \nu(\sig)$. By~\cite[Prop.\ 6.1]{chevyrev2016}, this implies that the pushforward laws $\sig\#\mu$ and $\sig\#\nu$ coincide. But by~\cite[Theorem 1.1]{boedihardjo2016}, a rough path $\boldsymbol{x} \in G\Omega_p$ is determined uniquely by its signature $\sig(\boldsymbol{x})$ (we use here that $x^0(t) = t$), and therefore $\mu = \nu$.

    Finally, \eqref{eq:infinite radius of conv} follows from standard exponential tail bounds on Brownian motion and e.g.~\cite[Theorem 6.3]{chevyrev2016}.
\end{proof}

We say that a sequence $(\mu_n)$ in $\mathcal{M}$ is \emph{uniformly integrable} if
\begin{equation}\label{eq:mu uniformly integrable}
    \forall m\in \N\colon \sup_{n\in \N} \mu_n\left( \|\sig^m\| \ind_{\{\|\sig^m\| \ge K\}}\right) \longrightarrow 0,\quad K \to \infty.
\end{equation}

\begin{lemma}\label{lem:roughpaths}
    Let $(\mu_n)$ be a uniformly integrable sequence in $\mathcal{M}$, and $\mu \in \mathcal{M}_1(G\Omega_p)$.
    Then the following hold.
    \begin{enumerate}
        \item[(i)] If $\mu_n \implies \mu$ weakly, then $\mu \in \mathcal{M}$ and $\mu_n(\sig) \to \mu(\sig)$ pointwise.\medbreak
        \item[(ii)] If $(\mu_n)$ is tight, and $\mu\in \mathcal{M}$ satisfies \eqref{eq:exp signature characteristic}, then $\mu_n(\sig) \to \mu(\sig)$ pointwise implies $\mu_n \implies \mu$ weakly.
    \end{enumerate}
\end{lemma}
\begin{proof}
    For $K > 0$ and $x\in \R$, write $\clamp_K(x) \coloneqq \min(K,\max(-K,x))$. To prove (i), assume that $\mu_n \implies \mu$. We first show that $\mu \in \mathcal{M}$. Let $J$ be a multi-index, then
    \begin{align*}
        \mu(|\sig^J|) 
        \le \liminf_{n\to \infty} \mu_n\big(\m|\sig^J|\m\big)
        \le \liminf_{n\to \infty} \mu_n\big(\clamp_K\big(\m|\sig^J|\m\big)\big) + \sup_{n\in\N} \mu_n\big(\m|\sig^J| \ind_{\{|\sig^J| \ge K\}}\big).
    \end{align*}
    Since $\clamp_K(|\sig^J|)$ is continuous and bounded, the first term on the right-hand side converges to $\mu(\clamp_K(|\sig^J|)) < \infty$, and by uniform integrability the second term is finite for some $K > 0$ (in fact goes to zero as $K \to \infty$). We showed that $\mu \in \mathcal{M}$, in particular $\mu(S)$ exists. Next, again for any multi-index $J$ and any $K>0$,
    \begin{align*}
        \big| \mu_n(\sig^J) - \mu(\sig^J)\big|
        & \le \big| \mu_n(\clamp_K(\sig^J)) - \mu(\clamp_K(\sig^J))\big| \\
        &\mmm + \mu_n\big(\m|\sig^J| \ind_{\{|\sig^J| \ge K\}}\big) + \mu\big(\m|\sig^J| \ind_{\{|\sig^J| \ge K\}}\big).
    \end{align*}
    Since $\mu_n \implies \mu$, the first term on the right-hand side goes to zero as $n\to \infty$ for any fixed $K > 0$, so
    \begin{align*}
        \limsup_{n\to \infty}\big| \mu_n(\sig^J) - \mu(\sig^J)\big|
        \le \limsup_{n\to \infty} \mu_n\big(\m|\sig^J| \ind_{\{|\sig^J| \ge K\}}\big) + \mu(\m\big|\sig^J| \ind_{\{|\sig^J| \ge K\}}\big).
    \end{align*}
    By \eqref{eq:mu uniformly integrable}, the right-hand side goes to zero as $K \to \infty$, so the left-hand side is zero. Since $J$ was arbitrary, this implies $\mu_n(\sig) \to \mu(\sig)$ pointwise.

    Now assume that $\mu\in \mathcal{M}$ satisfies \eqref{eq:exp signature characteristic}, $\mu_n(\sig) \to \mu(\sig)$ pointwise, and that $(\mu_n)$ is tight. It suffices to prove uniqueness of subsequential limits, so take some convergent subsequence $\mu_{k(n)} \implies \nu$. By (i) which we've already proved, $\nu\in \mathcal{M}$ and $\mu_{k(n)}(\sig) \to \nu(\sig)$, but also $\mu_{k(n)}(\sig) \to \mu(\sig)$, so $\mu(\sig) = \nu(\sig)$, which implies $\nu = \mu$ by \eqref{eq:exp signature characteristic}.
\end{proof}

The next lemma shows that if we want to apply Lemma~\ref{lem:roughpaths}~(ii) in the case where $\mu$ is the distribution of (canonical rough) Brownian motion, then we get uniform integrability of $(\mu_n)$ for free.
\begin{lemma}\label{lem: uniform integrability free}
    If $\mu_n,\mu \in \mathcal{M},\, n\in \N$ such that $\mu_n(\sig) \to \mu(\sig)$, and $\mu(\|\sig^m\|^2) < \infty$ for all $m\in \N$, then $\sup_{n\in \N} \mu_n(\|\sig^m\|^2) < \infty$, and in particular $(\mu_n)$ is uniformly integrable.
\end{lemma}
\begin{proof}
    For any fixed $m\in \N$ and multi-index $J$ of length $m$,
    the shuffle product identity~\cite[Theorem 2.15]{terrystflour} implies that there exists a set $\mathcal{J}$ of multi-indices of length $2m$ (the ``shuffles'' of $J$ with itself) such that \[
    |\sig^J|^2 = (\sig^J)^2 = \sum_{I \in \mathcal{J}} \sig^I.
    \] Therefore $\mu_n (\sig) \to \mu(\sig)$ implies $\mu_n(|\sig^J|^2) \to \mu(|\sig^J|^2) < \infty$, in particular $\sup_{n\in \N} \mu_n(|\sig^J|^2) < \infty$.
\end{proof}
\
To be able to apply Lemma~\ref{lem:roughpaths}, we need to understand how to prove tightness for sequences of probability measures on $G\Omega_p$, which is the content of the next lemma.

For $\alpha \in (0,1]$ and a function $f\colon [0,T] \to X$ for some normed space $X$, denote by \[
[f]_\alpha \coloneqq \sup_{0 \le s < t \le T} \frac{\|f(t) - f(s)\|}{|t-s|^\alpha},\qquad \|f\|_\alpha \coloneqq [f]_\alpha + \sup_{t\in [0,T]} \| f(t)\|,
\]
the $\alpha$-H\"older seminorm and norm of $f$, respectively. We also write $[f]_\text{Lip}$ for $[f]_\alpha$ if $\alpha = 1$, and similarly for $\|f\|_\text{Lip}$. Sometimes it will be convenient to work with a version of the $\alpha$-H\"older (semi-)norm that only evaluates $f$ on subsequent dyadic points: \[
    [f]_{\alpha,\mathcal{D}} \coloneqq \sup \bigg\{ \frac{|f(t)-f(s)|}{|t-s|^\alpha}\colon 0 \le s < t \le 1, \exists k,n\in \N_0\colon s = k2^{-n}, t=s+2^{-n} \bigg\},
\]
and similarly for $\|f\|_{\alpha,\mathcal{D}}$.

\begin{lemma}\label{lem:tightness}
    Let $(\mu_n)$ be a sequence in $\mathcal{M}$, and suppose that, for some $\alpha > 0$,
    \begin{equation}\label{eq: tightness condition 1}
    \limsup_{n\to \infty} \mu_n(x\colon \|x\|_{\alpha,\mathcal{D}} \ge K) \longrightarrow 0,\quad K \to \infty,
    \end{equation}
    and 
    \begin{equation}\label{eq: tightness condition 2}
    \limsup_{n\to \infty} \mu_n\Bigg(x\colon \sup_{0 < s < t < T} \frac{\| \sig^2_{s,t}(x) \|}{|t-s|^{2\alpha}} \ge K\Bigg) \longrightarrow 0,\quad K \to \infty.
    \end{equation}
    Then $(\mu_n)$ is tight in $G\Omega_p$ for every $p > 1 / \alpha$.
\end{lemma}
\begin{proof}
    First of all, by Lemma~\ref{lem: holder dyadic} below, \eqref{eq: tightness condition 1} implies \[
        \limsup_{n\to \infty} \mu_n(x\colon \|x\|_{\alpha} \ge K) \longrightarrow 0,\quad K \to \infty.
    \]
    
    Now let \[
    A_K \coloneqq \left\{ x\colon \|x\|_\alpha \le K\right\} \cap \left\{x\colon \sup_{0 < s < t < T} \frac{\| \sig^2_{s,t}(x) \|}{|t-s|^{2\alpha}} \le K\right\}.
    \] By assumption, $\limsup_{n\to \infty}\mu_n(A_K^{c}) \to 0$ as $K\to \infty$, so it suffices to show that $A_K$ for any $K > 0$ is relatively compact in $G\Omega_p$.
    
    Let $1 / \alpha < q < \min(3,p)$.
    Then by Proposition~1.15 in~\cite{terrystflour}, if $A_K$ is a bounded subset of $G\Omega_q$ and equicontinuous, then $A_K$ is relatively compact in $G\Omega_p$. Equicontinuity follows from a uniform bound on the $\alpha$-H\"older seminorm, and furthermore, for any $x\in A_K$, and any subdivision $(t_i)$ of $[0,T]$, \[
    \sum_i |x(t_i) - x(t_{i-1})|^{q} \le [x]_\alpha \sum_i |t_i - t_{i-1}|^{\alpha q} \le K\sum_i |t_i - t_{i-1}| = KT.
    \] To prove boundedness of $A_K$ as a subset of $G\Omega_q$, it remains to show that the $q$-variation of the second level of the signature is bounded in the sense of \eqref{eq:pvarroughpath}. Indeed, for any subdivision $(t_i)$ of $[0,T]$, \[
    \sum_i \|\sig ^2_{s,t}(x)\|^{q/2} \le K \sum_i |t_i- t_{i-1}|^{\alpha q} \le K T.
    \]
\end{proof}

\begin{lemma}\label{lem: holder dyadic}
    For $f\colon [0,T] \to X$ for some normed space $X$, and $\alpha \in (0,1]$, \[
    [f]_\alpha \le \frac{2}{1-2^{-\alpha}} [f]_{\alpha,\mathcal{D}}.
    \]
\end{lemma}
\begin{proof}
    Let $M \coloneqq [f]_{\alpha,\mathcal{D}}$, and let $0 \le s < t \le 1$. Suppose that neither $s$ nor $t$ are dyadic, otherwise the proof simplifies. Let $[v_0,w_0] \subsetneq [s,t]$ be the largest dyadic interval of the form $[k2^{-n},(k+1)2^{-n}]$ that is fully contained in $[s,t]$. Then, \[
    \frac{|f(t)-f(s)|}{|t-s|^\alpha} \le \frac{|f(t)-f(w_0)|}{|t-s|^\alpha} + \underbrace{\frac{|f(w_1) - f(w_0)|}{|w_1-w_0|^\alpha}}_{\le M} + \frac{|f(v_0)-f(s)|}{|t-s|^\alpha}.
    \] Now let $n_1 \ge n+1$ be the smallest integer such that $s < v_1\coloneqq v_0 - 2^{-m}$. Then $|t-s| \ge 2 |v_1-v_0|$, so \[
    \frac{|f(v_0)-f(s)|}{|t-s|^\alpha} \le \underbrace{\frac{|f(v_0) - f(v_1)|}{|t-s|^\alpha}}_{\le 2^{-\alpha}M} + \frac{|f(v_1) - f(s)|}{|t-s|^\alpha}.
    \] Continuing in this way, we get \[
    \frac{|f(v_0)-f(s)|}{|t-s|^\alpha} \le M \sum_{i=1}^m 2^{-\alpha i} + \frac{|f(v_m) - f(s)|}{|t-s|^\alpha},
    \] for all $m$. By construction, $s < v_m < s + 2^{-n_m}$ where $n_m \ge n + m \to \infty$, so $f(v_m) \to f(s)$ and therefore letting $m\to \infty$ gives \[
    \frac{|f(v_0)-f(s)|}{|t-s|^\alpha} \le M \sum_{i=1}^\infty 2^{-\alpha i} = M \frac{2^{-\alpha}}{1 - 2^{-\alpha}}.
    \] Proceeding similarly with $\frac{|f(t) - f(w_1)|}{|t-s|^\alpha}$, we obtain \[
    \frac{|f(t)-f(s)|}{|t-s|^\alpha} \le M \bigg( 1 + 2 \frac{2^{-\alpha}}{1-2^{-\alpha}}\bigg) \le \frac{2M}{1-2^{-\alpha}}.
    \] Since $s,t$ were chosen arbitrarily, this finishes the proof.
\end{proof}

The primary remaining ingredient in proving Theorems~\ref{thm:maindyadic-stronger} and~\ref{thm:mainOA-stronger} is to confirm the conditions of Lemma~\ref{lem:tightness}. This is the content of the next two lemmas.

\begin{lemma}\label{lem:tightness dyadic}
    Let $\alpha \in (1/3,1/2)$. Suppose we have probability measures $\mu_n$ on $C^1([0,1],\R)$ that match the Brownian expected signature up to some degree $D > \frac{1}{1/2-\alpha}$ on dyadic intervals up to order $m_n\in \N$ (that is, on intervals of the form $[k2^{-m}, (k+1)2^{-m}]$ for $m \le m_n$), and 
    \begin{equation}\label{eq:lip condition}
    \limsup_{n\to \infty} \mu_n\Big( x\colon \|x\|_\text{Lip} \ge 2^{m_n(1-\alpha)} C\Big) \longrightarrow 0,\quad C \to \infty.
    \end{equation}
     Then, $(\mu_n)$ is tight in $G\Omega_p$ for all $p \in (1/\alpha, 3)$.
     
     If $\mu_n$ is supported on piecewise linear paths with maximum slope $\sqrt{N_n}$ where $N_n = 2^{k_n}$, then \eqref{eq:lip condition} is implied by $k_n \le 2(1-\alpha) m_n + C$ for some $C > 0$.
\end{lemma}
\begin{lemma}\label{lem:tightness OA}
    Let $N_n, D_n\in \N$ such that $N_n , D_n \to \infty$, and suppose that $\mu_n$ for $n\in \N$ is an atomic probability measure on $C^1([0,1],\R)$ consisting of piecewise linear paths of step length $1 / N$ and step size $\pm 1 / \sqrt{N}$ such that every set of $D_n$ increments are independent. Then $(\mu_n)$ is tight in $G\Omega_p$ for all $p\in(2,3)$.
\end{lemma}

Before proving Lemmas~\ref{lem:tightness dyadic} and~\ref{lem:tightness OA}, we show how they imply Theorems~\ref{thm:maindyadic-stronger} and~\ref{thm:mainOA-stronger}.

\begin{proof}[Proof of Theorems~\ref{thm:maindyadic-stronger} and~\ref{thm:mainOA-stronger}]
    Let $(\lambda_i^n,\omega_i^n)$ be the \short cubature specified in Theorem~\ref{thm:maindyadic} (resp.\ Theorem~\ref{thm:mainOA}), and let $\mu_n$ be the associated atomic probability measure on $G\Omega_p$: \[
    \mu_n = \sum_i \lambda_i^n \delta_{\omega_i^n}.
    \] Let $\mu$ denote the probability distribution of canonical Brownian motion on $G\Omega_p$, then by the universal limit theorem for rough paths (Theorem~\ref{thm:universal limit theorem}) it suffices to show that $\mu_n \implies \mu$. To do so, we verify the assumptions of Lemma~\ref{lem:roughpaths} (ii).

    By Lemma~\ref{lem:tightness dyadic} (resp.\ Lemma~\ref{lem:tightness OA}), $(\mu_n)$ is tight. Since $\mu_n$ is atomic it is a member of $\mathcal{M}$, and because $D_n \to \infty$ (and $N_n \to \infty$ for Theorem~\ref{thm:mainOA}) we have $\mu_n(\sig) \to \mu(\sig)$ pointwise, and therefore by Lemma~\ref{lem: uniform integrability free} also uniformly integrability. Finally, by Lemma~\ref{lem:expsig characteristic} we have that $\mu\in \mathcal{M}$ and that it satisfies \eqref{eq:exp signature characteristic}.
\end{proof}
\begin{proof}[Proof of Lemma~\ref{lem:tightness OA}]
    Since $\mu_n$ is atomic it is automatically in $\mathcal{M}$. Thus by Lemma~\ref{lem:tightness} it suffices to prove \eqref{eq: tightness condition 1} and \eqref{eq: tightness condition 2} for arbitrary $\alpha \in (0,1/2)$.
    
    We assume for simplicity that $d = 1$; the proof for $d > 1$ involves more notation but is otherwise the same. We further assume for simplicity that $N_n = 2^{m_n}$ for some $m_n\in \N$, $m_n \to \infty$. Fix $n\in \N$, and abbreviate, if no confusion is possible $m = m_n$, $N = N_n$, etc. Consider a particular dyadic interval $[v,w]$ of length $2^{-k}$. Note that if $k \ge m$, then, for any $x$ in the support of $\mu_n$, 
    \begin{equation}\label{eq: lem uniform holder bound:1}
    \frac{|x(w)-x(v)|}{|w-v|^\alpha} \le \frac{N^{-1/2}}{2^{-k\alpha}} = 2^{k\alpha - m / 2} \le 2^{-(1/2 - \alpha) m} \to 0,\quad n\to \infty.
    \end{equation}
    Hence let $k < m$. Then the increment $x(w) - x(v)$, under the probability measure $\mu_n$, is a sum $S$ of $2^{m - k}$ Rademacher random variables, any $D$ of which are independent, multiplied with $1 / \sqrt{N}$.
    In particular, moments of $S$ up to order $D$ coincide with the moments of the same sum where \emph{all} Rademacher variables are independent,
    which has the same distribution as $\frac{1}{\sqrt{N}}\big(2\text{Bin}\big(2^{m-k},\frac 12\big) - 2^{m-k}\big)$.
    Thus, using Markov's inequality, for any $l\in\N$ with $2l \le D$,
    \begin{align*}
        \mu_n\left(\frac{|x(w)-x(v)|}{|w-v|^\alpha} \ge C\right)
        &= \mu_n\Big(|S| \ge C 2^{-k\alpha}\Big)\\
        &\le C^{-2l} 2^{2\alpha lk} \E \big[\sig^{2l}\m\big]\\
        &= C^{-2l} 2^{2\alpha lk} \big(2/\sqrt{N}\m\big)^{2l}\E\Bigg[ \bigg(\textrm{Bin}\bigg(2^{m-k},\frac 12\bigg) - 2^{m-k-1}\bigg)^{2l}\m\Bigg]\\
        &\le c(l)\m C^{-2l} 2^{2\alpha lk} 2^{-ml} \big(2^{m-k-1}\big)^l\\
        &\le c(l)\m C^{-2l} 2^{-(1-2\alpha) lk}.
    \end{align*}
    Here we used the well-known bound $\E[(\text{Bin}(n,p)-np)^{2l}] \le c(l) (np)^l$.
    Using union bound over all $2^k$ such dyadic intervals, as well as over $k\in\{0,\ldots,m-1\}$, we get, as long as $2^{-(1/2-\alpha) m} \le C$ (recall \eqref{eq: lem uniform holder bound:1}), 
    \begin{align*}
    \mu_n \big(x\colon [x]_{\alpha,\mathcal{D}} \ge C\big) 
    &\le c(l)\sum_{k=0}^{m_n} 2^k C^{-2l} 2^{-(1-2\alpha) lk}\\
    &= c(l)\m C^{-2l} \sum_{k=0}^{m_n} 2^{-k[(1-2\alpha)l - 1]}.
    \end{align*}
    Now if $l > 1 / (1-2\alpha)$, then we get an upper bound of the form \[
        \mu_n \big(x\colon [x]_{\alpha,\mathcal{D}} \ge C\big) \le c(l,\alpha)\m C^{-2l}
    \] for some constant $c = c(l,\alpha) > 0$. Since $\alpha$ is fixed, we can choose $l = l(\alpha) = \left\lceil 1 / (1-2\alpha) \right\rceil$ and then for all $n \ge n_0(\alpha)$ we have $D_n \ge l$ and therefore \[
    \limsup_{n\to \infty} \mu_n(x\colon [x]_{\alpha,\mathcal{D}} \ge C) \le c(\alpha)\m C^{-2l} \longrightarrow 0,\quad n\to \infty.
    \]

    It remains to show \eqref{eq: tightness condition 2}. Note first that $\sig^{(00)}_{s,t}(x) = (t-s)^2/2$ and $\sig^{(11)}_{s,t}(x) = (x(t)-x(s))^2/2$, so it suffices to study $\sig^{(10)}$ and $\sig^{(01)}$. Since further $\sig^{(10)} + \sig^{(01)} = \sig^{(1)} \sig^{(0)}$, it in fact suffices to study only $\sig^{(10)}$. Now, if $(s,t)=(v,w)$ is dyadic as above with $w-v = 2^{-k}$, and the $N' \coloneqq 2^{m-k}$, $D$-wise independent Rademacher variables determining the increments in $[v,w]$ are denoted $\varepsilon_i$, then 
    \begin{align*}
        \sig^{(10)}_{v,w}(x)
        = \sum_{i=1}^{2^{m-k}} \frac{\varepsilon_i}{2N^{3/2}} + \sum_{i < j} \frac{\varepsilon_j}{N^{3/2}}.
    \end{align*}
    The first term is just $\frac{1}{2N} (x(w)-x(v)) \le (w-v)(x(w)-x(v))$, the absolute value of which is at most $C|w-v|^{1+\alpha} \le C |w-v|^{2\alpha}$ with high probability by what we already proved. The second term is equal to $N^{-3/2}\sum_j (j-1) \varepsilon_j$; note that 
    \begin{align*}
    \E\Bigg[ \bigg( \sum_j (j-1) \varepsilon_j \bigg)^{2l}\m\Bigg] 
    &= \sum_{j_1,\ldots,j_{2l}=1}^{N'} \prod_{i=1}^{2l} \underbrace{(j_i-1)}_{\le N'} \underbrace{\E\left[ \varepsilon_{j_1}\ldots \varepsilon_{j_{2l}}\right]}_{\in\{0,1\}}\\
    &\le \sum_{j_1,\ldots,j_{2l}=1}^{N'} N'^{2l} \E\big[ \varepsilon_{j_1}\ldots \varepsilon_{j_{2l}}\big]\\
    &= \E\Bigg[ \bigg( \sum_j N' \varepsilon_j \bigg)^{2l}\m\Bigg] \\
    &= N'^{2l} \E\Bigg[\bigg(2\m\text{Bin}\bigg(N',\frac 12\bigg) - N'\bigg)^{2l}\m\Bigg]\\
    &\le c(l) N'^{3l} .
    \end{align*}
    All together, this gives 
    \begin{align*}
    \mu_n\bigg(\m\bigg|\sum_{i < j} \frac{\varepsilon_j}{N^{3/2}}\bigg| \ge C |v-w|^{2\alpha}\bigg) 
    &\le c(l) C^{-2l} 2^{4lk\alpha} \bigg(\frac{N'}{N}\bigg)^{3l}\\
    &\le c(l) C^{-2l} 2^{lk(4\alpha -3)},
    \end{align*}
    which, with the usual union bound over all $2^{-k}$-based dyadic intervals and $k=0,\ldots,m$, gives an upper bound of 
    \[
    c(l) C^{-2l} \sum_{k=0}^m 2^{k[1+l(4\alpha - 3)]},
    \]
    which again is summable for sufficiently large $l=l(\alpha)$.
\end{proof}

\begin{proof}[Proof of Lemma~\ref{lem:tightness dyadic}]
    Consider a particular dyadic interval $[v,w]$ of length $2^{-k}$ with $k \ge m$, then \[
    \frac{|x(w)-x(v)|}{|w-v|^\alpha} \le \|x\|_\text{Lip} |v-w|^{1-\alpha} \le \|x\|_\text{Lip}\m 2^{-k(1-\alpha)} \le \|x\|_\text{Lip}\m 2^{-m(1-\alpha)}.
    \] Therefore, 
    \begin{equation}\label{eq:prf tightness lemma 1}
    \mu_n\Bigg( x\colon \sup_{|s-t| > 2^{-m_n}} \frac{|x(t)-x(s)|}{|t-s|^\alpha} \ge C\Bigg) \le \mu_n\Big(x\colon \|x\|_\text{Lip} \ge 2^{m_n(1-\alpha)} C\Big).
    \end{equation}
    If $\mu_n$ is supported on paths with Lipschitz constant $\sqrt{N_n} = 2^{k_n / 2}$, then it is sufficient if $k_n \le 2(1-\alpha) m_n + C$ for some $C > 0$. By \eqref{eq:lip condition}, the $\limsup_{n\to \infty}$ of the left-hand side above goes to zero as $C\to \infty$.

    Now if $[v,w]$ is a dyadic interval of length $2^{-k}$ for some $k \le m$, then the expected signature on that interval matches that of Brownian motion up to order $D$; in particular, the moments of the increment $|x(w)-x(v)|$ matches up to order $D$, so 
    \begin{align*}
        \mu_n \left( \frac{|x(w)-x(v)|}{|w-v|^\alpha} \ge C\right)
        &\le C^{-D} 2^{k\alpha D} \E\Big[|x(w)-x(v)|^D\Big]\\
        &\le C^{-D} 2^{k\alpha D} \E\Big[|\m\mathcal{N}(0,2^{-k})|^D\Big]\\
        &= C^{-D} 2^{k\alpha D} 2^{-kD / 2} \E\Big[|\m\mathcal{N}(0,1)|^D\Big]\\
        &\le c(D) C^{-D} 2^{-kD (1/2 - \alpha)}.
    \end{align*}
    Given that $D > 1 / (1/2 - \alpha)$, we can proceed just as in the proof of Lemma~\ref{lem:tightness OA} to obtain \[
    \limsup_{n\to \infty}\mu_n\Bigg( x\colon \sup_{|s-t| \le 2^{-m_n}} \frac{|x(t)-x(s)|}{|t-s|^\alpha} \ge C\Bigg) \longrightarrow 0,\quad C \to \infty,
    \] 
    complementing \eqref{eq:prf tightness lemma 1}.
    
    For the second level, just like in the proof of Lemma~\ref{lem:tightness OA} it suffices to study $\sig^{(10)}$. If $[v,w]$ is a dyadic interval of length $2^{-k}$ for some $k \le m_n$, then
    \begin{align*}
        \mu_n\Bigg(\frac{|\sig^{(10)}_{v,w}|}{|v-w|^{2\alpha}} \ge C\Bigg)
        &\le C^{-D} |v-w|^{-2\alpha D} \underbrace{\E_n\left[ |\sig^{(10)}_{v,w}|^D\right]}_{\le c(D) |v-w|^{3D/2}}\\
        &\le c(D) C^{-D} 2^{-2kD(3/4-\alpha)}.
    \end{align*}
    Again we proceed with a union bound over all dyadic intervals of length $2^{-k}$ and over all $k \le m_n$, which succeeds as long as \[
    2D(3/4 - \alpha) > 1 \iff 2D > \frac{1}{3/4-\alpha} \iff D > \frac{1}{3/4 - \alpha},
    \] which follows from our assumption that $D > 1 / (1/2 - \alpha)$.

    If $v-w < 2^{-m}$, then 
    \begin{align*}
        |\sig^{(10)}_{v,w}| \le \|x\|_\text{Lip} |\sig^{(00)}_{v,w}| = \frac 12 \|x\|_\text{Lip} |v-w|^2,
    \end{align*}
    so
    \begin{align*}
        \limsup_{n\to \infty}\mu_n\Bigg(\frac{|\sig^{(10)}_{v,w}|}{|v-w|^{2\alpha}} \ge C\Bigg)
        &\le \limsup_{n\to \infty} \mu_n\Big( \|x\|_\text{Lip} \ge 2C 2^{2m_n(1-\alpha)}\Big) \longrightarrow 0,\quad C \to \infty.
    \end{align*}
\end{proof}

\subsection{Product Cubatures}
We can think of the construction in Step~1 of the \short algorithm in a slightly different way that lends itself to an obvious generalisation. The set of $2^d$ paths on $[0,1]$ that linearly interpolate between $0$ and a point of the form $(\pm 1,\ldots, \pm 1)$ is itself a degree $3$ cubature on $d$-dimensional Wiener space. What we have done is to scale this cubature onto an interval of length $1/N$ and concatenate $N$ independent copies of it, obtaining a new cubature on $[0,1]$. From this perspective, we can more generally take any existing cubature, the \emph{base cubature}, say of some degree $m$ and size $n$, scale it onto an interval of length $1/N$, and concatenate $N$ independent copies to obtain a new cubature on $[0,1]$ of size $n^N$.

\begin{lemma}\label{lem:product-cubature}
    The product cubature obtained in this way is exact up to order $m$, and $O(N^{-(m-1)/2})$-approximate up to any fixed higher order.
\end{lemma}
\begin{proof}
    Denote by $(S_i)_{i=1}^n$ the set of signatures of the paths in the base cubature, and by $(\lambda_i)_{i=1}^n$ their weights. Fix $N\in \N$ and denote by $\tilde{S}_i$ the Brownian scaling of $S_i$ onto an interval of length $1 / N$, that is \[
    \tilde{S}_i^J = N^{-|J|/2} S_i^J.
    \] Denote by $\tilde{S}^{(1)},\ldots,\tilde{S}^{(N)}$ independent random variables distributed according to the scaled cubature, that is \[
    \P\Big(\tilde{S}^{(1)} = \tilde{S}_{i_1},\ldots , \tilde{S}^{(N)} = \tilde{S}_{i_N}\Big) = \prod_{j=1}^N \lambda_{i_j},\quad i_1,\ldots,i_N \in [n].
    \]
    Then the expected signature of the product cubature is 
    \begin{align*}
    \E\left[ \tilde{S}^{(1)} \otimes \ldots \otimes \tilde{S}^{(N)} \right]
    &= \E\left[ \tilde{S}^{(1)}\right]^{\otimes N},
    \end{align*}
    due to the independence. Now, since $(S_i,\lambda_i)$ is a degree $m$ cubature, $\E[\tilde{S}^{(1)}]$ coincides up to degree $m$ with the expected Brownian signature on an interval of length $1/N$, which is $\exp \left( \frac{1}{2N} \sum_{i=1}^d e_i \otimes e_i\right)$. Due to the Brownian scaling, all terms of degree higher than $m$ are of order $N^{-(m+1)/2}$, which we collectively denote $O_{(> m)}(N^{-(m+1)/2})$ implies 
    \begin{align*}
        \E \left[ \tilde{S}^{(1)}\right]^{\otimes N} 
        &= \left( \exp \left( \frac{1}{2N} \sum_{i=1}^d e_i \otimes e_i\right) + O_{(>m)}(N^{-(m+1)/2}) \right)^{\otimes N}\\
        &= \exp \left( \frac 12 \sum_{i=1}^d e_i \otimes e_i\right) + O_{(>m)} \left( N^{-(m-1)/2}\right).
    \end{align*}
    That is, the expected signature is exact up to order $m$ and has an absolute error of order $N^{-(m-1)/2}$ in higher degree terms, which was exactly our claim.
\end{proof}

The example we gave in the beginning, which is used in Step~1 of \short, is the special case where the base cubature is itself the $d$-fold product of the simplest non-trivial one-dimensional Gaussian cubature, a Rademacher random variable (i.e.\ $\pm 1$ with probability $1/2$ each), which has degree $3$.

    \clearpage
\section{Orthogonal Arrays}
A matrix of the form we use to construct the cubature in Step~1 of the ARCANE algorithm is known as an \emph{orthogonal array}.

\begin{definition}
    Let $q,s,R,C\in \N$ with $q \ge 2$ and $2 \le s \le C$. A matrix $A \in \{0,\ldots,q-1\}^{R \times C}$ is called a \emph{$q$-ary orthogonal array of strength $s$}, an $\OA(R, C, q, s)$ for short, if the $R \times s$ submatrix obtained by selecting any $s$ columns has the property that all $q^s$ possible rows appear the same number of times.
\end{definition}

The probabilistic interpretation of a $q$-ary orthogonal array with strength $s$ and $C$ columns is that of a coupling of $C$ uniform random variables on $\{0,\ldots,q-1\}$ that are $s$-wise independent.

In the context of the Rademacher cubature, we are interested in coupling a set of $C = Nw$ Rademacher random variables in an $s$-wise independent way, where $s$ is the desired degree of the approximate cubature.
This is realised by an $\OA(R,C,2,s)$, and then each of the $R$ rows corresponds to one path in the resulting cubature. Therefore, the parameter regime for orthogonal arrays that we are most interested in is that where $q = 2$ and $s$ (e.g.\ $5$ or $7$) are fixed, $C = Nw$ is variable and large, and $R$ is as small as possible as a function of $C$ (given $q$ and $s$).

\begin{remark}
    If we construct a product cubature based on a more complex base cubature with rational weights, then we can couple some number of copies of them in an $s$-wise independent way using a $q$-ary orthogonal array with strength $s$, where $q$ is the smallest integer such that all of the weights of the base cubature can be written as $p / q$ for an integer $p$.
\end{remark}

A well-known \emph{lower} bound on the number of rows of an orthogonal array with given $s,q,C$ is given by \emph{Rao's bound}.

\begin{theorem}[Rao \cite{rao1947}]
    Any $\OA(R,C,q,s)$ satisfies 
    \begin{equation}\label{eq:rao}
    R \ge \sum_{i=0}^{\lfloor s/2 \rfloor} \binom{C}{i} (q-1)^i = \Theta \left( C^{\lfloor s/2\rfloor}\right).
    \end{equation}
\end{theorem}

The Rao bound can be asymptotically achieved when $q=2$ (the case of most interest to us) as well as if $s=2$, but for most parameter combinations with $s,q \ge 3$ the smallest possible asymptotic of $R = R(C)$ is unknown, and in particular whether or not they achieve the asymptotic growth $\Theta(C^{\lfloor s/2 \rfloor})$ suggested by \eqref{eq:rao}. A known positive example is when $q=3$ and $s=4$, see Theorem~\ref{thm:oaconstructions} (4b)).

In the following theorem, we summarise the best construction methods we are aware of (in the asymptotic regime we are interested in, i.e.\ $s,q$ fixed and $C$ large), together with their asymptotic relationship between $R$ and $C$. Most of these are taken or adapted from existing literature, and we go into detail on both the constructions as well as the correct attributions in the subsequent sections.

Efficient JAX~\cite{jax2018github} implementations of all of the following orthogonal arrays are implemented in the ``orray'' python package~\cite{orray}.

\begin{theorem}\label{thm:oaconstructions}
    The following orthogonal arrays can be constructed.
    \begin{enumerate}
        \item[(1)] \textbf{Strength 2}: For a prime power $q$ and $n\in \N$, \[
        \OA\left(q^n,\frac{q^n-1}{q-1},q,2\right), \quad \text{having} \quad R = 1 + (q-1) C,
        \] which attains \eqref{eq:rao} exactly.
        \item[(2)] \textbf{Strength 3:} For a prime power $q$ and $n\in \N$, \[
                \OA\left(q^{3n+1}, q^{2n}, q, 3\right),\quad \text{having} \quad R = q C^{3/2},
            \] which does not attain \eqref{eq:rao} asymptotically.
        \item[(3)] \textbf{Binary Arrays ($q=2$):}
        All of the following arrays attain \eqref{eq:rao} asymptotically.
        \begin{enumerate}
            \item (Bose-Ray-Chaudhuri~\cite{bose-ray1960}) For $n\in \N$ and even $s\in \N$, \[
            \OA\left(2^{n\lfloor s/2 \rfloor},2^n-1,2,s\right),\quad \text{having} \quad R=(C+1)^{\lfloor s/2\rfloor}.
            \]
            \item (Bose-Ray-Chaudhuri~\cite{bose-ray1960}) For $n\in \N$ and odd $s\in \N$, \[
            \OA\left(2^{1+n\lfloor s/2 \rfloor}, 2^n, 2, s\right), \quad \text{having} \quad R = 2 C^{\lfloor s/2 \rfloor}.
            \]
            \item (Kerdock~\cite{kerdock}) For $s=5$ and an even integer $n\ge 4$,
            \[
            \OA\left(2^{2n}, 2^n, 2, 5\right), \quad \text{having} \quad R = C^{\lfloor s/2 \rfloor}. 
            \]
            \item (Delsarte-Goethals~\cite{delsarte-goethals}) For $s=7$ and an even integer $n \ge 4$, \[
                \OA\left(2^{3n-1}, 2^n, 2, 7\right), \quad \text{having} \quad R = \frac 12 C^{\lfloor s/2\rfloor}. 
            \]
        \end{enumerate}
        \item[(4)] \textbf{Ternary Arrays ($q=3$):}
        \begin{enumerate}
            \item For $s=3$, and $n\in \N$,
            \begin{IEEEeqnarray*}{rrCll}
                \OA(3^{4n+1},& 20^n, 3, 3), & \quad \text{having} \quad & R = 3 C^{4 \log 3 / \log 20} &\approx 3 C^{1.467},\\
                \OA(3^{5n+1},& 45^n, 3, 3), & \quad \text{having} \quad & R = 3 C^{5 \log 3 / \log 45} &\approx 3 C^{1.443},\\
                \OA(3^{6n+1},& 112^n,3,3), &\quad \text{having} \quad & R = 3 C^{6 \log 3 / \log 112} &\approx 3 C^{1.397},
            \end{IEEEeqnarray*}
            all three of which do not attain \eqref{eq:rao} asymptotically.
            \item For $s=4$, and $n\in \N$, \[
            \OA\left(3^{2n+1}, 3^n, 3, 4\right), \quad \text{having} \quad R = 3C^2,
            \] which attain \eqref{eq:rao} asymptotically.
        \end{enumerate}
        \item[(5)] \textbf{General:} For a prime number $q\in \N$, $s\in \N$, and $n\in \N$, \[
            \OA\left(q^{ns}, q^n, q, s\right),\quad \text{having} \quad R = C^s,
        \] which does not attain \eqref{eq:rao} asymptotically.
        \end{enumerate}
\end{theorem}

\subsection{Linear Orthogonal Arrays}

With the exception of (3c) and (3d), all of the orthogonal arrays in Theorem~\ref{thm:oaconstructions} belong to the special class of \emph{linear orthogonal arrays}, which are orthogonal arrays $\OA(R,C,q,s)$ whose rows form a vector space over a finite field with $q$ elements. Such a field exists if and only if $q$ is a prime power---in which case the field is unique up to isomorphism and denoted $\mathbb{F}_q$---and in that case $R = q^n$ where $n$ is the dimension of the vector space spanned by the rows. This is why all of the constructions above require $q$ to be a prime power.

To construct a linear array, say $\OA(q^n,C,q,s)$, it suffices to specify the basis of the vector space spanned by the rows, which is a set of $n$ vectors in $\mathbb{F}_q^C$. If we arrange them as the rows of an $n \times C$ matrix $M$, then it turns out that they give rise to (i.e.\ are the basis of the row space of) a linear orthogonal array of strength $s$ if and only if the \emph{columns} of $M$ are $s$-wise linearly independent.

\begin{lemma}\label{lem:OA from linearly independent vectors}
    Let $q$ be a prime power, and $M$ an $n \times C$ matrix with elements in $\mathbb{F}_q$. Then the $q^n \times C$ matrix whose rows are all possible $\mathbb{F}_q$--linear combinations of the rows of $M$ form an orthogonal array of strength $s$ if and only if any $s$ columns of $M$ are $\mathbb{F}_q$--linearly independent. In that case, we call $M$ a \emph{generator} of the orthogonal array.
\end{lemma}

See e.g.\ Theorem 10.4 in \cite{stinson2008}. They only state the ``if'' direction, but the ``only if'' statement follows from the same proof. Using Lemma~\ref{lem:OA from linearly independent vectors}, the task reduces to finding a set of $s$-wise linearly independent vectors in $\mathbb{F}_q^n$ that is as large as possible. The next section contains constructions based directly on this idea.

\subsubsection{Direct linear constructions}
An instructive example is the case of strength $s=2$: a set of vectors in $\mathbb{F}_q^n$ for some $n\in \N$ is $2$-wise linearly independent if and only if no vector is a multiple of another. The $q^n-1$ non-zero elements of $\mathbb{F}_q^n$ can be partitioned into disjoint sets of $q-1$ vectors that are multiples of each other, and if we pick an arbitrary representative out of each of the $C \coloneqq (q^n-1)/(q-1)$ classes (for example the set of vectors whose first non-zero entry is equal to $1$), then we obtain an $n \times C$ matrix over $\mathbb{F}_q$ whose columns are $2$-wise linearly independent. By Lemma~\ref{lem:OA from linearly independent vectors}, it generates an $\OA(q^n,(q^n-1)/(q-1),2,q)$, and this is exactly construction (1) from Theorem~\ref{thm:oaconstructions}. A reference for this construction is Corollary 10.5 in \cite{stinson2008}.

\begin{construction}[Theorem~\ref{thm:oaconstructions} (1)]
    For a prime power $q$ and $n\in \N$, there exists a linear \[
    \OA \left( q^n, \frac{q^n-1}{q-1},2,q\right).
    \] It is generated in the sense of Lemma~\ref{lem:OA from linearly independent vectors} by an $n \times C$ matrix whose columns are all non-zero elements of $\mathbb{F}_q^n$ whose first non-zero coordinate is equal to $1$.
\end{construction}

Another relatively simple construction is Tneorem~\ref{thm:oaconstructions} (5).

\begin{construction}[Theorem~\ref{thm:oaconstructions} (5)]
    For a prime number $q$, strength $s \in \N$, and $n\in \N$, there exists a linear \[
    \OA \left( q^{ns}, q^n, q, s\right).
    \] It is generated by an $ns \times q^n$ matrix $M$ that can be obtained as follows. First, create an $s \times q^n$ matrix $M'$ over $\mathbb{F}_{q^n}$ whose columns are of the form $(1,x,x^2,\ldots, x^{s-1})$, with $x$ ranging through all $q^n$ elements of $\mathbb{F}_{q^n}$. Then, replace each entry in $M'$ with a column vector of length $n$ over $\mathbb{F}_q$ through an(y) identification of $\mathbb{F}_{q^n}$ with $\mathbb{F}_q^n$ as vector spaces over $\mathbb{F}_q$. The resulting $ns \times q^n$ matrix generates the array.
\end{construction}
\begin{proof}
    It suffices to show that any $s$ columns of $M'$ are $\mathbb{F}_{q^n}$--linearly independent. Indeed, then it follows that they are also $\mathbb{F}_q$--linearly independent, and therefore any $s$ columns of $M$ are $\mathbb{F}_q$--linearly independent.

    Let $x_1,\ldots,x_s \in \mathbb{F}_{q^n}$ be distinct. Then we have to show that the $s \times s$ submatrix of $M'$ corresponding to those $s$ columns is non-singular, and indeed its determinant is given through the Vandermonde formula by \[
    \det \begin{pmatrix}
        1 & \ldots & 1\\
        x_1 & \ldots & x_s\\
        \vdots & \ddots & \vdots \\
        x_1^{s-1} & \ldots & x_s^{s-1}
    \end{pmatrix}
    = \prod_{i < j} (x_j - x_i) \neq 0,
    \] where multiplication etc.\ are performed in $\mathbb{F}_{q^n}$.
\end{proof}

If $n=1$, then this is an $\OA(q^s,q,s,q)$; by definition, \emph{any} $q$-ary orthogonal array of strength $s$ must have at least $q^s$ rows, so the number of rows is minimal, and if $q > s$ then the number of columns is larger than that of the trivial array with the same $N,q$, and $s$ (which has $s$ columns). The largest possible number of columns given $s$, $q>s$, and $N=q^s$ is an open problem outside of a few known special cases; even if $s=2$, the answer is unknown for numerous values of $q$ (see e.g.\ Section 2 in \cite{hedayat2012}, particularly the discussion following Corollary 2.22).

If $n > 1$, then this array stops being optimal in the same sense, and in fact the scaling of its number of rows, $R = C^s$, is worse than that of all the other arrays listed in Theorem~\ref{thm:oaconstructions} in overlapping parameter regimes. But it is the best construction that we are aware of that works for arbitrary strength and $q \ge 3$. For $n=1$, the construction is well-known, see e.g.\ \cite[Corollary 10.7]{stinson2008}. We haven't directly found the generalisation to $n>1$ in the literature, but due to its simplicity assume that it is likewise well-known.

\begin{construction}[Theorem~\ref{thm:oaconstructions} (3a)]
    For $n\in \N$ and even $s=2u\in \N$, there exists a linear \[
    \OA(2^{nu},2^n-1,2,s).
    \] It is generated by an $(nu) \times (2^n-1)$ matrix $M$ over $\mathbb{F}_2$ obtained in the following way. For every non-zero $x\in \mathbb{F}_{2^n}$, create the vector $(x,x^3,\ldots,x^{2u-1}) \in \mathbb{F}_{2^n}^u$. Replace each entry with a binary vector of length $n$ through a (vector space) identification of $\mathbb{F}_{2^n}$ with $\mathbb{F}_2^n$ to yield a vector in $\mathbb{F}_2^{nu}$. The $2^n-1$ vectors obtained in this way form the columns of $M$.
\end{construction}
\begin{proof}
    It suffices to prove that the columns of $M$ constructed in this way are $s$-wise linearly independent, which is in Section 4 of a paper by Bose and Ray-Chaudhuri~\cite{bose-ray1960}.
\end{proof}

This construction is the dual of the well-known BCH code, see e.g.~\cite{BCH}.

Theorem~\ref{thm:oaconstructions} (3b) is obtained from (3a) through the following general construction for binary arrays, see e.g.\ Theorem 2.24 in \cite{hedayat2012}. In words, any binary array of even strength $s=2u$ can be ``upgraded'' to strength $s=2u+1$ at the cost of doubling the number of rows and the additional benefit of adding one column.

\begin{lemma}
    If $A$ is an $\OA(R,C,2,s)$ for even $s\in \N$, then the $(2R) \times (C+1)$ matrix \[
    \left[\begin{array}{c|c}
    \enspace A \enspace & 0\\[5pt]
    \enspace \overline{A} \enspace & 1
    \end{array}\right]
    \] is an $\OA(2R,C+1,2,s+1)$, where $\overline{A}$ denotes the array obtained from $A$ by changing $1$'s for $0$'s and vice versa.
\end{lemma}

\subsubsection{Constructions from cap sets}
Let $q$ be a prime power, and $d\in \N$. A \emph{cap set} in $\mathbb{F}_q^d$ is a set of points such that no three of them are collinear, or, in other words, a $3$-wise affinely independent set of points.

\begin{lemma}\label{lem:OA from affinely independent vectors}
    If the columns of an $n \times C$ matrix $M$ over $\mathbb{F}_q$ are $s$-wise affinely independent, then the $(n+1) \times C$ matrix obtained by adding a row that only consists of $1$'s generates an $\OA(q^{n+1}, C, q, s)$ in the sense of Lemma~\ref{lem:OA from linearly independent vectors}.
\end{lemma}
\begin{proof}
    A set of vectors that is affinely independent becomes linearly independent when an additional coordinate equal to $1$ is added to each vector. Then the statement follows from Lemma~\ref{lem:OA from linearly independent vectors}.
\end{proof}

Furthermore, it is easy to see that products of cap sets are cap sets. In particular, if $S \subset \mathbb{F}_q^d$ is a cap set say of size $k$, then $S^n \subset \mathbb{F}_q^{dn}$ is a cap set of size $k^n$. Combining with Lemma~\ref{lem:OA from affinely independent vectors} gives the following.

\begin{corollary}\label{cor:OA from capset}
    If $S \subset \mathbb{F}_q^d$ is a cap set of size $C$, then there exists an $\OA(q^{nd+1},C^n,q,3)$.
\end{corollary}

Using Corrolary~\ref{cor:OA from capset}, we can turn known cap-set constructions into orthogonal arrays. For instance, Theorem~\ref{thm:oaconstructions} (2) is based on a known construction of a cap set in $\mathbb{F}_q^3$ for prime $q$.

\begin{construction}[Theorem~\ref{thm:oaconstructions} (2)]
    For a prime number $q$ and $n\in \N$, there exists a linear \[
    \OA(q^{3n+1},q^{2n},q,3).
    \] To construct it, let $aX^2 + X + b$ be an irreducible degree two polynomial over $\mathbb{F}_q$, and put \[
    S = \{(x, y, ax^2 + xy + by^2)\colon x,y\in \mathbb{F}_q\}.
    \] Then $S$ is a cap set of size $q^2$ in $\mathbb{F}_q^3$, which generates the orthogonal array in the sense of Corollary~\ref{cor:OA from capset}.
\end{construction}
\begin{proof}
    The cap set above can be derived from the cap set (or \emph{ovoid}) in the projective space $PG(3,q)$ constructed in Example 1.4(1) in \cite{keefe96}.
\end{proof}

The arrays in Theorem~\ref{thm:oaconstructions} (4a) are based on known cap sets in $\mathbb{F}_3^d$ for $d=4,5,6$.

\begin{construction}[Theorem~\ref{thm:oaconstructions} (4a)]
    The three classes of orthogonal arrays in Theorem~\ref{thm:oaconstructions} (4a) are generated in the sense of Lemma~\ref{lem:OA from affinely independent vectors} by the following cap sets:
    \begin{enumerate}
        \item[\textbullet] A cap set of size $20$ in $\mathbb{F}_3^4$,
        \item[\textbullet] A cap set of size $45$ in $\mathbb{F}_3^5$,
        \item[\textbullet] A cap set of size $112$ in $\mathbb{F}_3^6$,
    \end{enumerate}
    which are explicitly given, constructed, or can be derived from the constructions given in \cite[Fig.\ 1]{hill83}, \cite[Section 2]{hill73}, and \cite{potechin2008}, respectively.
\end{construction}
The three cap sets above are all of maximal size \cite{hill83,hill73,potechin2008}, see also sequence A090245 in the OEIS.

Sets of points in $\mathbb{F}_q^d$ that are $s$-wise affinely independent for $s > 3$ are sometimes called \emph{generalised cap sets}, and are much less well-studied. They also lack the convenient product property of (ordinary) cap sets: in fact, if $S \subset \mathbb{F}_q^d$ contains at least two points and is $4$-wise affinely independent, then $S^2 \subset \mathbb{F}_q^{2d}$ is not. Indeed, let $u,v \in S$ be distinct, then $(u,u)$, $(u,v)$, $(v,u)$, and $(v,v)$ are all in $S^2$ and lie on a plane (i.e.\ are affinely dependent).

The only result in this area that we are aware of is a construction by Huang \cite{huang2019} of $4$-wise affinely independent sets in $\mathbb{F}_3^n$.

\begin{construction}[Theorem~\ref{thm:oaconstructions} (4b)]
    For every integer $n\in \N$, there exists a linear \[
    \OA(3^{2n+1}, 3^n, 3, 4).
    \] For $x\in \mathbb{F}_{3^n}$, map the vector $(x,x^2) \in \mathbb{F}_{3^n}^2$ to a ternary vector of length $2n$ using an identification of $\mathbb{F}_{3^n}$ with $\mathbb{F}_3^n$ as vector spaces over $\mathbb{F}_3$. Looping through all $3^n$ possible values for $x$, this gives $3^n$ points in $\mathbb{F}_3^{2n}$. They are $4$-wise affinely independent \cite{huang2019} and generate the array in the sense of Corollary~\ref{cor:OA from capset}.
\end{construction}

\subsubsection{Constructions from non-linear codes}
There is a connection between orthogonal arrays and codes, which we outline briefly.

\begin{definition}\label{def:code}
    Let $N,k\in \N$, and $q\in \N$ be a prime power.
    \begin{enumerate}
        \item[(i)] A \emph{code of size $N\in \N$ and length $k\in \N$ over $\mathbb{F}_q$} is a matrix $C \in \mathbb{F}_q^{N \times k}$, whose rows are called \emph{code words}, and we slightly abuse notation by writing $u\in C$ if $u$ is a row in $C$.
        \item[(ii)] The \emph{distance} $d$ of $C$ is the minimal Hamming distance between any two distinct codewords, that is, \[
    d = \min_{\substack{u,v\in C\\ u\neq v}} d_H(u,v),\quad \text{where} \quad d_H(u,v) = \#\{ i\in [k]\colon u_i \neq v_i\}.
    \]
        \item[(iii)] A code $C$ is called \emph{linear} if its code words are a vector space over $\mathbb{F}_q$. In that case, the \emph{dual code} of $C$ is its orthogonal complement, \[
        C^\perp \coloneqq \{ u\in \mathbb{F}_q^k\colon \forall v\in C\colon \left< u,v\right> = 0\}.
        \] The distance $d^\perp$ of $C^\perp$ is called the \emph{dual distance} of $C$.
    \end{enumerate}
\end{definition}

It is possible to extend the notion of dual distance to non-linear codes, see e.g.\ section~4.4 of~\cite{hedayat2012}. This is important for the statement of the following theorem, but it is not necessary in any way to understand the definition for the purpose of reading the remainder of this section.

\begin{theorem}[Delsarte 1973 \cite{delsarte1973algebraic}]\label{thm:dualdistance}
    A matrix with elements in $\mathbb{F}_q$ for some prime power $q\in \N$ is a code with dual distance at least $t+1$ if and only if it is an orthogonal array with strength at least~$t$.
\end{theorem}

See also Theorem 4.9 in \cite{hedayat2012} for a more recent exposition.

The two remaining constructions, Theorem~\ref{thm:oaconstructions} (3c) and (3d), are obtained from two non-linear codes called the \emph{Kerdock code} \cite{kerdock} and the \emph{Delsarte-Goethals code} \cite{delsarte-goethals}, binary codes with dual distance $6$ and $8$, respectively. 

\begin{construction}[Theorem~\ref{thm:oaconstructions} (3c)]
    For $n\ge 4$ even, there exists a non-linear \[
    \OA(2^{2n}, 2^n, 2, 5).
    \] The rows of the array are the codewords of the Kerdock code, whose construction can be derived from sections 3 and 4 of \cite{hammons94}.
\end{construction}

\begin{construction}[Theorem~\ref{thm:oaconstructions} (3d)]
    For $n\ge 4$ even, there exists a non-linear \[
    \OA(2^{3n-1}, 2^n, 2, 7).
    \] The rows of the array are the codewords of the Delsarte-Goethals code, whose construction can be derived from sections 3, 4, and 6 of \cite{hammons94}.
\end{construction}

    \clearpage
\section{Full-Sized Plots}\label{app:plots}
This appendix includes full-sized versions of the plots in Fig.~\ref{fig:sde-multifigure}. In addition to our cubatures, they include the degree 5 cubature constructed in the original paper by Lyons and Victoir~\cite{lyons2004cubature}, as well as the degree~5 and~7 cubatures constructed in~\cite{nohrouzian}.

\vfill
\begin{center}
    \textit{[Left intentionally blank. Figures continue on the following page.]}
\end{center}
\vfill

\begin{figure}[p]
    \centering
    \includegraphics[width=0.8\textwidth]{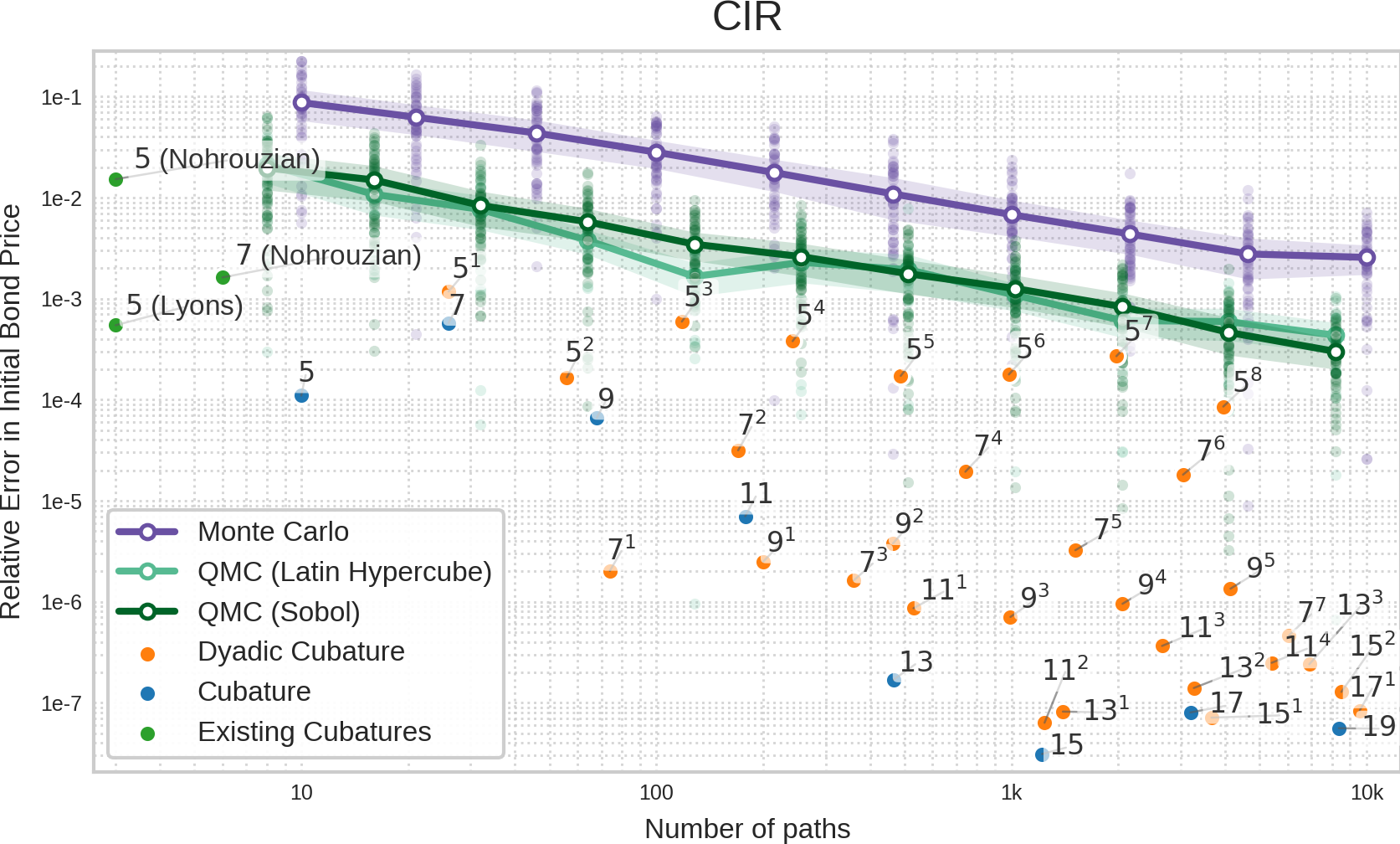}
    \caption{Error plot showing the performance of our cubature formulae in comparison with plain Monte Carlo and QMC methods measured in \emph{Initial Bond Price Relative Error} (see~\eqref{eq:initialbondprice}) in the CIR model. Superscripts in cubature labels refer to the \emph{dyadic depth} of the cubature, see Section~\ref{sec:methods}.}
    \label{fig:cir-bondprice}
\end{figure}

\begin{figure}[p]
    \centering
    \includegraphics[width=0.8\textwidth]{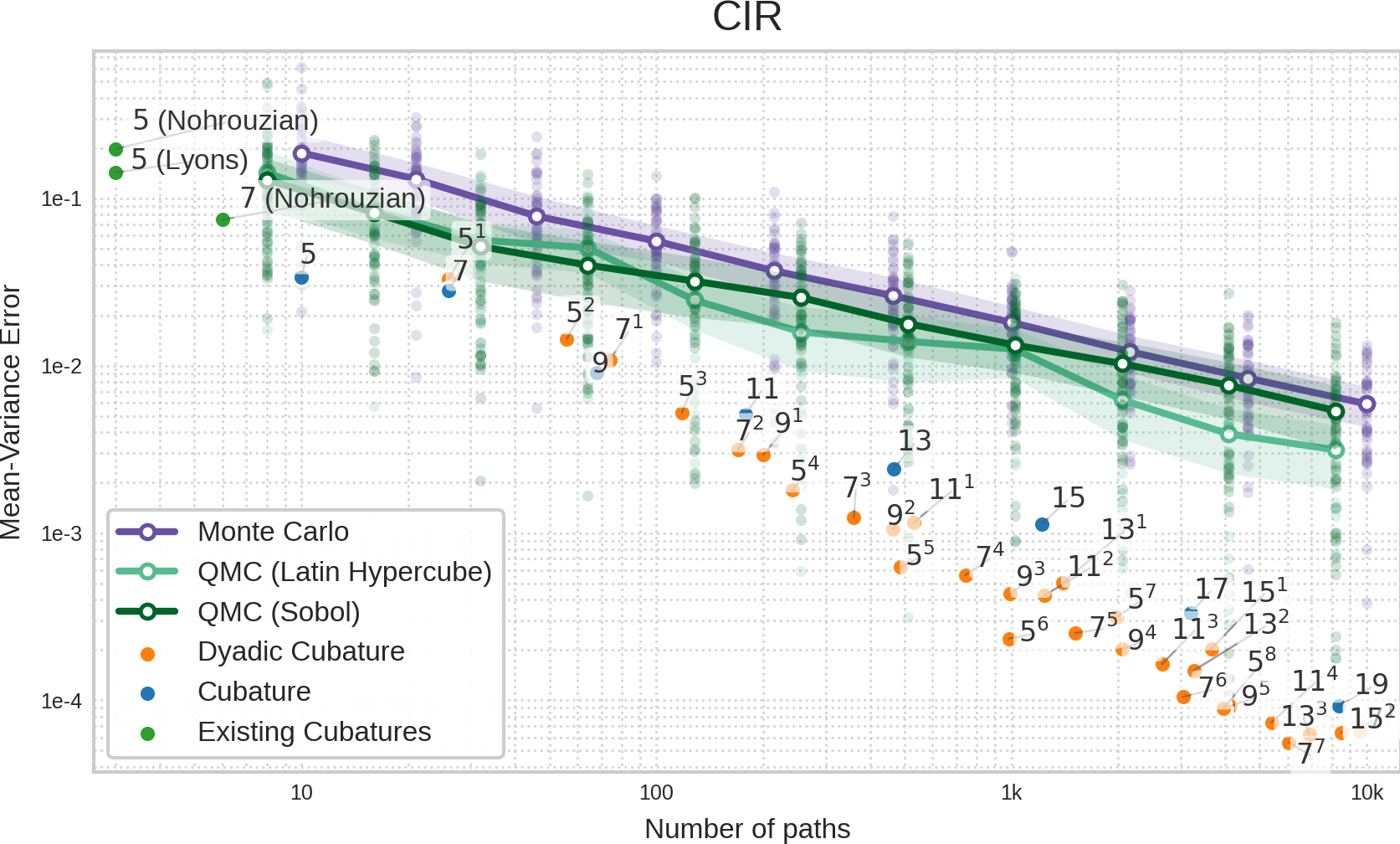}
    \caption{Error plot showing the performance of our cubature formulae in comparison with plain Monte Carlo and QMC methods measured in Mean--Variance Error (see~\eqref{eq:mean-variance-error}) in the CIR model. Superscripts in cubature labels refer to the \emph{dyadic depth} of the cubature, see Section~\ref{sec:methods}.}
    \label{fig:cir-mve}
\end{figure}

\begin{figure}[p]
    \centering
    \includegraphics[width=0.8\textwidth]{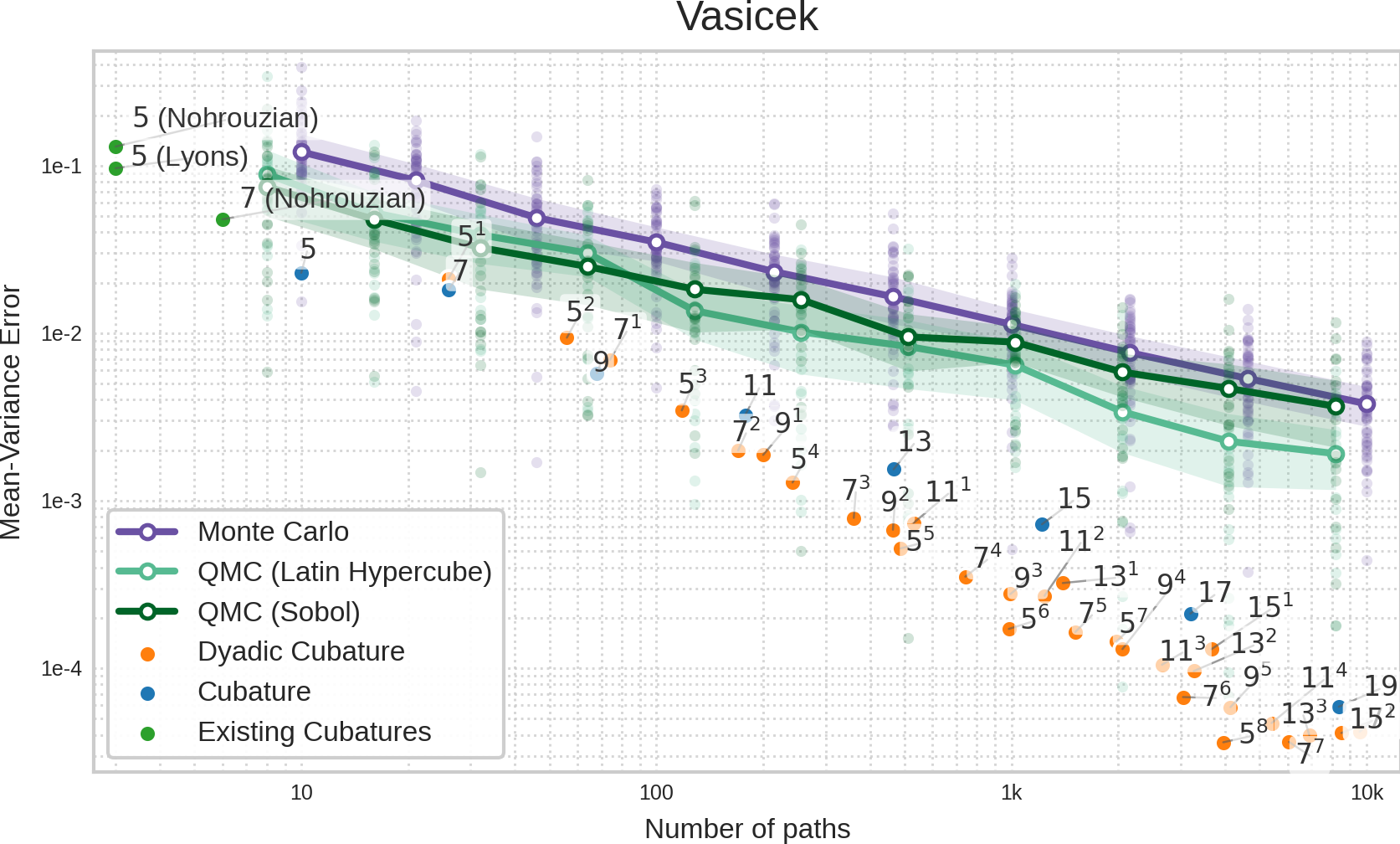}
    \caption{Error plot showing the performance of our cubature formulae in comparison with plain Monte Carlo and QMC methods measured in Mean--Variance Error (see~\eqref{eq:mean-variance-error}) in the Vasicek model. Superscripts in cubature labels refer to the \emph{dyadic depth} of the cubature, see Section~\ref{sec:methods}.}
    \label{fig:vasicek-mve}
\end{figure}

\begin{figure}[p]
    \centering
    \includegraphics[width=0.8\textwidth]{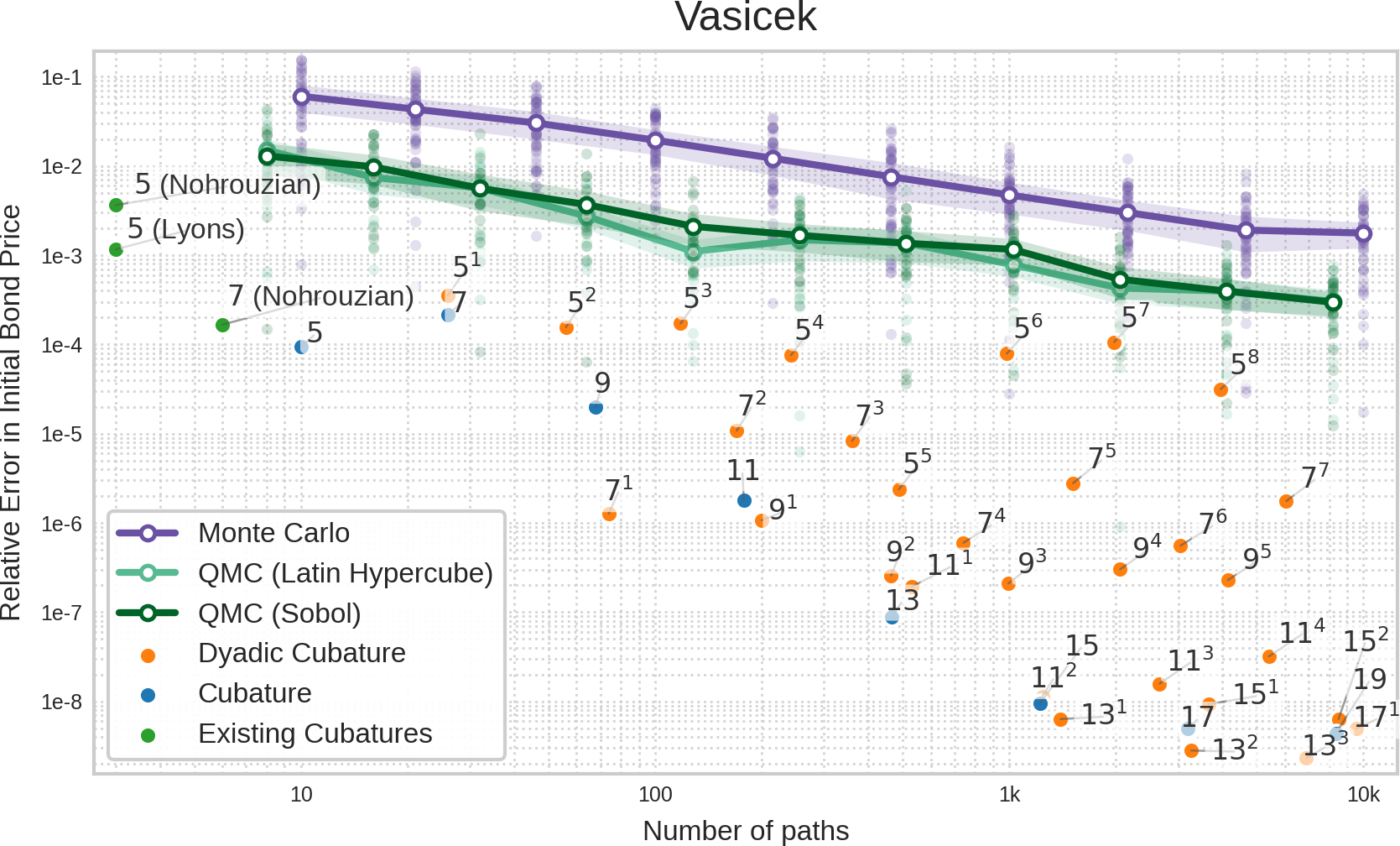}
    \caption{Error plot showing the performance of our cubature formulae in comparison with plain Monte Carlo and QMC methods measured in \emph{Initial Bond Price Relative Error} (see~\eqref{eq:initialbondprice}) in the Vasicek model. Superscripts in cubature labels refer to the \emph{dyadic depth} of the cubature, see Section~\ref{sec:methods}.}
    \label{fig:vasicek-bondprice}
\end{figure}

\begin{figure}[p]
    \centering
    \includegraphics[width=0.8\textwidth]{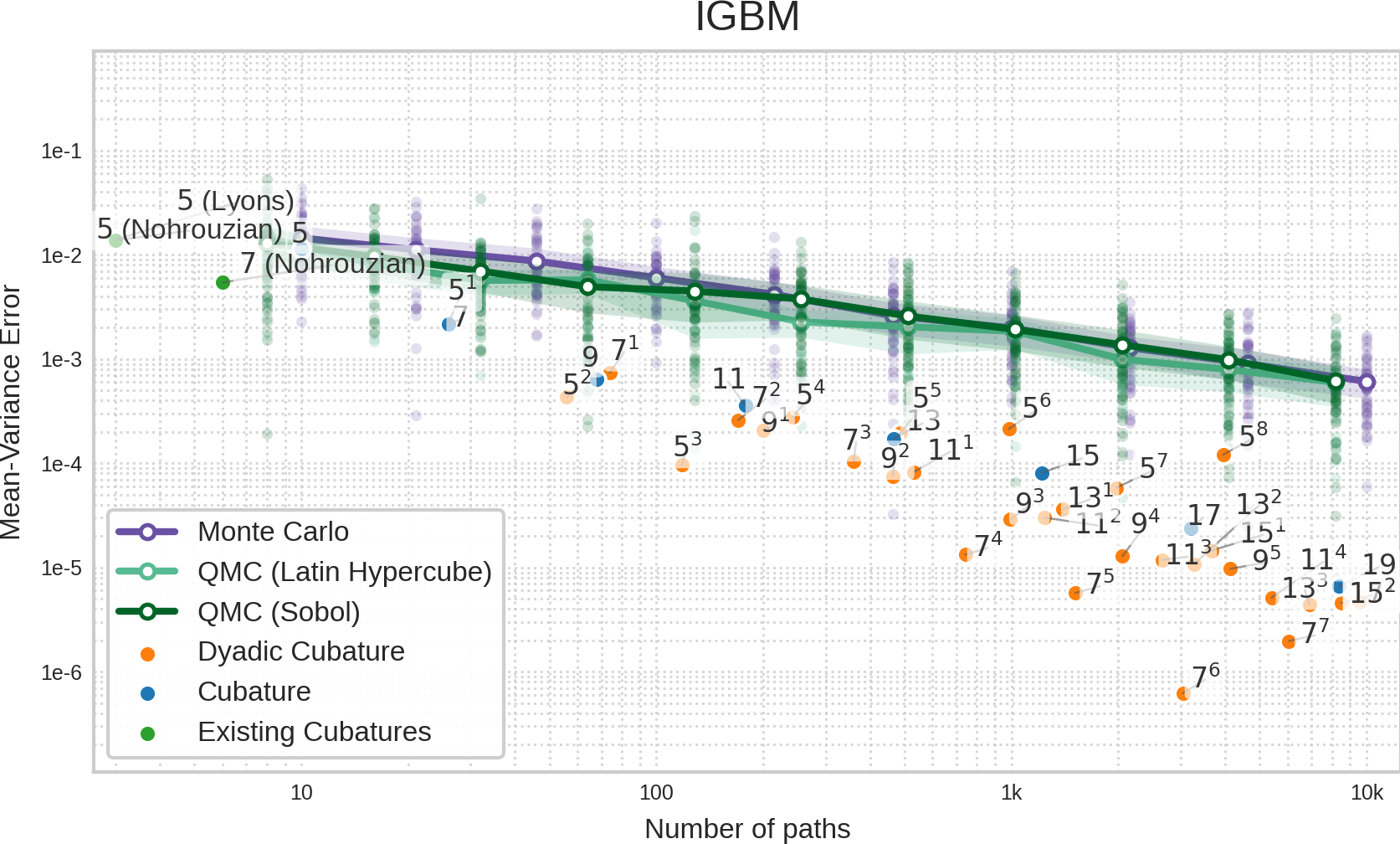}
    \caption{Error plot showing the performance of our cubature formulae in comparison with plain Monte Carlo and QMC methods measured in Mean--Variance Error (see~\eqref{eq:mean-variance-error}) in the IGBM model. Superscripts in cubature labels refer to the \emph{dyadic depth} of the cubature, see Section~\ref{sec:methods}.}
    \label{fig:igbm-mve}
\end{figure}

\begin{figure}[p]
    \centering
    \includegraphics[width=0.8\textwidth]{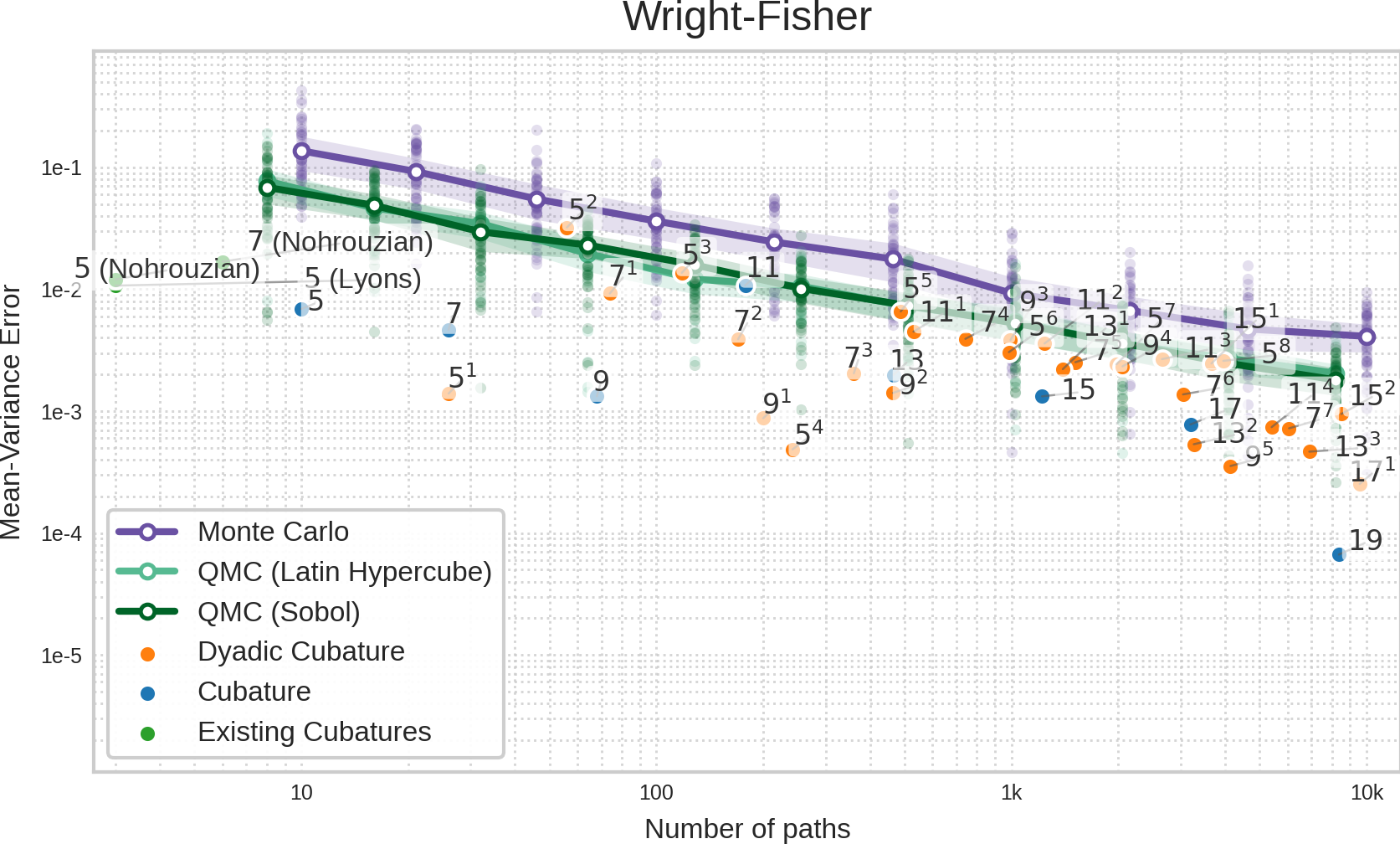}
    \caption{Error plot showing the performance of our cubature formulae in comparison with plain Monte Carlo and QMC methods measured in Mean--Variance Error (see~\eqref{eq:mean-variance-error}) in the Wright--Fisher diffusion model. Superscripts in cubature labels refer to the \emph{dyadic depth} of the cubature, see Section~\ref{sec:methods}.}
    \label{fig:wf-mve}
\end{figure}
\end{appendices}
\end{document}